\documentclass[11pt,reqno]{amsart}
\usepackage{amsthm,amssymb,amsmath,epsfig,graphics,appendix,bm, mathrsfs,bbm,float}

\usepackage[left=1in, right=1in, top=1.1in,bottom=1.1in]{geometry}
\usepackage{hyperref}
\hypersetup{
	colorlinks   = true,
	citecolor    = black,
	linkcolor=black
}

\allowdisplaybreaks

\newtheorem{assumption}{Assumption}
\newtheorem{assumptionalt}{Assumption}[assumption]
\newenvironment{assumptionp}[1]{
	\renewcommand\theassumptionalt{#1}
	\assumptionalt
}{\endassumptionalt}

\newtheorem{lemma}{Lemma}[section]
\newtheorem{theorem}{Theorem}[section]
\newtheorem{corollary}{Corollary}[section]
\newtheorem{definition}{Definition}[section]

\newtheorem{proposition}{Proposition}[section]

\newtheorem{remark}{Remark}[section]

\def\la{\left\langle}
\def\ra{\right\rangle}

\def\R{\mathbb R}
\def\bP{\mathbb P}
\def\bR{\mathbb R}
\def\bN{\mathbb N}
\def\bE{\mathbb E}
\def\E{\mathbb E}

\def\cH{\mathcal H}
\def\cF{\mathcal F}
\def\eps{\varepsilon}

 \numberwithin{equation}{section}

\begin{document}
\title[SPDEs associated with Feller processes]{Stochastic partial differential equations associated with Feller processes}
\author[J. Song]{Jian Song}
\address{Research Center for     Mathematics and Interdisciplinary Sciences, Shandong University, China}
\email{txjsong@sdu.edu.cn}

\author[M. Wang]{Meng Wang}
\address{Center for Applied Mathematics and KL-AAGDM, Tianjin University, Tianjin, China}
\email{wmeng\_2206@tju.edu.cn}

\author[W. Yuan]{Wangjun Yuan}
\address{Department of Mathematics, Southern University of Science and Technology, Shenzhen, China.}
\email{ywangjun@connect.hku.hk}

\begin{abstract}
For the stochastic partial differential equation  $\frac{\partial }{\partial t}u=\mathcal L u +u\dot W$ where $\dot W$ is Gaussian noise colored in time and  $\mathcal L$ is the infinitesimal generator of a Feller process $X$, we obtain Feynman--Kac type of representations for the Stratonovich and Skorohod solutions as well as for their moments. The H\"older continuity of the solutions 
 and the regularity of the law are also studied. 
 
 \smallskip

\end{abstract}

\maketitle
\tableofcontents

\section{Introduction}

Consider the following SPDE in $\mathbb R^d$,
\begin{equation}\label{spde}
	\begin{cases}
		\displaystyle\frac{\partial }{\partial t}u(t,x)=\mathcal {L} u(t,x)+u(t,x)\dot {W}(t,x),& t> 0, x\in \mathbb R^d\\
		u(0,x)=u_0(x),& x\in\mathbb R^d,	\end{cases}
\end{equation}
where  $u_0(x)$ is a  bounded measurable function,   $\dot W(t,x)$  is Gaussian noise with covariance given by
\begin{equation}\label{kernel-general}
    \E[\dot{W}(t,x)\dot{W}(s,y)]= |t-s|^{-\beta_0}\gamma(x-y),
\end{equation}
with $\beta_0\in[0,1)$ and $\gamma(x)$  being a non-negative and non-negative  definite (generalized) function, and $\mathcal{L}$ is the infinitesimal generator of a  { Feller process} $X=\{X_t,t\geq0\}$ which is independent of the noise $\dot W$.  Let $\mu(d\xi)=\hat\gamma(\xi)d\xi$ be the spectral measure of the noise,  i.e., $\int_{\R^d} \varphi(x) \gamma(x) dx = \int_{\R^d} \widehat \varphi(\xi) \mu(d\xi)$ for any Schwartz function $\varphi(x)$. Recall that $\mu$ is a positive and tempered measure, i.e., $\int_{\R^d} (1+|\xi|^2)^{-p} \mu(d\xi) <\infty \text{ for some } p>0.$  Here $\hat f$ (also denoted by $\cF f$) is the Fourier transform of $f$ in the space variable. In particular, for $f\in L^1(\R^d)$, its Fourier transform can be defined by the integral
$\hat f(\xi) =\int_{\R^d} e^{-\iota x\cdot \xi} f(x) dx,$ where $\iota$ is the imaginary unit.

A typical example of $\dot W$  is given by  the partial derivative of fractional Brownian sheet $W^H$ with Hurst parameter $H_0$ in time and $(H_1,\cdots,H_d)$ in space satisfying $H=(H_0, H_1, \dots, H_d)\in(\frac 12,1)^{1+d}$: the Gaussian noise $\dot{W}^H(t,x)=\frac{\partial ^{d+1}W^H}{\partial t\partial x_1\cdots\partial x_d}(t,x)$ is a distribution-valued Gaussian random variable with the covariance function (up to a multiplicative constant)
\begin{equation}\label{kernel}
\E[\dot{W}^H(t,x)\dot{W}^H(s,y)]=|t-s|^{2H_0-2}\prod\limits_{i=1}^d|x_i-y_i|^{2H_i-2}.
\end{equation}
  In this case, $\gamma(x) = \prod\limits_{i=1}^d|x_i|^{2H_i-2}$ with $\hat \gamma(\xi) =\prod_{i=1}^d |\xi_i|^{1-2H_i}$. Another typical example of $\gamma(x)$ is the Dirac delta function $\boldsymbol{\delta}(x)$ whose Fourier transform is the constant 1.  Some other interesting examples of  the spatial covariance function are, for instance, the Riesz kernel $\gamma(x)=|x|^{-\alpha}, \alpha\in(0,d)$ with $\hat \gamma(\xi) = |\xi|^{\alpha-d}$, the Cauchy kernel $\gamma(x)=\prod_{i=1}^d(1+x^2)^{-1}$ with $\hat \gamma(\xi) = \exp\left(\sum_{i=1}^d |\xi_i|\right)$, the Poisson kernel $\gamma(x)=(1+|x|^2)^{-(d+1)/2}$ with $\hat \gamma(\xi) = \exp\left(-|\xi|\right)$, and the Ornstein-Uhlenbeck kernel $\gamma(x)= \exp\left(-|x|^\alpha\right), \alpha\in(0,2]$.

Denote by \[p_{t}^{(x)}(y)=\mathbb P(X_t=y|X_0=x)=\mathbb P(X_{t+s}=y|X_s=x)\]  the transition probability  density of { the Feller process $X$}.    We assume the following condition throughout the paper:
\begin{assumptionp}{(H)}\label{H}
     We assume $p_t^{(x)}(y)\le c_0 P_t(y-x)$ for some non-negative function $P_t(x)$ satisfying     \begin{equation}\label{e:P-bound}
     0\le \hat P_t(\xi) \le  c_1 \exp\left(-c_2 t \Psi(\xi)\right),
     \end{equation}
     where $c_0, c_1$ and $c_2$ are positive constants and $\Psi(\xi)=\Psi(|\xi|)$ is a non-negative measurable function satisfying $\lim_{|\xi|\to \infty} \Psi(\xi) =\infty$.
\end{assumptionp}
Clearly, symmetric L\'evy processes $X$ with $\Psi(\xi)$ being the characteristic exponent of $X$ satisfy Assumption \ref{H}, noting that  $\lim_{|\xi|\to\infty} \Psi(\xi)=\infty$ by Riemann–Lebesgue lemma. Moreover, the diffusion process $X^x$ governed by the stochastic differential equation
\begin{align}\label{x}
	X^x_t=x+\int_0^t b(X^x_s)ds+\int_0^t\sigma(X^x_s)dB_s,\,\,\,t>0,
\end{align}
where $\{B_t,t\geq 0\}$ is a $d$-dimensional Brownian motion, also satisfies Assumption \ref{H}, if we assume the following uniform ellipticity condition:  there exists a constant $c>0$ such that for all $x,y\in\R^d$,
 \begin{align}\label{e:con-p}
 	y^*\sigma(x)\sigma^*(x)y\geq c|y|^2.
 \end{align}
Indeed, under the condition \eqref{e:con-p}, we have (see e.g. \cite{Friedman}) for all $y \in \R^d$, \begin{align} \label{eq-transition density}
	p_t^{(x)} (y) \le c_1 t^{-d/2} \exp \left( - \dfrac{|y-x|^2}{c_2t} \right),
\end{align}
and hence Assumption \ref{H} is satisfied. In this case, the differential operator $\mathcal L$ takes the  form $
	\mathcal L =\frac 12\sum\limits_{i,j=1}^d(\sigma\sigma^{\text{T}})_{i,j}\frac{\partial^2}{\partial x_i\partial x_j}+\sum\limits_{i=1}^db_i\frac{\partial}{\partial x_i}.$

The equation \eqref{spde} can be understood in the Stratonovich sense and in the Skorohod sense, depending on the definition of the product $u\dot W$. We say that it is a Stratonovich equation if the product $u\dot W$ is an ordinary product and a Skorohod equation if it is a Wick product. Note that if the noise is white in time, Skorohod solution coincides with the classical It\^o solution.  

To obtain the existence of the solutions, we propose the following Dalang's conditions (see \cite{Dalang99}): 
\begin{equation}\label{dalang-stra}
    \int_{\R^d} \left(\frac{1}{1+\Psi(\xi)}\right)^{1-\beta_0}\mu(d\xi)<\infty
\end{equation}
and 
\begin{equation}\label{dalang-SKo}
    \int_{\R^d} \frac{1}{1+\Psi(\xi)}\mu(d\xi)<\infty 
\end{equation}
for Stratonovich equation and Skorohod equation, respectively. Clearly, \eqref{dalang-stra} implies \eqref{dalang-SKo}, noting that $\beta_0\in[0,1)$ and $\Psi(\xi)\to \infty$ as $|\xi|\to \infty$.

  Denote by $X^x$ the Feller process $X$  starting at $x$ at time $t=0$. We shall prove that under the stronger condition \eqref{dalang-stra}, the Stratonovich solution $u^{\text{st}}$ and the Skorohod solution $u^{\text{sk}}$ can be represented by the following Feynman--Kac formulas respectively,
\begin{equation}\label{e:FK1}
    u^{\text{st}}(t,x) =\E_{X} \left[u_0(X_t^x)\exp\left( \int_0^t\int_{\R^d} \boldsymbol{\delta}( X^x_{t-s}-y)W(ds,dy)\right)\right]
    \end{equation}
    and 
    \begin{equation}\label{e:FK2}
    \begin{aligned}
    u^{\text{sk}}(t,x)=&\E_{X} \Bigg[u_0(X_t^x)\exp\bigg( \int_0^t\int_{\R^d} \boldsymbol{\delta}( X^x_{t-s}-y)W(ds,dy)\\
    &\qquad\qquad \qquad   -\frac12 \int_0^t\int_0^t |r-s|^{-\beta_0}\gamma(X_r^x-X_s^x)drds\bigg)\Bigg],    
    \end{aligned}
\end{equation}
where $\boldsymbol{\delta}(\cdot)$ is the Dirac delta function, $\E_{X}$ is the expectation in the probability space generated by $X$ (similarly, we will use $\E_W$ to denote the expectation with respect to the noise $\dot W$).

  Note that there exists a unique Skorohod solution under the weaker condition \eqref{dalang-SKo} { (see Theorem~\ref{Thm-solution})} while the Feynman--Kac formula \eqref{e:FK2} may not hold as the It\^o-Stratonovich correction term $\frac12 \int_0^t\int_0^t |r-s|^{-\beta_0}\gamma(X_r^x-X_s^x)drds$ is infinite if the condition \eqref{dalang-stra} is violated. As a contrast, the moments of the Stratonovich solution and the Skorohod solution can be represented by Feynman--Kac type formulas under \eqref{dalang-stra} and under \eqref{dalang-SKo},   respectively.  With the help of Feynman--Kac formula, we also study the properties of the solutions such as the H\"older continuity and regularity of the probability law.

To conclude the introduction, we briefly review several related works and comment on the connections between our results and the existing literature.

  Walsh \cite{Walsh} initiated the theory of stochastic integrals with respect to martingale measures and used it to study  stochastic partial differential equations (SPDEs) driven by space–time white noise; Dalang \cite{Dalang99} extended Walsh’s stochastic integral and applied it to solve SPDEs with Gaussian noise white in time and colored in space. In recent decades, theories of SPDEs driven by Gaussian noise which is white in time have been extensively developed, and we refer to   \cite{DZ92, DKMNX, K14, PR07} and the references therein for more literature.  We remark that the Gaussian noise $\dot W$ in \eqref{spde} is colored in time and hence the approach based on the martingale structure of the noise cannot be applied directly to our problem.

For the heat equation on $\mathbb R$ driven by space-time white noise, it has a unique It\^o (Skorohod) solution which cannot be presented by a formula in the form of \eqref{e:FK2} since the It\^o-Stratonovich correction term is infinite (see \cite{Walsh}). For heat equations driven by Gaussian noise induced by fractional Brownian sheet, the existence of Feynman--Kac formulae was conjectured in  \cite{mv} and was later on established in \cite{HNS11}. The result in \cite{HNS11} was extended to general Gaussian noise in \cite{HHNT2015} and to the equation $\frac{\partial u}{\partial t}=Lu + u\dot W$ in \cite{Song17} where $L$ is the infinitesimal generator of a symmetric L\'evy process.

{ It is natural to conjecture that the results of \cite{HNS11} (resp.~\cite{Song17}) can be extended to SPDE~\eqref{spde}, with the Brownian motion (resp.~symmetric L\'evy process) replaced by the Feller process associated with the differential operator $\mathcal L$. The main difference, as well as the primary difficulty, lies in establishing the exponential integrability of the Hamiltonian
$\int_0^t\int_0^t |r-s|^{-\beta_0}\gamma(X_r^x-X_s^x)drds$
which is required to ensure the well-definedness of the Feynman--Kac formulas \eqref{e:FK1} and \eqref{e:FK2}. In \cite{HNS11}, where $X$ is a Brownian motion, the exponential integrability was established using techniques developed by Le Gall in \cite{LeGall94}, which rely on specific properties of Brownian motion. In \cite{Song17}, the symmetry and independence properties of $X$ also played a crucial role. Our approach (see Section~\ref{sec:Hamiltonian}) is inspired by the recent work \cite{RSW24}, where only the Markov property of $X$, together with certain Fourier analytic techniques, is required. }


A direct consequence  of Feynman--Kac formulae is the non-negative property of the solutions. Under the stronger Dalang's condition \eqref{dalang-stra}, both the Stratonovich solution and the Skorohod solution are strictly positive a.s. (see Propositions \ref{prop:pos} and Proposition \ref{prop:pos1}), while under the weaker Dalang's condition \eqref{dalang-SKo}, the Skorohod solution is non-negative a.s. (see Proposition \ref{prop:pos2}). We refer to \cite{mueller91,flores14,gp17,ch19} and the references therein for related results on the positivity of solutions to stochastic heat equations.

In view of the Feynman--Kac formula, it is convenient to study the H\"older continuity of the Stratonovich solution (see Section~\ref{sec:holder1}). By contrast, the H\"older continuity of the Skorohod solution is analysed through its Wiener chaos expansion (see Section~\ref{sec:holder2}), since in general it does not admit the Feynman--Kac representation \eqref{e:FK2} under condition \eqref{dalang-SKo}. We remark that the H\"older continuity of solutions to (fractional) heat equations has been extensively studied in, for instance, \cite{HNS11, HHNT2015, Song17, BQS2019, HL19}; see also the survey \cite{Hu19} and the references therein.

{ The regularity of the probability distribution of solutions to SPDEs has been extensively studied.} For the one-dimensional heat equation driven by space-time white noise, the absolute continuity of the law of the solution was established in \cite{PZ93} using Malliavin calculus, while the smoothness of the density was proved in \cite{BP98,MN08} under suitable conditions. For SPDEs of the form $Lu=b(u)+\sigma(u)\dot W$, the absolute continuity of the law was obtained in \cite{Tindel00}, where $L$ is a pseudodifferential operator and $\dot W$ is space-time white noise. The smoothness of the density was later established in \cite{NQ07} when $L$ is a parabolic or hyperbolic operator and $\dot W$ is a Gaussian noise that is white in time and homogeneous in space. The smoothness of the density for stochastic heat equations was further investigated in \cite{HNS11,HNS13} by combining Malliavin calculus with Feynman--Kac formulas.

{ In our setting, when the stronger Dalang's condition~\eqref{dalang-stra} holds, the Feynman--Kac formulas can be applied to study the regularity of the law for both Stratonovich and Skorohod solutions. Under the weaker condition~\eqref{dalang-SKo}, however, the analysis becomes more delicate due to the lack of a Feynman--Kac formula for the Skorohod solution. We refer to Sections~\ref{sec:law} and~\ref{sec:law2} for further details on the Stratonovich solution and the Skorohod solution, respectively.}

Intermittency, which in the SPDE literature usually refers to the phenomenon that the solution develops exceptionally high peaks, has been extensively studied, particularly for the so-called parabolic Anderson model (see, e.g., \cite{CM94, FK09, K14, CD15}) and also for the hyperbolic Anderson model (see \cite{CJKS13, BS19, BCC22, CGS23}). The intermittency property is closely related to the long-time asymptotic behaviour of the solution and can be characterized by the $p$th moment Lyapunov exponents, which can often be computed explicitly for (fractional) stochastic heat equations via Feynman--Kac formulas (see \cite{CHSX2015,HLN2017,CHSS2018}). Investigating the precise long-time asymptotics for \eqref{spde} will be the subject of future work.

This paper is organized as follows. In Section \ref{sec:pre}, some preliminaries on the Wiener space associated with the noise $\dot W$, in particular some fundamental ingredients in Malliavin calculus, are collected. In Section \ref{sec:stra} and Section \ref{sec:sko}, the Stratonovich solution and the Skorohod solution are studied, respectively. In Appendix \ref{sec:FK}, the Feynman--Kac formula for PDE $\partial_t u = \mathcal L u + f(t,x)u$ where $\mathcal L$ is the infinitesimal generator of a Feller process is provided. 

Throughout the article, we use $C$ to denote a generic positive constant whose value may change from line to line.

\section{Preliminaries}\label{sec:pre}

In this section, we provide some preliminaries on the Wiener space associated with the Gaussian noise $\dot W$ of which the covariance is given  by \eqref{kernel-general}. In particular, some basic elements in Malliavin calculus is recalled. We refer to \cite{Nualart,Hu} for more details.

Let $\mathcal D(\R_+\times\R^d)$ be the space of smooth functions on $\R_+\times \R^d$ with compact support. Let $\mathcal H$ be the completion of $\mathcal D(\R_+\times\R^d)$ with respect to the inner product given below,
\begin{equation}\label{e:inner-prod}
\begin{aligned}
\la f,g \ra_{\mathcal H}:&=\int_{\R^2_+}\int_{\R^{2d}} f(t,x) g(s,y) |t-s|^{-\beta_0} \gamma(x-y) dxdy ds dt \\
&=\int_{\R_+^2}\int_{\R^d} \hat f(t,\xi) \overline{\hat g(s,\xi)} |t-s|^{-\beta_0}\mu(d\xi) dsdt.
\end{aligned}
\end{equation}

Let $W=\{W(f), f\in\mathcal H\}$ be a centered Gaussian family (also called an  isonormal Gaussian process) defined on a
complete probability space $(\Omega, \mathcal F, P)$, with covariance 
\[\E[W(f)W(g)]=\la f, g \ra_\mathcal H, \,\text{ for } f,g\in \mathcal H. \]
We call $W(f)$ a Wiener integral which is also denoted by $$W(f):=\int_{\R_+}\int_{\R^d} f(s,x) W(ds,dx).$$ 
Denoting $W(t,x):=W(\mathbf{1}_{[0,t]\times [0,x]})$, then $\{W(t,x), t\ge 0, x\in\R^d\}$ is a random field and  we have $\dot W(t,x)=\frac{\partial^{1+d}}{\partial t \partial x_1\cdots\partial x_d}W(t,x)$ in the sense of distribution. It is of interest to consider the following case as a toy model: when $\beta_0=2-2H_0$ and $\gamma(x-y)=\prod_{i=1}^d |x_i-y_i|^{2H_i-2}$ with $H_i\in(0,1)$ for $i=0,1,\dots, d$, the random field $\{W(t,x),t\ge 0, x\in\R^d\}$ is the so-called fractional Brownian sheet (up to some multiplicative constant) with Hurst parameter $H_0$ in time and $(H_1,\dots, H_d)$ in space. 

Denote the $m$th Hermite polynomial by $H_m(x):=(-1)^m e^{x^2/2}\frac{d ^m}{d x^m}e^{-x^2/2}$ for $m\in\mathbb N\cup\{0\}$. For  $g\in \mathcal  H$ with $\|g\|_{\mathcal H}=1$,  the $m$th multiple Wiener integral of $g^{\otimes m}\in \mathcal H^{\otimes m}$ is defined by  
$I_m(g^{\otimes m}):= H_m(W(g)).$ In particular, we have $W(g)= I_1(g)$. For $f\in \mathcal H^{\otimes m}$, denote the  symmetrization of $f$ by 
\[\tilde  f(t_1, x_1, \dots, t_m,x_m) := \frac1{m!}\sum_{\sigma\in S_m} f(t_{\sigma(1)}, x_{\sigma(1)}, \dots,t_{\sigma(m)}, x_{\sigma(m)}),\]
where $S_m$ is the set of all permutations of $\{1,2,\dots,m\}.$  Let ${\mathcal H}^{\tilde \otimes m}$ be the symmetrization  of ${\mathcal H}^{\otimes m}$. Then for $f\in {\mathcal H}^{\tilde \otimes m}$,  one can define the $m$th multiple Wiener integral $I_m(f)$ by a limiting argument.  Moreover, for $f \in {\mathcal H}^{\tilde \otimes m}$ and  $g\in {\mathcal  H}^{\tilde \otimes n}$, we have
\begin{equation}\label{e:ImIn}
	{\mathbb E}[I_m(f)I_n(g)]=m!\langle f,g\rangle_{\mathcal  H^{\otimes m}}\delta_{mn}, 
\end{equation}
where $\delta_{mn}$ is the Kronecker delta function.  For $f\in {\mathcal  H}^{\otimes m}$ which is not necessarily symmetric, we simply define $I_m(f):= I_m(\tilde f).$  For a square integrable random variable $F$ which is measurable with respect to the $\sigma$-algebra generated by $W$, it has the following unique expansion (Wiener chaos expansion): $$F=\E[F]+\sum_{n=1}^\infty I_n(f_n), ~ f_n\in \mathcal H^{\tilde  \otimes n}.$$

Now, let us collect some  knowledge on  Malliavin calculus that will be used in this paper.   Let $F$ be a smooth  cylindrical random variable, i.e., $F$ is of the form
$	F=f(W(\phi_1),\cdots,W(\phi_n)),$
where $\phi_i\in \mathcal H$ and $f$ is a smooth function of which all the derivatives are of polynomial growth. Then the Malliavin derivative $DF$ of $F$ is defined by
\begin{align*}
	DF=\sum\limits_{j=1}^n\frac{\partial f}{\partial x_j}(W(\phi_1),\cdots,W(\phi_n))\phi_j.
\end{align*}
Noting that the operator $D$ is closable from $L^2(\Omega)$ into $L^2(\Omega;\mathcal H)$, we can define the Sobolev space $\mathbb D^{1,2}$ as the closure of the space $\mathfrak S$ of all  smooth cylindrical random variables under the norm
\begin{align*}
	\|F\|_{1,2}=\left(\E[F^2]+\E[\|DF\|^2_{\mathcal H}]\right)^{1/2}.
\end{align*}
Similarly, one can define the $k$th Malliavin derivative $D^k F$ as an $\mathcal H^{\otimes k}$-valued variable for $k\ge 2$, and for any $p>1$ let $\mathbb D^{k,p}$ be the completion 
of $\mathfrak S$ under the norm \[\|F\|_{k,p}=\bigg(\E[|F|^p]+\sum_{j=1}^k\E[\|D^jF\|^p_{\mathcal H^{\otimes j}}]\bigg)^{1/p}.\]Then we define $\mathbb D^\infty = \cap_{k,p=1}^\infty\mathbb D^{k,p}$.

The divergence operator $\boldsymbol{\delta}$ (also called Skorohod integral) is defined by the duality 
\begin{align*}
	\E[\boldsymbol{\delta}(u)F]=\E[\la DF,u\ra_{\mathcal H}]
\end{align*}
for $F\in\mathbb D^{1,2}$ and $u\in \text{Dom}(\boldsymbol{\delta})$, where  $\text{Dom}(\boldsymbol{\delta})$ is the domain of the divergence operator which is the set of $u\in L^2(\Omega,\mathcal H)$ such that $|\E[\la DF, u]\ra_{\mathcal H}|\le c_u \|F\|_2$ for all $F\in\mathbb D^{1,2}$. The second moment of $\boldsymbol{\delta}(u)$ has the following upper bound:
\begin{equation}\label{e:formula1}
\E[|\boldsymbol{\delta}(u)|^2]\le \E [\|u\|^2_{\mathcal H}]+ \E[\|Du\|^2_{\mathcal H^{\otimes2}}].
\end{equation}
The following formula will be used in the proof 
\begin{align}\label{formula}
	FW(\phi)=\boldsymbol{\delta}(F\phi)+\la DF,\phi\ra_\mathcal H,
\end{align}
for $\phi\in \mathcal H$ and $F\in\mathbb D^{1,2}$.


\section{Stratonovich solution}\label{sec:stra}
In this section, we derive the Feynman--Kac formula~\eqref{e:FK1} for the Stratonovich equation~\eqref{spde}.

\subsection{Exponential integrability of the Hamiltonian}\label{sec:Hamiltonian} In this subsection, we  define the Hamiltonian, via approximation,
\begin{equation}\label{e:Ham}
    \int_0^t\int_{\R^d} \boldsymbol{\delta}(X_{t-s}^x-y)W(ds,dy)
\end{equation}    
appearing in Feynman--Kac formulas \eqref{e:FK1}, and  prove its exponential integrability. 


We  use $\varphi_\delta(t):=\frac{1}{\delta}\mathbf{1}_{[0,\delta]}(t), t\ge0$ and $q_\eps(x):=  (2\pi\varepsilon)^{-\frac{d}{2}}e^{\frac{-|x|^2}{2\varepsilon}}, x\in\R^d$ to approximate the Dirac delta functions in  time and in space, respectively. Then, 
\begin{align}\label{A}
	A_{t,x}^{\varepsilon,\delta}(r,y):=\int_0^t\varphi_\delta(t-s-r)q_\eps(X_s^x-y)ds,
\end{align}
is an approximation of  $\delta(X_{t-r}^x-y)$ when $\eps$ and $\delta$ are small. In the following result, we first prove that $A_{t,x}^{\varepsilon,\delta}\in\mathcal H$  almost surely  for all $\varepsilon,\delta>0$, indicating that \begin{align}\label{V}
V_{t,x}^{\varepsilon,\delta}:=\int_0^t\int_{\R^d}A_{t,x}^{\varepsilon,\delta}(r,y)W(dr,dy)=W(A_{t,x}^{\varepsilon,\delta})
\end{align}
is a well-defined Wiener integral (conditional on $X$), and then we establish the $L^2(\Omega)$-convergence of $V_{t,x}^{\varepsilon, \delta}$ as $(\varepsilon, \delta)\to0$, which defines the Hamiltonian given by \eqref{e:Ham}.
\begin{theorem}\label{convergence}
	Assume Assumption \ref{H} and condition \eqref{dalang-stra}. Then,  $A_{t,x}^{\varepsilon,\delta}$ belongs to $\mathcal H$ almost surely for all $\varepsilon, \delta>0$.  Furthermore, 
$V_{t,x}^{\varepsilon,\delta}$ converges in { $L^2$} as $(\eps,\delta)\to0$, with the limit $V_{t,x}$  denoted by
	\begin{align}\label{V_{t,x}}
		V_{t,x}=:\int_0^t\int_{\R^d}\boldsymbol{\delta}(X_{t-r}^x-y)W(dr,dy)=W(\boldsymbol{\delta}(X_{t-\cdot}^x-\cdot)).
	\end{align}
Conditional on $X$, $V_{t,x}$ is  a Gaussian random variable with mean $0$ and variance
\begin{align}\label{cov}
\text{Var}[V_{t,x}|X]=\int_0^t\int_0^t|r-s|^{-\beta_0}\gamma(X_r^x-X_s^x) drds.
\end{align}
\end{theorem}
\begin{proof}
Firstly, in order to show $A_{t,x}^{\varepsilon,\delta}\in\mathcal H$ for $\varepsilon, \delta>0$, we compute 
	\begin{equation} 
 \begin{aligned}
 \label{e:AA}
	\la A_{t,x}^{\varepsilon,\delta},A_{t,x}^{\varepsilon',\delta'}\ra_{\mathcal H}&=\int_{[0,t]^4}\int_{\R^{2d}}\varphi_{\delta}(t-s-u)q_\eps(X_s^x-y)|u-v|^{-\beta_0}\\&\qquad\qquad \times \gamma(y-z)\varphi_{\delta'}(t-r-v)q_{\eps'}(X_r^x-z)dydzdudvdsdr.	
	\end{aligned}
 \end{equation}
 By Lemma  \ref{l2} below, we have, for $\varepsilon, \delta, \varepsilon', \delta'>0$, 
\begin{equation}\label{e:est-A}
	\la A_{t,x}^{\varepsilon,\delta},A_{t,x}^{\varepsilon',\delta'}\ra_{\mathcal H} \le C \int_0^t\int_0^t\int_{\R^{2d}}q_\eps(X_s^x-y)q_{\eps'}(X_r^x-z) |s-r|^{-\beta_0}\gamma(y-z) dydzdrds,
\end{equation}
and then the Parseval–Plancherel identity implies  that 
\begin{align*}
	\la A_{t,x}^{\varepsilon,\delta},A_{t,x}^{\varepsilon',\delta'}\ra_{\mathcal H} &\le C \int_0^t\int_0^t |s-r|^{-\beta_0}  dsdr \times \int_{\R^{d}}\hat q_\eps(\xi)\hat q_{\eps'}(\xi)\mu(d\xi). 
\end{align*}
The product of the integrals on the right-hand side is finite due to the condition $\beta_0<1$ and  the fact \[\int_{\R^{d}}\hat q_\eps(\xi)\hat q_{\eps'}(\xi)\mu(d\xi)\le \int_{\R^{d}}\hat q_\eps(\xi)\mu(d\xi)<\infty,\] where the second step follows from the fact that $\mu(d\xi)$ is a tempered distribution while $\hat q_\eps(\xi)$ decays exponentially fast as $|\xi|\to\infty$. Thus, we have $A_{t,x}^{\varepsilon,\delta}\in \mathcal H$ a.s.

Now we show that $\{A_{t,x}^{\varepsilon,\delta}\}_{\varepsilon, \delta>0}$ is a Cauchy sequence in $L^2(\Omega;\mathcal H)$ as $(\varepsilon,\delta)\to0.$  Taking expectation for \eqref{e:AA}, we have
\begin{align*}
    \E \la A_{t,x}^{\varepsilon,\delta},A_{t,x}^{\varepsilon',\delta'}\ra_{\mathcal H}&=\int_{[0,t]^4}\int_{\R^{2d}}\varphi_{\delta}(t-s-u)\varphi_{\delta'}(t-r-v) |u-v|^{-\beta_0}\\
    &\qquad\qquad \times  q_\eps(y)  q_{\eps'}(z) \E[\gamma(X_s^x-X_r^x-(y-z))]dydz dudvdsdr.
\end{align*}
Without loss of generality,  we assume $r<s$. Then, by the Markov property of $X$ and  Assumption~\ref{H}, we have
\begin{align*}
    &\E[\gamma(X_s^x-X_r^x-(y-z))]=\E\big[\E\left[\gamma(X_s^x-X_r^x-(y-z))|X_{r}^x \right]\big]\\
    &= \E \int_{\R^d} \gamma(u-X_r^x-(y-z)) p_{s-r}^{(X_r^x)}(u) du \\
    &\le c_0\E
\int_{\R^d} \gamma(u-X_r^x-(y-z)) P_{s-r}(u-X_r^x) du\\
    &\le C \int_{\R^d} e^{-c_2 (s-r)\Psi(\xi) }\mu(d\xi),
\end{align*}
and thus 
\begin{align*}
\int_{\R^{2d}} q_\eps(y)  q_{\eps'}(z) \E[\gamma(X_s^x-X_r^x-(y-z))]dydz\le C \int_{\R^d} e^{-c_2 (s-r)\Psi(\xi) }\mu(d\xi). 
\end{align*}

Combining this fact with Lemma \ref{l2}, by dominated convergence theorem we can get that, as $\eps, \delta, \eps',\delta'$ go to zero, 
\[ \E \la A_{t,x}^{\varepsilon,\delta},A_{t,x}^{\varepsilon',\delta'}\ra_{\mathcal H} \to \E\int_0^t\int_0^t |s-r|^{-\beta_0} \gamma(X_s^x-X_r^x)drds,\]
where the right-hand side is finite. 
Indeed, noting $\E\big[\gamma(X_s^x-X_r^x)\big]
\leq C\int_{\R^d} e^{-c_2(s-r)\Psi(\xi)}\mu(d\xi),$ we get 
by Fubini's theorem,
\[
\begin{aligned}
\E\int_0^t\int_0^t |s-r|^{-\beta_0}\gamma(X_s^x-X_r^x)drds 
\leq C\int_{\R^d}\left(\int_0^t\int_0^t |s-r|^{-\beta_0}e^{-c_2(s-r)\Psi(\xi)}drds\right)\mu(d\xi).
\end{aligned}
\]
By applying a change of variables together with Lemma~\ref{lem:time-singular-integral}, we get
\[
\int_0^t\int_0^t |s-r|^{-\beta_0}e^{-c_2(s-r)\Psi(\xi)}drds
=
\int_0^t (t-u)u^{-\beta_0}e^{-c_2u\Psi(\xi)}du\le C_{t,\beta_0} (1+\Psi(\xi))^{\beta_0-1}.
\]
Consequently, this with condition \eqref{dalang-stra} yields
\[
\E\int_0^t\int_0^t |s-r|^{-\beta_0}\gamma(X_s^x-X_r^x)drds
\le C_{t,\beta_0} \int_{\R^d} \frac{1}{(1+\Psi(\xi))^{1-\beta_0}}\mu(d\xi)<\infty.
\] This implies that $\big\{V_{t,x}^{\eps, \delta}=W(A_{t,x}^{\eps, \delta})\big\}_{\eps, \delta>0}$ is an $L^2$-Cauchy sequence  as $(\eps,\delta)\to0$. 

Finally, the formula \eqref{cov} holds true because \[\lim_{\eps,\delta\to 0}\la A_{t,x}^{\varepsilon,\delta},A_{t,x}^{\varepsilon,\delta}\ra_{\mathcal H} = \int_0^t\int_0^t |r-s|^{-\beta_0} \gamma(X_r^x-X_s^x) drds\]
 by Scheff\'e's Lemma, where the limit is taken in  $L^1$.
	\end{proof}

\begin{theorem}\label{thm:exp}
	Assume Assumption \ref{H} and condition \eqref{dalang-stra}. Then, for any $\beta\in \R$, we have\begin{align*}
	\E_W\E_{X}\left[\exp\left(\beta\int_0^t\int_{\R^d}\boldsymbol{\delta}(X^x_{t-r}-y)W(dr,dy)\right)\right]<\infty.
	\end{align*}
\end{theorem}
\begin{proof}
Recalling \eqref{V_{t,x}}, we have
	\begin{align*}
			&\E_W\E_{X}\left[\exp\left(\beta V_{t,x}\right)\right]=\E_X\left[\exp\left(\tfrac{\beta^2}{2}Var[V_{t,x}]\right)\right]\\
   &=\E_X\left[\exp\left(\tfrac{\beta^2}{2}\int_0^t\int_0^t|r-s|^{-\beta_0}\gamma(X_r^x-X_s^x)drds\right)\right].
	\end{align*}
Then the desired result follows from Theorem \ref{exp'}.
\end{proof}

\begin{theorem}\label{exp'}
Assume Assumption \ref{H} and condition \eqref{dalang-stra}. Then, we have
for all $\beta\in\R$,
	\begin{align*}
		\E\left[\exp\left(\beta\int_0^t\int_0^t|r-s|^{-\beta_0}\gamma(X_r^x-X_s^x)drds\right)\right]<\infty.
	\end{align*}
\end{theorem}
\begin{proof}

In the proof we shall omit the superscript $x$ in $X^x_t$. By Taylor's expansion, we have
\begin{align*}
	&\E\left[\exp\left(\beta\int_0^t\int_0^t|r-s|^{-\beta_0}\gamma(X_r-X_s)drds\right)\right]\\
&= \sum_{m=0}^\infty \frac{\beta^m}{m!} \int_{[0,t]^{2m}}\prod_{i=1}^m|s_{2i}-s_{2i-1}|^{-\beta_0} \E\left[\prod_{i=1}^m\gamma(X_{s_{2i}}-X_{s_{2i-1}})\right]d{\bf s} \\
&=\sum\limits_{m=0}^\infty \frac{\beta^m}{m!}\sum\limits_{\tau\in S_{2m}}\int_{[0,t]_{<}^{2m}}\prod\limits_{i=1}^m|s_{\tau(2i)}-s_{\tau(2i-1)}|^{-\beta_0}\E\left[\prod\limits_{i=1}^m\gamma(X_{s_{\tau(2i)}}-X_{s_{\tau(2i-1)}})\right]d{\bf s},
\end{align*}
where $[0,t]^{2m}_< = \{0<s_1<\ldots<s_{2m}<t\}$, $d{\bf s} = ds_1\ldots ds_{2m}$, and $S_{2m}$ is the set of all permutations on $\{1, 2,\dots, 2m\}$.

For each fixed $\tau \in S_{2m}$, we denote by $t_i^+ = \max \{s_{\tau(2i)}, s_{\tau(2i-1)}\}$ and $t_i^- = \min \{s_{\tau(2i)}, s_{\tau(2i-1)}\}$. We have,
\begin{equation}\label{e:Taylor-u}
\begin{aligned}
&\E\left[\exp\left(\beta\int_0^t\int_0^t|r-s|^{-\beta_0}\gamma(X_r-X_s)dsdr\right)\right]\\
&=\sum\limits_{m=0}^\infty\frac{\beta^m}{m!}\sum\limits_{\tau\in S_{2m}}\int_{[0,t]_{<}^{2m}}\prod\limits_{i=1}^m|t_i^+-t_i^-|^{-\beta_0}\E\left[\prod\limits_{i=1}^m\gamma(X_{t_i^+}-X_{t_i^-})\right]d{\bf s},
\end{aligned}
\end{equation}

Let $t_i^*$ be the unique $s_i$ that is closest to $t_i^+$ from the left. Then we always have $t_i^+ \ge t_i^* \ge t_i^-$ for $1 \le i \le m$. Let $k$ be such that $t_k^+ = s_{2m}$. Then,
\begin{align*}
	\E \left[ \prod_{i=1}^m \gamma( X_{t_i^+}-X_{t_i^-}) \right]
	= \E \left[ \prod_{i\not=k} \gamma( X_{t_i^+}-X_{t_i^-} ) \E \left[ \gamma( X_{t_k^+}-X_{t_k^-} ) \bigg| \mathcal F_{t_k^*} \right] \right]
\end{align*}
By Assumption \ref{H}, we have 
\begin{align*}
	& \E \left[ \gamma(X_{t_k^+}-X_{t_k^-} ) \bigg| \mathcal F_{t_k^*} \right]
	= \int_{\R^d} p_{t_k^+-t_k^*}^{(X_{t_k^*})}(y) \gamma( y-X_{t_k^-}) dy \\
	&\le c_0 \int_{\R^d} P_{t_k^+-t_k^*}(y) \gamma( y + X_{t_k^*} - X_{t_k^-} ) dy \\
	&\le c_0c_1 \int_{\R^d} \exp\left(-c_2(t_k^+-t_k^*) \Psi(\xi) \right) \hat \gamma(\xi) \exp\left(-\iota \xi\cdot (X_{t_k^*} - X_{t_k^-}) \right) d\xi \\
 &\le C \int_{\R^d} \exp\left(-c_2(t_k^+-t_k^*) \Psi(\xi) \right)   \mu(d\xi), 
\end{align*}
where the second equality follows from the Parseval-Plancherel identity. Hence, we have
\begin{align*}
	\E \left[ \prod_{i=1}^m \gamma( X_{t_i^+}-X_{t_i^-} ) \right]
	\le C \int_{\R^d} \exp\left(-c_2(t_k^+-t_k^*) \Psi(\xi)\right)   \mu(d\xi) \times \E \left[ \prod_{i\not=k} \gamma( X_{t_i^+}-X_{t_i^-} ) \right],
\end{align*}
and by repeating this procedure, we have
\begin{equation}\label{e:prod-gamma-X}
	\E \left[ \prod_{i=1}^m \gamma(X_{t_i^+}-X_{t_i^-} ) \right]
	\le C^m \int_{\R^{md}}  \prod_{i=1}^m \exp\left(-c_2(t_i^+-t_i^*) \Psi(\xi_i) \right)   \mu(d\xi_i).
\end{equation}
Thus, we have
\begin{align*}
	& \E \left[ \exp \left( \beta \int_0^t\int_0^t |r-s|^{-\beta_0} \gamma(X_r-X_s) drds \right) \right] \\
	\le& \sum_{m=0}^{\infty} C^m \dfrac{\beta^m}{m!}  \sum_{\tau \in S_{2m}} \int_{[0,{ t}]^{2m}_<} d{\bf s}  \int_{\R^{md}}  \boldsymbol{\mu}(d\boldsymbol{\xi})  \exp\left(-c_2\sum_{i=1}^m(t_i^+-t_i^*) \Psi(\xi_i) \right)    \prod_{i=1}^m |t_i^+-t_i^-|^{-\beta_0}.
 \end{align*}

 To deal with the above integrals, 
 we apply the change of variables $r_i=s_{i}-s_{i-1} $ for $i=1,\dots, 2m$ with the convention $s_0=0$ and then we have \[[0,t]_<^{2m}=[0<s_1<\cdots<s_{2m}<t] = [0<r_1+\cdots+r_{2m}<t]\cap \R_{+}^{2m}:=\Sigma^{2m}_t.\] By the symmetry of the integrals, we have  
 \begin{align} \label{eq-exp}
	& \E \left[ \exp \left( \beta \int_0^t\int_0^t |r-s|^{-\beta_0} \gamma(X_r-X_s) drds \right) \right] \nonumber \\
	\le& \sum_{m=0}^{\infty} (C\beta)^m\dfrac{(2m)!}{m!}  \int_{\Sigma_t^{2m}} d{\bf r}  \int_{\R^{md}}  \boldsymbol{\mu}(d\boldsymbol{\xi})  \exp\bigg(-c_2\sum_{i=1}^m r_i \Psi(\xi_i) \bigg)    \prod_{i=1}^m |r_i|^{-\beta_0} \nonumber \\
 \leq & \sum_{m=0}^{\infty} (C\beta)^m \dfrac{(2m)!}{m!} \frac{t^m}{m!} \int_{\Sigma_t^{m}} d{\bf r}  \int_{\R^{md}}  \boldsymbol{\mu}(d\boldsymbol{\xi})  \exp\left(-c_2\sum_{i=1}^m r_i \Psi(\xi_i) \right)    \prod_{i=1}^m |r_i|^{-\beta_0} \nonumber \\
\le  & e^{c_2Mt } \sum_{m=0}^{\infty} (C\beta)^m \dfrac{(2m)!}{m!} \frac{t^m}{m!} \int_{\Sigma_t^{m}} d{\bf r}  \int_{\R^{md}}  \boldsymbol{\mu}(d\boldsymbol{\xi})  \exp\left(-c_2\sum_{i=1}^m r_i (M+\Psi(\xi_i)) \right) \prod_{i=1}^m |r_i|^{-\beta_0} \nonumber \\
  \le & e^{c_2Mt} \sum_{m=0}^{\infty} (C\beta)^m\dfrac{(2m)!}{m!} \frac{t^m}{m!} \left(\int_0^t dr  \int_{\R^{d}} \mu(d\xi)  \exp\left(-c_2 r (M+\Psi(\xi)) \right)    |r|^{-\beta_0}\right)^m \nonumber \\
   \le & e^{c_2Mt} \sum_{m=0}^{\infty} (C\beta)^m \dfrac{(2m)!}{m!} \frac{t^m}{m!} \left(  \int_{\R^{d}} \left(\frac{1}{M+\Psi(\xi)} \right)^{1-\beta_0}\mu(d\xi)\right)^m,
 \end{align}
of which the right-hand side is finite if we choose $M$ sufficiently large, by Stirling's formula and Dalang's condition \eqref{dalang-stra}. 
\end{proof}

The following lemma is borrowed from \cite[Lemma A.3]{HNS11}. 
\begin{lemma}\label{l2}
Suppose that $\alpha\in(0,1)$. There exists a constant $C>0$ such that
\begin{align*}	\sup\limits_{\delta,\delta'>0}\int_0^t\int_0^t\varphi_{\delta}(t-s-r)\varphi_{\delta'}(t-s'-r')|r-r'|^{-\alpha}drdr'\leq C|s-s'|^{-\alpha}.
 \end{align*}
\end{lemma}

{
\begin{lemma}\label{lem:time-singular-integral}
Let $\beta_0\in[0,1)$ and $t>0$. Then, there exists a constant $C_{t,\beta_0}>0$ depending only on $(t,\beta_0)$ such that for all $\lambda\ge0$,
\[
\int_0^t u^{-\beta_0}e^{-\lambda u}\,du
\le C_{t,\beta_0} (1+\lambda)^{\beta_0-1}.
\]
\end{lemma}

\begin{proof}

\textbf{Case 1: $\lambda\in[0,1]$.}
Since $e^{-\lambda u}\le 1$ for all $u\ge0$, we have
\[
\int_0^t u^{-\beta_0}e^{-\lambda u}du
\le \int_0^t u^{-\beta_0}\,du
= \frac{t^{1-\beta_0}}{1-\beta_0}.
\]
On the other hand, for $\lambda\in[0,1]$,
$(1+\lambda)^{\beta_0-1}\ge 2^{\beta_0-1}>0.$
Hence,
\[
\int_0^t u^{-\beta_0}e^{-\lambda u}du
\le \frac{t^{1-\beta_0}}{1-\beta_0}\,2^{1-\beta_0}(1+\lambda)^{\beta_0-1}.
\]

\medskip
\noindent
\textbf{Case 2: $\lambda>1$.}
By the change of variables $v=\lambda u$, we obtain
\[
\int_0^t u^{-\beta_0}e^{-\lambda u}du
= \lambda^{\beta_0-1}\int_0^{\lambda t} v^{-\beta_0}e^{-v}dv\leq \Gamma(1-\beta_0)\lambda^{\beta_0-1}.
\]
Moreover, since $\lambda>1$ and $\beta_0-1<0$, we have
$\lambda^{\beta_0-1}\le 2^{1-\beta_0}(1+\lambda)^{\beta_0-1}.$
Thus,
\[
\int_0^t u^{-\beta_0}e^{-\lambda u}\,du
\le 2^{1-\beta_0}\Gamma(1-\beta_0)(1+\lambda)^{\beta_0-1}.
\]

Thus, the desired inequality follows with 
 $C_{t, \beta_0}:=2^{1-\beta_0}\max\left\{\frac{t^{1-\beta_0}}{1-\beta_0},\,\Gamma(1-\beta_0)\right\}$.
\end{proof}
}

\subsection{Feynman--Kac formulae}\label{sec:FK-str}
In this subsection, we prove that the Feynman--Kac type representation \eqref{e:FK1} is a mild Stratonovich solution to \eqref{spde}. 
Let 
\begin{equation}\label{e:Wed}
\dot W^{\varepsilon,\delta}(t,x)=\int_0^t\int_{\R^d}\varphi_{\delta}(t-s)q_\eps(x-y)W(ds,dy)=W(\varphi_\delta(t-\cdot)q_\eps(x-\cdot))
\end{equation}
be an approximation of $\dot W(t,x)$. As in \cite{HNS11}, we define Stratonovich integral as follows.
\begin{definition}\label{Stratonovich}
	Given a random field $v=\{v(t,x),t\geq 0,x\in\R^d\}$ such that \begin{align*}
		\int_0^T\int_{\R^d}|v(t,x)|dxdt<\infty,
		\end{align*}
	almost surely for all $T>0$, the Stratonovich integral $\int_0^T\int_{\R^d}v(t,x)W(dt,dx)$ is defined as the following limit in probability,
	\begin{align*}
		\lim\limits_{\varepsilon,\delta\to 0} \int_0^T\int_{\R^d}v(t,x)\dot W^{\varepsilon,\delta}(t,x)dxdt.
	\end{align*}
\end{definition}

Denote $\mathcal F^W_t:=\sigma\big\{W(s,x), 0\le s\le t, x\in\R^d\big\}$. A mild solution to \eqref{spde} is defined below.
\begin{definition}\label{def:mild-Stra}
   A random field $\{u(t,x),t\ge0, x\in\R^d\}$ is a mild Stratonovich solution to \eqref{spde} if for all $t\ge 0$ and $x\in\R^d$,  $u(t,x)$ is $\mathcal F_t^W$-measurable and the following integral equation holds
\begin{equation}\label{e:mild-stra}
 u(t,x) = \int_{\R^d} p_t^{(x)}(y) u_0(y) dy +\int_0^t \int_{\R^d} p_{t-s}^{(x)}(y) u(s,y) W(ds,dy),  
 \end{equation}
   where the stochastic integral is in the Stratonovich sense of Definition \ref{Stratonovich}.
\end{definition}

\begin{theorem}\label{thm:FK}
	Assume Assumption \ref{H} and condition \eqref{dalang-stra}. Then,\begin{align}\label{u}
		u(t,x)=\E_X\left[u_0(X_t^x)\exp\left(\int_0^t\int_{\R^d}\boldsymbol{\delta}(X_{t-r}^x-y)W(dr,dy)\right)\right]
	\end{align}
is a mild Stratonovich solution of \eqref{spde}.
\end{theorem}
\begin{proof}
	Consider the approximation of \eqref{spde} 	\begin{equation}\label{approximate}
\left\{\begin{aligned}
\frac{\partial u^{\varepsilon,\delta}(t,x)}{\partial t}=&\mathcal {L} u^{\varepsilon,\delta}(t,x)+u^{\varepsilon,\delta}(t,x) \dot{W}^{\varepsilon,\delta}(t,x),\,\,\ t\ge 0,\, x\in \mathbb R^d,\\
			u^{\varepsilon,\delta}(0,x)=&u_0(x),\,\,\, x\in\mathbb R^d,
		\end{aligned}
	\right.
	\end{equation}
 where $\dot W^{\eps,\delta}$ is given in \eqref{e:Wed}. 
 Recall the definition of $V_{t,x}^{\eps,\delta}=W(A_{t,x}^{\eps,\delta})$ by \eqref{V} and note that 
\begin{equation}\label{e:W-V}
\E_W\E_X\left[\exp\left(\left|\int_0^t\dot W^{\varepsilon,\delta}(t-s,X_{s}^x)ds\right|\right)\right]=\E_W\E_X\left[\exp\left(\left|V_{t,x}^{\varepsilon,\delta}\right|\right)\right].
\end{equation}
We now show that for any $p>0$, 
\begin{equation}\label{e:exp-V}
    \sup_{\eps,\delta>0}\E\left[\exp\left(pV_{t,x}^{\eps,\delta}\right)\right]<\infty.
\end{equation}
Indeed, by \eqref{e:est-A} and Taylor's expansion, we have the following estimation parallel with \eqref{e:Taylor-u}
\begin{align*}
&\E\left[ \exp\left(pV_{t,x}^{\eps,\delta}\right)\right]=\E_{X}\left[\exp\left(\frac{p^2}2\|A_{t,x}^{\eps,\delta}\|^2_{\mathcal H}\right)\right]\\
&\le \sum\limits_{m=0}^\infty C^m \frac{p^{2m}}{m!} \int_{[0,t]^{2m}}\int_{\R^{2md}} \prod_{i=1}^m q_\eps(y_i) q_\eps(z_i) \prod\limits_{i=1}^m|s_{2i}-s_{2i-1}|^{-\beta_0}\\
&\qquad \qquad \qquad \qquad\qquad  \times \E\left[\prod\limits_{i=1}^m\gamma(X_{s_{2i}}-X_{s_{2i-1}}-(y_i-z_i))\right]d{\bf y}d{\bf z}d{\bf s}\\
&
= \sum\limits_{m=0}^\infty C^m \frac{p^{2m}}{m!} \sum\limits_{\tau\in S_{2m}}\int_{[0,t]_{<}^{2m}}\int_{\R^{2md}} \prod_{i=1}^m q_\eps(y_i) q_\eps(z_i) \prod\limits_{i=1}^m|t_i^+-t_i^-|^{-\beta_0}\\
&\qquad\qquad \qquad \qquad \qquad   \times \E\left[\prod\limits_{i=1}^m\gamma(X_{t_i^+}-X_{t_i^-}-(y_i-z_i))\right]d{\bf y}d{\bf z}d{\bf s},
\end{align*}
where the symbols $t_i^{\pm}$ are defined in the proof of Theorem \ref{exp'}. By the same argument that yields \eqref{e:prod-gamma-X}, we can show 
\[\E \left[ \prod_{i=1}^m \gamma(X_{t_i^+}-X_{t_i^-}-(y_i-z_i) ) \right]
	\le C^m \int_{\R^{md}}  \prod_{i=1}^m \exp\left(-c_2(t_i^+-t_i^*) \Psi(\xi_i) \right)   \mu(d\xi_i),\]
the right-hand side of which is independent of  $(y_i,z_i)'s$. Thus, use the same proof as in the rest part of the proof of Theorem \ref{exp'}, we can prove \eqref{e:exp-V}.
Therefore, we have
\begin{equation}\label{e:exp-V'}
\E\left[ \exp\left(p|V_{t,x}^{\eps,\delta}|\right)\right] \le 2 \E\left[ \exp\left(pV_{t,x}^{\eps,\delta}\right)\right]<\infty.
\end{equation}
Noting \eqref{e:W-V}  and \eqref{e:exp-V'}, we can apply the classical Feynman--Kac formula (see Appendix \ref{sec:FK}) and get that 
\begin{equation}\label{e:u-ed}
	u^{\varepsilon,\delta}(t,x)=\E_X\left[u_0(X_t^x)\exp\left(\int_0^t\dot W^{\varepsilon,\delta}(t-s,X_{s}^x)ds\right)\right]=\E_X\left[u_0(X_t^x)\exp\left(W(A_{t,x}^{\varepsilon,\delta})\right)\right],
\end{equation}
is a mild solution to \eqref{approximate}, that is 
\begin{equation}\label{e:mild-sol-app}
u^{\eps,\delta}(t,x) = \int_{\R^d} p_t^{(x)}(y)u_0(y) dy + \int_0^t \int_{\R^d} p_{t-s}^{(x)}(y) u^{\eps, \delta}(s,y) \dot W^{\eps,\delta}(s,y)dyds.
\end{equation}

To prove $u(t,x)$ given in \eqref{u} is a mild solution to \eqref{spde}, it suffices to show that both sides of \eqref{e:mild-sol-app} converge to those of \eqref{e:mild-stra} in probability as $(\varepsilon,\delta)\to 0$, respectively.

 {\bf Step 1.} Firstly,   
we shall prove that for any $(t,x)\in\R_{+}\times \R^d$ and all $p\geq 1$,
\begin{align*}
	\lim\limits_{\varepsilon,\delta\to0}\E\big[|u^{\varepsilon,\delta}(t,x)-u(t,x)|^p\big]=0.
\end{align*}
By Theorem \ref{convergence}, $V^{\varepsilon,\delta}_{t,x}$ converges to $V_{t,x}$ in probability as $(\varepsilon,\delta)\to0$.  Hence, to get the $L^p$-convergence of $u^{\varepsilon,\delta}(t,x)$ to $u(t,x)$, noting that $u_0(x)$ is bounded, we only need to show 
\begin{equation}\label{e:p-mom}
\sup_{x\in\R^d}\sup\limits_{\varepsilon,\delta>0}\E\big[|u^{\varepsilon,\delta}(t,x)|^p\big]<\infty,
\end{equation}
which is a direct consequence of \eqref{e:exp-V}. Moreover, we can show $u^{\eps,\delta}(t,x)\to u(t,x)$ in $\mathbb D^{1,2}$, for which it suffices to show noting that the Malliavin derivative $D$ is closable,
\begin{equation}\label{e:conv-Du}
\lim\limits_{\varepsilon,\delta, \eps',\delta'\to0}\E\big[\|D u^{\varepsilon,\delta}(t,x)-Du^{\varepsilon',\delta'}(t,x)\|^2_{\mathcal H}\big]=0,
\end{equation}
where the Malliavin derivative is taken with respect to $\dot W$.  For simplicity of expressions, we assume $u_0(x)\equiv1$ throughout the rest of the proof. 
Noting that 
\begin{equation}\label{e:Dued}
    D u^{\eps,\delta}(t,x)= \E_X\left[\exp\left(W(A_{t,x}^{\eps,\delta})\right) A_{t,x}^{\eps,\delta} \right], 
    \end{equation}
we have
\begin{align*}
\E\la     D u^{\eps,\delta}(t,x), D u^{\eps',\delta'}(t,x)\ra_\mathcal H=\E\left[ \exp\left(W(A_{t,x}^{\eps,\delta}+\tilde A_{t,x}^{\eps',\delta'})\right) \la A_{t,x}^{\eps,\delta}, \tilde A_{t,x}^{\eps',\delta'}\ra_{\mathcal H}\right], 
\end{align*}
where we recall that $A_{t,x}^{\eps,\delta}(r,y)$ given in \eqref{A} is an approximation of $\delta(X_{t-r}^x-y)$ and $\tilde A_{t,x}^{\eps',\delta'}$ is obtained from replacing $X$ by its independent copy $\tilde X$ in $A_{t,x}^{\eps',\delta'}$.
 {
Fubini's theorem yields
\[
\begin{aligned}
&\E\left[
\exp\left(W(A_{t,x}^{\eps,\delta}
+\tilde A_{t,x}^{\eps',\delta'})\right)
\langle A_{t,x}^{\eps,\delta},
\tilde A_{t,x}^{\eps',\delta'}\rangle_{\mathcal H}
\right]  \\
&=
\E_{X,\tilde X}\left[
\exp\left(\frac12
\|A_{t,x}^{\eps,\delta}
+\tilde A_{t,x}^{\eps',\delta'}\|_{\mathcal H}^2\right)
\langle A_{t,x}^{\eps,\delta},
\tilde A_{t,x}^{\eps',\delta'}\rangle_{\mathcal H}
\right].
\end{aligned}
\]
By the same  argument as in Theorem~\ref{convergence}, we can show 
\[
\langle A_{t,x}^{\eps,\delta},
\tilde A_{t,x}^{\eps',\delta'}\rangle_{\mathcal H}
\to
\int_0^t\int_0^t
|s-r|^{-\beta_0}\gamma(X_s^x-\tilde X_r^x)\,dr\,ds
\]
in $L^1$ as $\eps,\delta,\eps',\delta'\to0$.  Hence, the following $L^1$-convergence holds:
\[
\begin{aligned}
\|A_{t,x}^{\eps,\delta}
+\tilde A_{t,x}^{\eps',\delta'}\|_{\mathcal H}^2 
&=
\|A_{t,x}^{\eps,\delta}\|_{\mathcal H}^2
+
\|\tilde A_{t,x}^{\eps',\delta'}\|_{\mathcal H}^2
+
2\langle A_{t,x}^{\eps,\delta},
\tilde A_{t,x}^{\eps',\delta'}\rangle_{\mathcal H} \\
&\to
\int_0^t\int_0^t |s-r|^{-\beta_0}
\left[
\gamma(X_s^x-X_r^x)
+\gamma(\tilde X_s^x-\tilde X_r^x)
+2\gamma(X_s^x-\tilde X_r^x)
\right]\,dr\,ds. 
\end{aligned}
\] Furthermore, the exponential integrability obtained in
Theorem~\ref{thm:exp} gives the required uniform integrability. Therefore, we can pass to the limit and obtain
\[
\begin{aligned}
&\E\left[
\exp\left(W(A_{t,x}^{\eps,\delta}
+\tilde A_{t,x}^{\eps',\delta'})\right)
\langle A_{t,x}^{\eps,\delta},
\tilde A_{t,x}^{\eps',\delta'}\rangle_{\mathcal H}
\right] \\
&\to
\E\Bigg[
\exp\left(
\frac12\int_0^t\int_0^t |s-r|^{-\beta_0}
\left[
\gamma(X_s^x-X_r^x)
+\gamma(\tilde X_s^x-\tilde X_r^x)
+2\gamma(X_s^x-\tilde X_r^x)
\right]\,dr\,ds
\right) \\
&\qquad\qquad \qquad \times
\int_0^t\int_0^t
|s-r|^{-\beta_0}\gamma(X_s^x-\tilde X_r^x)\,dr\,ds
\Bigg]<\infty.
\end{aligned}
\]
 }

 
 This proves \eqref{e:conv-Du}, and consequently we have
 \begin{equation}\label{e:Du}
     Du(t,x)= \E_X\left[ \exp\left( W(\boldsymbol{\delta}(X_{t-\cdot}^x-\cdot))\right)\boldsymbol{\delta}(X_{t-\cdot}^x-\cdot)\right].
 \end{equation}

{\bf Step 2.}  In this step, we prove the convergence of the right-hand side. Noting Definition \ref{Stratonovich}, it suffices to prove
\[I^{\eps,\delta}:=\int_0^t\int_{\R^d} p_{t-s}^{(x)} (y) v^{\eps, \delta}_{s,y} \dot W^{\eps,\delta}(s,y) dyds\] converges to 0 in $L^1(\Omega)$ as $(\eps,\delta)$ tends to zero, where $v^{\eps,\delta}_{s,y}= u^{\eps,\delta}(s,y)-u(s,y)$.  Applying \eqref{formula} to $v^{\eps,\delta}_{s,y}\dot W^{\eps,\delta}(s,y) =v^{\eps,\delta}_{s,y}W(\phi_{s,y}^{\eps,\delta})$ where \[\phi_{s,y}^{\eps,\delta}(r,z)=\varphi_\delta(s-r)q_\eps(y-z),\] we have
\begin{equation}\label{e:Ied}
\begin{aligned}
I^{\eps,\delta}&= \int_0^t\int_{\R^d} p^{(x)}_{t-s}(y)\left[\boldsymbol{\delta}(v^{\eps,\delta}_{s,y} \phi_{s,y}^{\eps,\delta}(\cdot,\cdot))+ \la D v^{\eps,\delta}_{s,y}, \phi_{s,y}^{\eps,\delta}\ra_{\mathcal H} \right]dyds \\
&=\boldsymbol{\delta}\left( \int_0^t \int_{\R^d} p_{t-s}^{(x)}(y) v_{s,y}^{\eps,\delta} \phi_{s,y}^{\eps,\delta}(\cdot,\cdot) dyds\right)+ \int_0^t\int_{\R^d} p^{(x)}_{t-s}(y)\la D v^{\eps,\delta}_{s,y}, \phi_{s,y}^{\eps,\delta}\ra_{\mathcal H}dyds\\
&=:I_1^{\eps,\delta}+I_2^{\eps,\delta}.
\end{aligned}
\end{equation}
Denote 
\[ \Phi_{t,x}^{\eps,\delta} (r,z) := \int_0^t \int_{\R^d} p_{t-s}^{(x)}(y) v_{s,y}^{\eps,\delta} \phi_{s,y}^{\eps,\delta}(r,z) dyds. \]
To deal with the first term $I_1^{\eps,\delta}=\boldsymbol{\delta}(\Phi_{t,x}^{\eps,\delta})$ of $I^{\eps,\delta}$, by \eqref{e:formula1}  it suffices to show 
\[\lim_{\eps,\delta\to0} \E\big[\|\Phi_{t,x}^{\eps,\delta}\|_{\mathcal H}^2\big] +  \E\big[\|D \Phi_{t,x}^{\eps,\delta}\|_{\mathcal H^{\otimes 2}}^2\big]=0.\]

Note that
\begin{equation}\label{e:Phi-norm}
\E\big[\|\Phi_{t,x}^{\eps,\delta}\|_{\mathcal H}^2\big] =\int_0^t\int_0^t\int_{\R^{2d}} p_{t-s_1}^{(x)}(y_1)p^{(x)}_{t-s_2}(y_2) \E\left[ v_{s_1,y_1}^{\eps,\delta}v_{s_2,y_2}^{\eps,\delta}\right]\la \phi_{s_1,y_1}^{\eps,\delta}, \phi_{s_2,y_2}^{\eps,\delta} \ra_{\mathcal H} d{\bf y} d{\bf s}.
\end{equation}
By Step 1, we know that $\E\big[ v_{s_1,y_1}^{\eps,\delta}v_{s_2,y_2}^{\eps,\delta}\big]$ is uniformly bounded and converges to 0 as $\eps,\delta$ go to zero. For the  integral without $\E\big[ v_{s_1,y_1}^{\eps,\delta}v_{s_2,y_2}^{\eps,\delta}\big]$, noting that
\begin{align*}
   \la \phi_{s_1,y_1}^{\eps,\delta}, \phi_{s_2,y_2}^{\eps,\delta} \ra_{\mathcal H}    &=\int_0^t\int_0^t\int_{\R^{2d}} \phi_{s_1,y_1}^{\eps,\delta}(r_1,z_1)\phi_{s_2,y_2}^{\eps,\delta} (r_2,z_2)|r_1-r_2|^{-\beta_0} \gamma(z_1-z_2)
 d{\bf z} d{\bf r}\\
 &=\int_0^t\int_0^t\int_{\R^d} \varphi_\delta(s_1-r_1) \varphi_\delta(s_2-r_2) |\hat q_\eps(\eta)|^2 e^{\iota \eta\cdot(y_1-y_2)} |r_1-r_2|^{-\beta_0} \mu(d\eta) d{\bf r},
\end{align*}
 we have
\begin{align*}
   &\int_0^t\int_0^t\int_{\R^{2d}} p^{(x)}_{t-s_1}(y_1)p^{(x)}_{t-s_2}(y_2) \la \phi_{s_1,y_1}^{\eps,\delta}, \phi_{s_2,y_2}^{\eps,\delta} \ra_{\mathcal H} d{\bf y} d{\bf s} \\
   &\le c_0^2 \int_0^t\int_0^t\int_{\R^{2d}} P_{t-s_1}(y_1-x)P_{t-s_2}(y_2-x) \la \phi_{s_1,y_1}^{\eps,\delta}, \phi_{s_2,y_2}^{\eps,\delta} \ra_{\mathcal H} d{\bf y} d{\bf s} \\
   &\le C \int_{[0,t]^4}\int_{\R^{d}} \varphi_\delta(s_1-r_1) \varphi_\delta(s_2-r_2) |\hat q_\eps(\eta)|^2 e^{-c_2[(t-s_1)+(t-s_2)]\Psi(\eta)} |r_1-r_2|^{-\beta_0} 
   \mu(d\eta)  d{\bf r}d{\bf s}\\
   &\le C \int_0^t\int_0^t\int_{\R^d}|s_1-s_2|^{-\beta_0}  e^{-c_2(s_1+s_2)\Psi(\eta)}\mu(d\eta) d{\bf s}\\
   &{\le 2C \int_{\R^d} \int_0^t \int_0^r s^{-\beta_0} e^{-c_2 s\Psi(\eta)} ds dr \mu (d\eta)},
   \end{align*}
{  the right-hand side of which is finite by Fubini's theorem, Lemma~\ref{lem:time-singular-integral} and condition \eqref{dalang-stra}.} 
Then, by dominated convergence theorem, we have 
$\lim_{\eps,\delta\to0} \E\big[\|\Phi_{t,x}^{\eps,\delta}\|^2_{\mathcal H}\big]=0.$ Similar to \eqref{e:Phi-norm}, we have
\begin{align*}
\E\big[\|D\Phi_{t,x}^{\eps,\delta}\|^2_{\mathcal H^{\otimes 2}}\big] =\int_0^t\int_0^t\int_{\R^{2d}} p_{t-s_1}^{(x)}(y_1)p^{(x)}_{t-s_2}(y_2) \E\left[\la D v_{s_1,y_1}^{\eps,\delta}, Dv_{s_2,y_2}^{\eps,\delta}\ra_{\mathcal H}\right]\la \phi_{s_1,y_1}^{\eps,\delta}, \phi_{s_2,y_2}^{\eps,\delta} \ra_{\mathcal H} d{\bf y} d{\bf s},
\end{align*}
which converges to 0 as $\eps,\delta$ tend to zero, noting that $\E\big[\big< D v_{s_1,y_1}^{\eps,\delta}, Dv_{s_2,y_2}^{\eps,\delta}\big>_{\mathcal H}\big]$ is uniformly bounded and tends to 0 as $(\eps,\delta)\to 0$ by Step 1. This implies that the first term $I_1^{\eps,\delta}$ in \eqref{e:Ied} converges to 0 in $L^2$. 

To prove the second term $I_2^{\eps,\delta}$ in \eqref{e:Ied} converges to 0 in $L^1$, it suffices to show that both 
$\int_0^t\int_{\R^d} p^{(x)}_{t-s}(y)\big< D u^{\eps,\delta}(s,y), \phi_{s,y}^{\eps,\delta}\big>_{\mathcal H}dyds$ and $\int_0^t\int_{\R^d} p^{(x)}_{t-s}(y)\big< D u(s,y), \phi_{s,y}^{\eps,\delta}\big>_{\mathcal H}dyds$ converge to the same limit in $L^1$ (note that  $\|\phi_{s,y}^{\eps, \delta}\|_{\mathcal H}\to\infty$ as $(\eps, \delta)\to 0$). Note that by \eqref{e:Dued} we have
\begin{align*}
 & \la D u^{\eps,\delta}(s,y), \phi_{s,y}^{\eps,\delta}\ra_{\mathcal H} \\
    =& \int_{[0,s]^3} \int_{\R^{2d}} \varphi_\delta(s-\tau-r_1) \E_X \left[
 \exp\left(W(A_{s,y}^{\eps, \delta})\right) q_\eps(X_\tau^y-z_1) \right]\\
 &\qquad\qquad \times \varphi_\delta(s-r_2) q_\eps(y-z_2) |r_1-r_2|^{-\beta_0}\gamma(z_1-z_2)d{\bf z} d\tau d{\bf r}.
\end{align*}
{
As $\eps,\delta\to0$, the mollifiers force
\[
r_1\to s-\tau,\qquad r_2\to s,\qquad
z_1\to X_\tau^y,\qquad z_2\to y.
\]
Moreover, by Theorem~\ref{convergence},
\[
W(A_{s,y}^{\eps,\delta})
\to
W(\boldsymbol{\delta}(X^y_{s-\cdot}-\cdot))
\quad \text{in }L^2(\Omega).
\]
Therefore, together with the uniform exponential moment estimate \eqref{e:p-mom}, we obtain
\[
\la D u^{\eps,\delta}(s,y),\phi_{s,y}^{\eps,\delta}\ra_{\mathcal H}
\to
\int_0^s \tau^{-\beta_0}
\E_X\left[
\exp\left(W(\boldsymbol{\delta}(X^y_{s-\cdot}-\cdot))\right)
\gamma(X_\tau^y-y)
\right]d\tau
\]
in $L^1(\Omega)$. 
}
This shows that $\int_0^t\int_{\R^d} p^{(x)}_{t-s}(y)\big< D u^{\eps,\delta}(s,y), \phi_{s,y}^{\eps,\delta}\big>_{\mathcal H}dyds$ converges in $L^1$ to \[\E_X \bigg[\int_0^t\int_{\R^d} p_{t-s}^{(x)}(y)\exp\left(W(\boldsymbol{\delta}(X^y_{s-\cdot}-\cdot))\right) \int_0^s \tau^{-\beta_0} \gamma(X_\tau^y-y)  d\tau dy ds\bigg].\] The convergence of $\int_0^t\int_{\R^d} p^{(x)}_{t-s}(y)\big< D u(s,y), \phi_{s,y}^{\eps,\delta}\big>_{\mathcal H}dyds$ to the same limit can be proven in a similar way.  This justifies the $L^1$-convergence of $I_2^{\eps,\delta}$ to zero. 

Combining the convergences of $I_1^{\eps,\delta}$ and $I_2^{\eps,\delta}$, we get that $I^{\eps,\delta}$ converges to 0 in $L^1$. 
\end{proof}

By the Feynman--Kac formula \eqref{u}, it is straightforward to obtain the strict positivity of the Stratonovich solution given some non-degenerate non-negative initial condition. 
\begin{proposition}\label{prop:pos}
    Assume Assumption \ref{H} and the Dalang's condition \eqref{dalang-stra}. We further assume that $u_0(x)$ is a non-negative function satisfying $\mathbb P\{u_0(X_t^x)>0\}>0$. Then the Stratonovich solution given by \eqref{u} is strictly positive a.s. 
\end{proposition}

Another direct application of \eqref{u} yields the Feynman--Kac formula for $p$th moments of the solution to~\eqref{spde}. 
\begin{proposition}
    	Assume Assumption \ref{H} and condition \eqref{dalang-stra}.  Then, we have for $p\in\mathbb N_+$, 
     \begin{align*}
		\E[|u(t,x)|^p]=\E\Bigg[\prod_{i=1}^p u_0(X_t^{i,x})\exp\bigg(\frac12\sum_{1\le i,j\le p} \int_0^t\int_0^t  |s-r|^{-\beta_0} \gamma(X_r^{x,(i)}-X_s^{x,(j)})drds\bigg)\Bigg],
	\end{align*}
 where $X^{x,(i)}, i=1, \dots, p$ are independent copies of $X^x$.
\end{proposition}

\begin{lemma}\label{lemma2}
	Assume Assumption  \ref{H} and condition \eqref{dalang-stra}. Then, we  have \begin{align*}
		\sup_{x,y\in\R^d}\E\left(\int_0^ts^{-\beta_0} \gamma(X_s^x-y)ds\right)^2<\infty.
	\end{align*}
 \end{lemma}
 
\begin{proof}
	Note that
	\begin{align*}
	\E\left(\int_0^ts^{-\beta_0}\gamma(X_s^x-y)ds\right)^2=2\E\int_0^t\int_0^s(sr)^{-\beta_0} \gamma(X_r^x-y)\gamma(X_s^x-y) drds.
	\end{align*}
By Assumption \ref{H} and the Markov property of $X$, we have for $0<r<s<t$, 
\begin{align*}
\E\left[ \gamma(X_r^x-y)\gamma(X_s^x-y) \right]	&=\E\left[\gamma(X_r^x-y)\int_{\R^d}p_{s-r}^{(X_r^x)}(z) \gamma(z-y)dz\right]\\
&\le C \E\left[\gamma(X_r^x-y)\int_{\R^d} \exp\left(-c_2(s-r)\Psi(\xi)\right) \mu(d\xi)\right]\\
&\le C^2 \int_{\R^d} \exp\left(-c_2 r \Psi(\xi)\right) \mu(d\xi) \int_{\R^d} \exp\left(-c_2(s-r)\Psi(\xi)\right) \mu(d\xi).
\end{align*}
Thus,  
\begin{align*}
    &\E\int_0^t\int_0^s(sr)^{-\beta_0} \gamma(X_r^x-y)\gamma(X_s^x-y) drds \\
    &\le C \int_0^t\int_0^s\int_{\R^{d}}  (sr)^{- \beta_0}  \exp\left(-c_2 r \Psi(\xi_1)\right)  \int_{\R^d} \exp\left(-c_2(s-r)\Psi(\xi_2)\right)  \mu(d\xi_2)\mu(d\xi_1)drds \\
    &= C \int_{\left\{\substack{u_1,u_2>0\\0<u_1+u_2<t}\right\}}\int_{\R^{d}} (u_1(u_1+u_2))^{- \beta_0}  \exp\left(-c_2 u_1 \Psi(\xi_1)\right)  \int_{\R^d} \exp\left(-c_2u_2\Psi(\xi_2)\right) \mu(d\xi_2) \mu(d\xi_1) du_1 du_2 \\
    &\le C \left(\int_0^t\int_{\R^d } s^{-\beta_0} \exp\left(-c_2 s\Psi(\xi)\right)   \mu(d\xi) ds\right)^2,
\end{align*}
{
which is finite due to Lemma~\ref{lem:time-singular-integral} and condition \eqref{dalang-stra}.}
\end{proof}

\subsection{H\"older continuity}\label{sec:holder1}

In this subsection, we will study the H\"older continuity for the Stratonovich solution $u(t,x)$ of \eqref{spde}.  In the following, we use the convention $p_t(x):=p_t^{(0)}(x)$. 

\begin{theorem}
 Assume   $p_t^{(x)}(y)=p_t(y-x)$. Furthermore, we assume there exist $\theta_1,\theta_2\in(0,1]$ and $C>0$ such that for all $T\in(0,\infty)$,    \begin{equation}\label{e:Holder1}
 \int_0^T \int_0^T \int_{\R^d}|r-s|^{-\beta_0} p_{|r-s|}(y) \Big[\gamma(y) -\gamma(y+z)\Big]dydr ds \le C |z|^{2\theta_1},
    \end{equation}
    and for all $h$ in a bounded subset of  $\R$,
    \begin{align}\label{e:Holder2}
\sup_{t\in[0,T]}       \int_0^t\int_0^t|r-s|^{-\beta_0}\int_{\R^d}\big(p_{|r-s|}(y)-p_{|r-s+h|}(y)\big)\gamma(y)dydrds\leq C|h|^{\theta_2}.
    \end{align}
Under Assumption \ref{H}  and condition \eqref{dalang-stra},  the solution $u(t,x)$ given by the Feynman--Kac formula \eqref{e:FK1} is $\kappa_1$-H\"older continuous in $x$ with $\kappa_1\in(0,\theta_1)$ and $\kappa_2$-H\"older continuous in $t$ with $\kappa_2\in(0,\frac{\theta_2}{2})$ on any compact set of $[0,\infty)\times \R^d$.  
\end{theorem}
\begin{remark}
If we assume $\hat p_t(\xi) \sim e^{-t\Psi(\xi)}$,  condition \eqref{e:Holder1} is equivalent to the condition in \cite[Hypothesis S1]{Song17}:
\begin{equation*}
    \int_0^T \int_0^T \int_{\R^d}|r-s|^{-\beta_0} \exp\left(-|r-s|\Psi(\xi)\right)(1-\cos(\xi\cdot z))\mu(d\xi) drds \le C|z|^{2\theta_1}, 
    \end{equation*}
    for which a sufficient condition  is (see \cite[Remark 4.10]{Song17})
\[\int_{\R^d}  \frac{|\xi|^{2\theta_1}}{1+(\Psi(\xi))^{1-\beta_0}}\mu(d\xi) <\infty;
\]
and the condition \eqref{e:Holder2} is a consequence of the condition in \cite[Hypothesis T1]{Song17}:
\begin{align*}
    \int_0^T\int_0^T|r-s|^{-\beta_0}\int_{\R^d}\Big|\exp\left(-|r-s|\Psi(\xi)\right)-\exp\left(-|r-s+h|\Psi(\xi)\right)\Big|\mu(d\xi)drds\leq C|h|^{\theta_2}.
    \end{align*}
\end{remark}
\begin{remark}
 If we assume  $\E[ e^{\iota \xi \cdot X_t}]= e^{-t\Psi(\xi)}$,  \cite[Theorem 4.11]{Song17} also holds in our setting.
\end{remark}

\begin{proof}
 Recalling \eqref{V_{t,x}}: $V_{t,x}=\int_0^t \int_{\R^d}  \boldsymbol{\delta}(X_{t-s}^x-y) W(ds,dy)$ and using the inequality $|e^a-e^b|\le (e^a+e^b)|a-b|$, we have for all $p\ge 2$, 
\begin{align*}
   & \E[|u(t_1,x_1)-u(t_2,x_2)|^p]=\E_W\left[\left|\E_X\left[\exp(V_{t_1,x_1})-\exp\left(V_{t_2,x_2}\right)\right]\right|^p\right]\\
   &\le \E_W \left[\left(\E_X\left[\exp(2V_{t_1,x_1}) +\exp(2V_{t_2, x_2})\right] \right)^{p/2} \left(\E_X|V_{t_1,x_1}-V_{t_2,x_2}|^2\right)^{p/2} \right]\\
   &\le C \E\left[\exp(pV_{t_1,x_1})+\exp(pV_{t_2,x_2})\right]\left( \E_W \left[\left(\E_X|V_{t_1, x_1}-V_{t_2,x_2}|^2\right)^p \right] \right)^{1/2}\\
   &\le C \left( \E_W \left[\left(\E_X|V_{t_1, x_1}-V_{t_2,x_2}|^2\right)^p \right] \right)^{1/2}
\end{align*}
where the last step follows from Theorem \ref{thm:exp}.  Applying Minkowski's inequality and noting the equivalence  between the $L^p$-norm and $L^2$-norm of Gaussian random variables, we have
\begin{equation}\label{e:Lp-L2}
\begin{aligned}
    \left( \E_W \left[\left(\E_X|V_{t_1, x_1}-V_{t_2,x_2}|^2\right)^p \right] \right)^{1/2}&\le  \left( \E_X \left[\left(\E_W|V_{t_1, x_1}-V_{t_2,x_2}|^{2p}\right)^{1/p} \right] \right)^{p/2}\\
    &\le C_p \left(\E|V_{t_1,x_1}-V_{t_2, x_2}|^2\right)^{p/2}. 
\end{aligned}
\end{equation}
Thus, we have
\begin{align}\label{e:u-u}
     \E[|u(t_1,x_1)-u(t_2,x_2)|^p] \le C \left(\E|V_{t_1,x_1}-V_{t_2, x_2}|^2\right)^{p/2}. 
\end{align}
Let $t_1=t_2$. Note that
\begin{align*}
 \E|V_{t,x_1}-V_{t, x_2}|^2  &= \int_0^t\int_0^t |r-s|^{-\beta_0}\E_X\left[\gamma(X_r^{x_1}-X_s^{x_1})+\gamma(X_r^{x_2}-X_s^{x_2})-2\gamma(X_r^{x_1}-X_s^{x_2})\right]drds\\
 &=2\int_0^t\int_0^t |r-s|^{-\beta_0}\E_X\left[\gamma(X_r-X_s)-\gamma(X_r-X_s+x_1-x_2)\right]drds,
\end{align*}
 where the second equality holds since $(X_r^x)_{r\ge 0}$ has the same distribution of $(X_r+x=X_r^{0}+x)_{r\ge 0}$ due to the condition $p_t^{(x)}(y)=p_t(y-x).$ Hence, 
\begin{align}\label{e:v-v}
 \E|V_{t,x_1}-V_{t, x_2}|^2 
 &=2\int_0^t\int_0^t \int_{\R^d} |r-s|^{-\beta_0}p_{|r-s|}(y)\Big[\gamma(y) -\gamma(y+x_1-x_2)\Big] dydrds.
\end{align}
Thus, by \eqref{e:Lp-L2}, \eqref{e:u-u}, \eqref{e:v-v} and \eqref{e:Holder1}, we have 
\[\E|u(t,x_1)-u(t,x_2)|^p\le C |x_1-x_2|^{\theta_1 p},\]
for all {$p\ge 2$}. The desired spatial H\"older continuity then follows from Kolmogorov's continuity criterion.

Now we consider the H\"older continuity in time. By \eqref{e:u-u} with $x_1=x_2=x$, we have for $0\leq t_1<t_2\leq T$, 
\begin{align*}
&\E\big[|u(t_2,x)-u(t_1,x)|^p]
\leq C(\E\big|V_{t_2,x}-V_{t_1,x}\big|^2)^{\frac{p}{2}}.
\end{align*}
By the definition of $V_{t,x}$, we get
\begin{align}\label{e:V-V}
&\E\big|V_{t_2,x}-V_{t_1,x}\big|^2
\notag\\&=\E\left[\left(\int_0^{t_1}\int_{\R^d}\big(\delta(X_{t_2-r}^x-y)-\delta(X_{t_1-r}^x-y)\big)W(dr,dy)+\int_{t_1}^{t_2}\int_{\R^d}\delta(X_{t_2-r}^x-y)W(dr,dy)\right)^2\right]\notag\\&\leq 2\E\left[\left(\int_0^{t_1}\int_{\R^d}\big(\delta(X_{t_2-r}^x-y)-\delta(X_{t_1-r}^x-y)\big)W(dr,dy)\right)^2\right]\\
&\qquad +\E\left[\left(\int_{t_1}^{t_2}\int_{\R^d}\delta(X_{t_2-r}^x-y)W(dr,dy)\right)^2\right].\notag
\end{align}
For the first term on the right-hand side, we have
\begin{align}\label{e:bound1}
&\E\left[\left(\int_0^{t_1}\int_{\R^d}\big(\delta(X_{t_2-r}^x-y)-\delta(X_{t_1-r}^x-y)\big)W(dr,dy)\right)^2\right]\notag\\&=\int_0^{t_1}\int_0^{t_1}|s_1-s_2|^{-\beta_0}\E_X \big[\gamma(X_{t_2-s_1}-X_{t_2-s_2})+\gamma(X_{t_1-s_1}-X_{t_1-s_2})-2\gamma(X_{t_2-s_1}-X_{t_1-s_2})\big]ds_1ds_2\notag\\&=2\int_0^{t_1}\int_0^{t_1}|s_1-s_2|^{-\beta_0}\E_X\big[\gamma(X_{|s_1-s_2|})-\gamma(X_{|t_2-t_1-s_1+s_2|})\big]ds_1ds_2\notag\\&
=2\int_0^{t_1}\int_0^{t_1}|s_1-s_2|^{-\beta_0}\int_{\R^d}\big(p_{|s_1-s_2|}(y)-p_{|t_2-t_1-s_1+s_2|}(y)\big)\gamma(y)dyds_1ds_2\\&\leq C|t_2-t_1|^{\theta_2},\notag
\end{align}
where the second equality follows from the stationary increment property which is guaranteed by the condition $p_t^{(x)}(y)=p_t(y-x)$ and the last inequality from the condition \eqref{e:Holder2}.

For the second term on the right-hand side of  \eqref{e:V-V}, we have
\begin{align}\label{e:bound2}
&\E\left[\left(\int_{t_1}^{t_2}\int_{\R^d}\delta(X_{t_2-r}^x-y)W(dr,dy)\right)^2\right]\notag\\&=\int_{t_1}^{t_2}\int_{t_1}^{t_2}|s_1-s_2|^{-\beta_0} \E_X\big[\gamma(X_{t_2-s_1}-X_{t_2-s_2})\big]ds_1ds_2\notag\\&=\int_{t_1}^{t_2}\int_{t_1}^{t_2}\int_{\R^d}|s_1-s_2|^{-\beta_0}p_{|s_1-s_2|}(y)\gamma(y)dyds_1ds_2\\&\leq c_0\int_{t_1}^{t_2}\int_{t_1}^{t_2}\int_{\R^d}|s_1-s_2|^{-\beta_0}P_{|s_1-s_2|}(y)\gamma(y)dyds_1ds_2\notag\\&\leq c_0c_1\int_{t_1}^{t_2}\int_{t_1}^{t_2}\int_{\R^d}|s_1-s_2|^{-\beta_0}\exp\left(-c_2|s_1-s_2|\Psi(\xi)\right)\mu(d\xi)ds_1ds_2\notag\\&\leq C(t_2-t_1)\int_{\R^d}\frac{1}{1+(\Psi(\xi))^{1-\beta_0}}\mu(d\xi)\notag\\&\leq C(t_2-t_1),\notag
\end{align}
{
where  we  used Lemma~\ref{lem:time-singular-integral} and condition \eqref{dalang-stra} in the last two inequalities, respectively. }Combining \eqref{e:Lp-L2}, \eqref{e:V-V}, \eqref{e:bound1}, and \eqref{e:bound2}, we get 
\begin{align*}
\E\big[|u(t_2,x)-u(t_1,x)|^p]\leq C(|t_2-t_1|^{\theta_2}+|t_2-t_1|)^{\frac{p}{2}}\leq C(t_2-t_1)^{\frac{\theta_2p}{2}}.
\end{align*}
The H\"older continuity in time follows from the Kolmogorov's criterion.
\end{proof}

\subsection{Regularity of the law}\label{sec:law}

In this subsection, we shall prove the existence of the probability density of 
 the Stratonovich solution $u(t,x)$ of \eqref{spde} given by \eqref{e:FK1}, and show that the density is smooth under proper conditions. 

Our main tool for studying the probability law is Malliavin calculus. By  Bouleau–Hirsch’s criterion (see \cite{BH91} or \cite[Theorem 2.1.2]{Nualart}), for a random variable $F\in \mathbb D^{1,2}$, if  $\|DF\|_{\mathcal H}>0$ a.s.,  the law of $F$ is absolutely continuous with respect to the Lebesgue measure, i.e., the law of $F$ has a probability density. Moreover, if $F\in \mathbb D^\infty$ and $\E[\|DF\|_\mathcal H^{-p}]<\infty$ for all $p>0$, then  the law of $F$ has an infinitely differentiable density (see \cite[Theorem 2.1.4]{Nualart}).

\begin{theorem}\label{thm:existence-densitiy}
Assume Assumption \ref{H} and condition \eqref{dalang-stra}. Furthermore, suppose that $u_0(x)>0$ almost everywhere and 
\begin{equation}\label{e:existence-density}
    \E[\gamma(X_r^x-\tilde X_s^x)]>0,
\end{equation} where $\tilde X$ is an independent copy of $X$.  Then, the law of $u(t,x)$ given by \eqref{e:FK1} has a probability density. 
\end{theorem}
\begin{remark}
The condition \eqref{e:existence-density} is satisfied, for instance, if we assume $\gamma(x)$ is a strictly positive measurable function, or  if we assume $\gamma(x)=\boldsymbol{\delta}(x)$ and the density function $p_t^{(x)}(y) >0$ for all $x,y\in\R$.  
\end{remark}

\begin{proof}
By Theorem \ref{thm:FK}, we have $u(t,x) =\E_X\left[u_0(X_t^x)\exp\left( V_{t,x}\right) \right]$, where  $V_{t,x}=W(\boldsymbol{\delta}(X_{t-\cdot}^x-\cdot))$ is given by \eqref{V_{t,x}}. By \eqref{e:Du}, The Malliavin derivative of $u(t,x)$ is, 
\begin{equation}\label{e:Du'}
D_{s,y}u(t,x)=\E_X\left[u_0(X_t^x)\exp\left( V_{t,x}\right)\boldsymbol{\delta}(X_{t-s}^x-y) \right], 
\end{equation}
and hence,  denoting $\tilde V_{t,x} = \boldsymbol{\delta}(\tilde X_{t-\cdot}-\cdot)$ where $\tilde X$ is an independent copy of $X$,
 \begin{equation}\label{e:DuH}
 \begin{aligned}
     \|Du(t,x)\|^2_{\mathcal H} &=\E_{X,\tilde X}\left[u_0(X_t^x)u_0(\tilde X_t^x) \exp\left(V_{t,x} +\tilde V_{t,x}\right) \la \boldsymbol{\delta}(X^x_{t-\cdot}-\cdot), \boldsymbol{\delta}(\tilde X^x_{t-\cdot}-\cdot)\ra_{\mathcal H} \right]\\
     &=\E_{X,\tilde X} \left[u_0(X_t^x)u_0(\tilde X_t^x) \exp\left(V_{t,x} +\tilde V_{t,x}\right)\int_0^t\int_0^t |r-s|^{-\beta_0}\gamma(X_r^x-\tilde X_s^x)drds \right].
 \end{aligned}
 \end{equation}
Noting that  $u_0(X_t^x)u_0(\tilde X_t^x) \exp\left(V_{t,x} +\tilde V_{t,x}\right)>0$,  by \eqref{e:existence-density} and \eqref{e:DuH}, we have $\|Du(t,x)\|^2_{\mathcal H}$ is positive a.s. 
\end{proof}

\begin{theorem}\label{thm:smooth-density1}
Assume Assumption \ref{H} and condition \eqref{dalang-stra}. In addition, we assume for all $p>0$, \begin{equation}\label{e:con-smooth} 
 \E|u_0(X_t^x)|^{-p}+\E \left(\int_0^t \int_0^t \gamma(X_r^x-\tilde X_s^x) drds\right)^{-p}<\infty,
\end{equation}
where $\tilde X$ is an independent copy of $X$. Then, the law of $u(t,x)$ given by \eqref{e:FK1} has a smooth probability density. 
\end{theorem}

\begin{proof}

Using a similar argument leading to \eqref{e:Du}, we can show $u(t,x)\in \mathbb D^\infty.$ Then, to prove that $u(t,x)$ has a smooth density, it suffices to show (see \cite[Theorem 2.1.4]{Nualart}) $\E[\|Du(t,x)\|_{\mathcal H}^{-2p}]<\infty.$
 Applying Jensen's inequality to \eqref{e:DuH} yields
\begin{align*}
     \|Du(t,x)\|^{-2p}_{\mathcal H} \le &\E_{X, \tilde X} \left[\left|u_0(X_t^x)u_0(\tilde X_t^x)\right|^{-p} \exp\left(-p[V_{t,x} +\tilde V_{t,x}]\right)\left| \la \boldsymbol{\delta}(X^x_{t-\cdot}-\cdot), \boldsymbol{\delta}(\tilde X^x_{t-\cdot}-\cdot)\ra_{\mathcal H}\right|^{-p}\right],
 \end{align*}
and then H\"older's inequality implies
\begin{align}
     \E \big[\|Du(t,x)\|^{-2p}_{\mathcal H}\big] \le I_1 I_2 I_3 
\end{align}
where $I_1=\left(\E \left|u_0(X_t^x)u_0(\tilde X_t^x)\right|^{-pp_1}\right)^{1/{p_1}}, I_2 = \left(\E \exp\left(-pp_2[V_{t,x} +\tilde V_{t,x}]\right)\right)^{1/p_2}$ and \[I_3=\left(\E \left| \la \boldsymbol{\delta}(X^x_{t-\cdot}-\cdot), \boldsymbol{\delta}(\tilde X^x_{t-\cdot}-\cdot)\ra_{\mathcal H}\right|^{-pp_3}\right)^{1/p_3}, \]
with $\frac1{p_1}+\frac1{p_2}+\frac1{p_3}=1$. By the assumption on $u_0(x)$ and Theorem \ref{thm:exp}, it is clear that both $I_1$ and $I_2$ are finite. For the term $I_3$, by \eqref{e:inner-prod} we have
\begin{align*}
    I_3^{p_3} =& \E \left(\int_0^t \int_0^t |r-s|^{-\beta_0} \gamma(X_r^x-\tilde X_s^x) drds \right)^{-pp_3}
    \le t^{\beta_0pp_3} \E \left(\int_0^t \int_0^t \gamma(X_r^x-\tilde X_s^x) drds\right)^{-pp_3},
\end{align*}
where we use the fact that $|r-s| \le t$. The right-hand side is finite due to the assumption \eqref{e:con-smooth}.
\end{proof}

Condition~\eqref{e:con-smooth} is related to the small-ball probability of $\int_0^t\int_0^t \gamma(X_r^x-\tilde X_s^x) drds$, that is, the probability  $\mathbb P(\int_0^t\int_0^t \gamma(X_r^x-\tilde X_s^x) drds<\eps)$ as $\eps\downarrow0$. { In particular, when $\gamma(x)=\boldsymbol{\delta}(x)$, this becomes the small-ball probability of the mutual intersection local time of $X$, which is usually difficult to estimate. It is known from the recent work~\cite{cs25} that, when $X$ is one-dimensional Brownian motion,  $\mathbb P(\int_0^t\int_0^t \delta(B_r-\tilde B_s) drds<\eps) \sim \eps^{2/3}$. Therefore, the inequality~\eqref{e:con-smooth} holds only for $p\in(0,2/3)$.}  However, if $\gamma(x)$ is a classical measurable function other than the Dirac delta function,  by Jensen's inequality, a sufficient condition for \eqref{e:con-smooth} which is easier to justify { (see Remark~\ref{rem:gamma--p} below)} is the following
    \begin{equation}\label{e:gamma-p}
    \sup_{r,s\in[0,t]} \E \left[|\gamma(X_r^x-\tilde X_s^x)|^{-p}\right]<\infty. 
    \end{equation}

\begin{corollary}
\label{thm:smooth-density11} Assume Assumption \ref{H} and condition \eqref{dalang-stra}.   Suppose that  $\E[|u_0(X_t^x)|^{-p}]<\infty$ and that the covariance function $\gamma$ is a measurable function such that \eqref{e:gamma-p} is satisfied for all $p>0$. Then, the law of $u(t,x)$ given by \eqref{e:FK1} has a smooth probability density. 
\end{corollary}

\begin{remark} \label{rem:gamma--p}   
By Assumption \ref{H}, we have
\begin{equation}\label{e:con-P-gamma}
\E \left[|\gamma(X_r^x-\tilde X_s^x)|^{-p}\right]= \E \int_{\R^d} p_{r}^{(\tilde X_s^{x})}(y)|\gamma(y-
\tilde X_r^x)|^{-p} dy\le C \int_{\R^d} P_r(y) |\gamma(y)|^{-p} dy.
\end{equation}
The right-hand side of \eqref{e:con-P-gamma} is uniformly bounded for $r,s\in[0,t]$, if we assume, for instance, the Feller process  $X$ is a diffusion process given by \eqref{x} whose  density function satisfies \eqref{eq-transition density} and  
\begin{equation}\label{e:con-gamma}
(\gamma(x))^{-1}\le C e^{C|x|^\alpha} \text{ for some } \alpha\in(0,2). 
\end{equation}
The spatial covariance satisfying \eqref{e:con-gamma} includes the kernels mentioned in the Introduction such as the Riesz kernel, Poisson kernel, Cauchy Kernel, and Ornstein-Uhlenbeck kernel with $\alpha\in(0,2)$.  Another typical situation in which \eqref{e:con-P-gamma} is uniformly bounded   is the following: $P_r(y)$ is  a rapidly decreasing function and $(\gamma(x))^{-1}$ is at polynomial growth. 
\end{remark}

\section{Skorohod solution}\label{sec:sko}

In this section, we consider the equation \eqref{spde} in the Skorohod sense. 

\begin{definition} \label{Def-Skorohod}
A random field $u = \{u(t,x): t \ge 0, x \in \bR^d\}$ is  a mild  Skorohod solution to \eqref{spde} if for all $t\ge0$ and $x\in\R^d$, $\E[|u(t,x)|^2]<\infty$, $u(t,x)$ is $\mathcal F_t^W$-measurable,  and the 
 following integral equation holds
\begin{equation}\label{e:mild-Skro}
    u(t,x) = \int_{\R^d}p_t^{(x)}(y)u_0(y)dy + \int_0^t \int_{\bR^d} p^{(x)}_{t-s}(y) u(s,y) W(ds,dy),
\end{equation}
where the stochastic integral is in the Skorohod sense.
\end{definition}

Suppose  $u$ is a mild Skorohod solution to \eqref{spde}.  Then repeating \eqref{e:mild-Skro}, we have the following Wiener chaos expansion of $u(t,x)$:
\begin{align} \label{eq-expansion}
    u(t,x) = \sum_{n=0}^\infty I_n(\tilde f_n(\cdot,t,x)),
\end{align}
where
\begin{align} \label{eq-fn}
    f_n(t_1,x_1,\ldots,t_n,x_n,t,x)
    = p^{(x)}_{t-t_n}(x_n) \ldots p^{(x_2)}_{t_2-t_1}(x_1) \int_{\R^d}p^{(x_1)}_{t_1} (y)) u_0(y)dy \cdot {\bf 1}_{\{0<t_1<\ldots<t_n<t\}}
\end{align}
and $\tilde f_n$ is the symmetrization of $f_n$, i.e., 
\begin{align*}
    \tilde f_n (t_1,x_1,\ldots,t_n,x_n,t,x)
    =& \dfrac{1}{n!} \sum_{\sigma \in S_n} f_n (t_{\sigma(1)},x_{\sigma(1)},\ldots,t_{\sigma(n)},x_{\sigma(n)},t,x) \\
    =& \dfrac{1}{n!}f_n (t_{\tau(1)},x_{\tau(1)},\ldots,t_{\tau(n)},x_{\tau(n)},t,x),
\end{align*}
where $S_n$ is the set of all permutations of $\{1, 2, \dots, n\}$ and $\tau \in S_n$ is the permutation such that $0<t_{\tau(1)}<\ldots<t_{\tau(n)}<t$. In view  of the expansion \eqref{eq-expansion}, there exists a unique  mild Skorohod solution to \eqref{spde} if and only if \begin{align}\label{eq-series}
    \sum_{n=0}^\infty n! \| \tilde f_n (\cdot,t,x) \|^2_{\cH^{\otimes n}}
    < \infty, \quad \text{ for all } \ (t,x) \in [0,T] \times \bR^d.
\end{align}
\subsection{Existence and uniqueness}
\begin{theorem}\label{Thm-solution}
Assume assumption \ref{H} and the Dalang's condition \eqref{dalang-SKo}. Then, \eqref{eq-series} holds, and hence  $u(t,x)$ given by \eqref{eq-expansion} is the unique mild Skorohod solution to \eqref{spde}.
\end{theorem}

\begin{proof}
Without loss of generality, we assume that $u_0 \equiv 1$. We define the following function:
\begin{align} \label{eq-Fn}
    F_n(t_1,x_1,\ldots,t_n,x_n,t,x)
    = P_{t-t_n}(x-x_n) \ldots P_{t_2-t_1}(x_2-x_1) {\bf 1}_{\{0<t_1<\ldots<t_n<t\}}.
\end{align}
Then by assumption \ref{H}, we have
\begin{align} \label{eq-fn<Fn}
    f_n(t_1,x_1,\ldots,t_n,x_n,t,x) \le C^n F_n(t_1,x_1,\ldots,t_n,x_n,t,x).
\end{align}

Note that
\begin{equation}\label{e:FFn}
   \cF F_n(t_1,\cdot,\ldots,t_n,\cdot,t,x) (\xi_1,\ldots,\xi_n)
    = e^{-\iota (\xi_1+\ldots+\xi_n)\cdot x } \prod_{j=1}^n\cF P_{t_{j+1}-t_j}(\xi_1+\ldots+\xi_j) {\bf 1}_{\{0<t_1<\ldots<t_n<t\}}.
\end{equation}
with the convention that $t_{n+1}=t$.
By the definition \eqref{e:inner-prod}  of $\cH$-norm together with \eqref{eq-fn<Fn}, we have
\begin{align*}
    & n! \| \tilde f_n (\cdot,t,x) \|^2_{\cH^{\otimes n}}
    \le n! \| f_n (\cdot,t,x) \|^2_{\cH^{\otimes n}} \nonumber \\
    =& n! \int_{\bR_+^{2n}} \int_{\bR^{2nd}} f_n(t_1,x_1,\ldots,t_n,x_n,t,x) f_n(s_1,y_1,\ldots,s_n,y_n,t,x) \prod_{j=1}^n |t_j-s_j|^{-\beta_0}  \gamma(x_j-y_j) d{\bf x}d{\bf y}d{\bf t} d{\bf s}   \\
    \le& C^{2n} n! \int_{\bR_+^{2n}} \int_{\bR^{2nd}} F_n(t_1,x_1,\ldots,t_n,x_n,t,x) F_n(s_1,y_1,\ldots,s_n,y_n,t,x) \\
    &\qquad \qquad \qquad \qquad \times\prod_{j=1}^n |t_j-s_j|^{-\beta_0}  \gamma(x_j-y_j) d{\bf x}d{\bf y} d{\bf t} d{\bf s}.
\end{align*}
Then, by \eqref{e:P-bound}  and \eqref{e:FFn}
we have
\begin{align*}
    &n! \| \tilde f_n (\cdot,t,x) \|^2_{\cH^{\otimes n}}\\ 
\leq& C^{2n} n! \int_{\bR_+^{2n}}  \prod_{j=1}^n |t_j-s_j|^{-\beta_0} \int_{\bR^{nd}} 
 \cF F_n(t_1,\cdot,\ldots,t_n,\cdot,t,x) (\xi_1,\ldots,\xi_n)\\
    & \qquad \qquad \qquad  \qquad\qquad \qquad  \times \overline{\cF F_n(s_1,\cdot,\ldots,s_n,\cdot,t,x) (\xi_1,\ldots,\xi_n)}\mu(d{\boldsymbol{\xi}}) d{\bf t} d{\bf s}   \\
    =& C^{2n} n! \int_{([0,t]^n_<)^2} \prod_{j=1}^n |t_j-s_j|^{-\beta_0}  \int_{\bR^{nd}} \prod_{j=1}^n \cF P_{t_{j+1}-t_j}(\xi_1+\ldots+\xi_j)  \cF P_{s_{j+1}-s_j}(\xi_1+\ldots+\xi_j)\mu(d\boldsymbol{\xi}) d{\bf t}d{\bf s}  \\
    \le&C^{2n} n! \int_{([0,t]^n_<)^2}  \prod_{j=1}^n |t_j-s_j|^{-\beta_0}   \int_{\R^{nd}}  
 \prod_{j=1}^n \exp \left( -c_2 (t_{j+1}-t_j+s_{j+1}-s_j) \Psi (\xi_1+\ldots+\xi_j) \right)\mu(d\boldsymbol{\xi})d{\bf t}d{\bf s}\\
 \le & C^{2n} n! \int_{([0,t]_<^n)^2}  \prod_{j=1}^n |t_j-s_j|^{-\beta_0}   \int_{\R^{nd}}  
 \prod_{j=1}^n \exp \left( -2c_2 (t_{j+1}-t_j) \Psi (\xi_1+\ldots+\xi_j) \right)\mu(d\boldsymbol{\xi})d{\bf t}d{\bf s},
\end{align*}
{
where in the last step we use the the fact $$2\prod_{j=1}^n\exp ( -c_2 (t_{j+1}-t_j+s_{j+1}-s_j))\le \prod_{j=1}^n\exp ( -2c_2 (t_{j+1}-t_j))+ \prod_{j=1}^n\exp ( -2c_2 (s_{j+1}-s_j))$$ and the symmetry of the integral. Taking integral with respect to $d{\bf s}$ and noting that for $u\in[0,t]$, we have $\int_0^t|s-u|^{-\beta_0}ds\le 2\int_0^t s^{-\beta_0}ds := A_t$, we get 
\begin{equation}\label{eq:tilde f}
\begin{aligned}
&n!\|\tilde f_n(\cdot,t,x)\|^2_{\mathcal H^{\otimes n}}\\
&\le C^n A_t^n\int_{[0,t]_<^n}\int_{\R^{nd}}\prod_{j=1}^n\exp \Big( -2c_2 (t_{j+1}-t_{j})\Psi (\xi_1+\dots+\xi_{j}) \Big)\mu(d\boldsymbol{\xi})d{\bf t}\\
&\le e^{2c_2Mt}C^n A_t^n \int_{[r_i\ge0:\, r_1+\cdots+r_n\le t]}\int_{\R^{nd}}\prod_{j=1}^n\exp \Big( -2c_2 r_j\big(M+\Psi (\xi_1+\dots+\xi_{j})\big) \Big)\mu(d\boldsymbol{\xi})d{\bf t}\\
&\le e^{2c_2Mt}C^n A_t^n \int_{\R^{nd}}\prod_{j=1}^n
\frac1{M+\Psi (\xi_1+\dots+\xi_{j})} \mu(d\boldsymbol{\xi})\\
&\le e^{2c_2Mt}C^n A_t^n \left(\int_{\R^{d}}
\frac1{M+\Psi (\xi)} \mu(d{\xi})\right)^n,
\end{aligned}
\end{equation}
where $M$ is any positive number.  Now, we get
\[
\begin{aligned}
\sum_{n\ge0}n!\|\tilde f_n(\cdot,t,x)\|_{\mathcal H^{\otimes n}}^2
&\le e^{2c_2Mt} \sum_{n\ge0} C^n A_t^n \left(\int_{\R^{d}}
\frac1{M+\Psi (\xi)} \mu(d{\xi})\right)^n,
\end{aligned}
\]
and the right-hand side is finite if we choose $M$ sufficiently large, since  $\lim\limits_{M\to\infty} \int_{\R^{d}}
\frac1{M+\Psi (\xi)} \mu(d{\xi})=0$ due to the condition~\eqref{dalang-SKo}. }
\end{proof}


{
Parallel to   \cite[Theorem 7.2]{HNS11}, we have the following result. }

\begin{theorem} \label{Thm-FK-Skorohod}
Assume Assumption \ref{H} and  the Dalang's condition \eqref{dalang-stra}. Then the unique Skorohod solution to \eqref{spde} can be represented by 
\begin{equation}\label{e:FK-sko}
\begin{aligned}
    u(t,x) = &\E_X \Big[ u_0(X_t^x) \exp \Big( \int_0^t \int_{\R^d} \boldsymbol{\delta}(X_{t-r}^x-y) W(dr,dy) \\
    &\qquad \qquad \qquad \quad - \dfrac{1}{2} \int_0^t \int_0^t |r-s|^{-\beta_0} \gamma(X_r^x-X_s^x) drds \Big)\Big].
\end{aligned}
\end{equation}
\end{theorem}
In parallel with Proposition \ref{prop:pos}, we also have the strict positivity of the Skorohod solution as a direct consequence of the Feynman--Kac formula \eqref{e:FK-sko}.
\begin{proposition}\label{prop:pos1}
Assume the same conditions as in Proposition \ref{prop:pos}. Then the Skorohod solution is strictly positive a.s. 
\end{proposition}

Theorem \ref{Thm-solution} indicates that under the condition \eqref{dalang-SKo} (which is weaker than \eqref{dalang-stra}), there is a unique Skorohod solution. On the other hand,  as shown in \cite[Proposition 3.2]{HNS11}, if the condition \eqref{dalang-stra} is violated, the sequence $V_{t,x}^{\eps, \delta}$ may diverge and hence the term  $\int_0^t \int_{\R^d} \boldsymbol{\delta}(X_{t-r}^x-y) W(dr,dy)$ given in \eqref{V_{t,x}} may not be well-defined. In this situation, Feynman--Kac representation of the form \eqref{e:FK2} for the Skorohod solution  is not available.  However, we do have Feynman--Kac representation for the moments of the Skorohod solution under \eqref{dalang-SKo}.

\begin{theorem}\label{Thm-FK-moment}
Assume  Assumption \ref{H} and  the Dalang's condition \eqref{dalang-SKo}. Then  for  $p \in \bN_+$, we have the following representation for the moments of the  Skorohod solution to \eqref{spde}:
\begin{align}\label{e:FK-mom}
    \bE \left[ u(t,x)^p \right]
    = \bE \left[ \prod_{j=1}^p u_0\big(X_t^{x,(j)}\big) \exp \left( \sum_{1 \le j < k \le p} \int_0^t \int_0^t |r-s|^{-\beta_0} \gamma \left( X_r^{x,(j)} - X_s^{x,(k)} \right) drds \right) \right],
\end{align}
where $X^{x,(1)}, \ldots, X^{x,(p)}$ are i.i.d. copies of $X^x$.
\end{theorem}

\begin{proof}
Without loss of generality, we assume $u_0(x) \equiv 1$. Similar to \eqref{approximate}, we consider the following approximation  of \eqref{spde}:
\begin{equation}\label{approximate1}
\left\{\begin{aligned}
\frac{\partial u^{\varepsilon,\delta}(t,x)}{\partial t}=&\mathcal {L} u^{\varepsilon,\delta}(t,x)+u^{\varepsilon,\delta}(t,x) 
\diamond\dot{W}^{\varepsilon,\delta}(t,x),\,\,\ t\ge 0,\, x\in \mathbb R^d,\\
 u^{\varepsilon,\delta}(0,x)=&u_0(x),\,\,\, x\in\mathbb R^d, 
		\end{aligned}
	\right.
	\end{equation}
 where the symbol $\diamond$ means Wick product and $\dot W^{\eps,\delta}$ is given in \eqref{e:Wed}.
The mid Skorohod solution $u^{\eps,\delta}(t,x)$ is given by the following integral equation
\begin{align*}
    u^{\eps,\delta}(t,x)
    =& 1 + \int_0^t \int_0^t \int_{\bR^{2d}}  p^{(x)}_{t-s}(y) \varphi_{\delta}(s-r) q_{\eps}(y-z) u^{\eps,\delta}(s,y) \diamond W(dr,dz) dsdy, 
\end{align*}
where the integral on the right-hand side is in the Skorohod sense. 
Recall the notation $A_{t,x}^{\eps,\delta}$ given by \eqref{A}.
By a similar argument used in Step 1 in the proof of \cite[Theorem 5.6]{Song17} (see also \cite[Proposition 5.2]{HN09}), we can show that $W(A_{t,x}^{\eps,\delta})$ is well-defined and  the following Feynman--Kac formula holds:
\begin{align} \label{eq-FK-approx}
    u^{\eps,\delta}(t,x)
    = \bE_X \left[ \exp \left( W \left( A_{t,x}^{\eps,\delta} \right) - \dfrac{1}{2} \left\| A_{t,x}^{\eps,\delta} \right\|_{\cH}^2 \right) \right].
\end{align}
As in Step 3 of the proof of  \cite[Theorem 5.6]{Song17}, we can demonstrate that $u^{\eps,\delta}(t,x)$ converges to $u(t,x)$ in $L^p$ for all $p>0$ as $(\eps,\delta) \to 0$.
By the formula \eqref{eq-FK-approx}, we can explicitly calculate the $p$th moment of $u^{\eps,\delta}(t,x)$:
\begin{align} \label{eq-moment-1}
    \bE \left[ \left( u^{\eps,\delta}(t,x) \right)^p \right]
    =& \bE \Bigg[ \exp \Bigg( W \bigg( \sum_{j=1}^p A^{\eps,\delta,(j)}_{t,x} \bigg) - \dfrac{1}{2} \sum_{j=1}^p \left\| A^{\eps,\delta,(j)}_{t,x} \right\|_{\cH}^2 \Bigg) \Bigg] \nonumber \\
    =& \bE_X \Bigg[ \exp \Bigg( \dfrac{1}{2} \bigg\| \sum_{j=1}^p A^{\eps,\delta,(j)}_{t,x} \bigg\|_{\cH}^2 - \dfrac{1}{2} \sum_{j=1}^p \left\| A^{\eps,\delta,(j)}_{t,x} \right\|_{\cH}^2 \Bigg) \Bigg] \nonumber \\
    =& \bE_X \Bigg[ \exp \Bigg( \sum_{1 \le j<k \le p} \left\langle A^{\eps,\delta,(j)}_{t,x}, A^{\eps,\delta,(k)}_{t,x} \right\rangle_{\cH} \Bigg) \Bigg],
\end{align}
where $A_{t,x}^{\eps,\delta,(j)}$ is given by \eqref{A} with $X$ being replaced by $X^{(j)}$, an independent copy of $X$, and $\bE_X$ is the expectation with respect to  $X^{(1)}, \ldots, X^{(p)}$.
Therefore, employing the same technique as in the proof of Theorem \ref{convergence}, we can establish that
\begin{align} \label{eq-convergence-inner prod}
    \lim_{\eps,\delta \to 0} \left\langle A^{\eps,\delta,(j)}_{t,x}, A^{\eps,\delta,(k)}_{t,x} \right\rangle_{\cH}
    = \int_{0}^t\int_{0}^t |r_1-r_2|^{-\beta_0} \gamma \left( X_{r_2}^{x,(j)} - X_{r_1}^{x,(k)} \right) dr_1dr_2,
\end{align}
in the sense of  $L^1$-convergence. We claim that under Dalang's condition \eqref{dalang-SKo}, it holds that for all $\lambda>0$,
\begin{align} \label{eq-bound-inner prod}
    \sup_{\epsilon,\delta>0} \bE \left[ \exp \left( \lambda \left\langle A^{\eps,\delta,(j)}_{t,x}, A^{\eps,\delta,(k)}_{t,x} \right\rangle_{\cH} \right) \right] < \infty, \, j\neq k.
\end{align}
Then the desired result \eqref{e:FK-mom} follows from \eqref{eq-moment-1}, \eqref{eq-convergence-inner prod} and \eqref{eq-bound-inner prod}.

To conclude the proof, we prove the claim \eqref{eq-bound-inner prod}. By \eqref{A} and Lemma \ref{l2}, we have
\begin{align*}
    \left\langle A^{\eps,\delta,(j)}_{t,x}, A^{\eps,\delta,(k)}_{t,x} \right\rangle_{\cH}
    \le& C \int_{0}^t\int_{0}^t |s-r|^{-\beta_0} drds \int_{\bR^{2d}} q_{\eps} (X_s^{x,(j)}-y) q_{\eps} (X_r^{x,(k)}-z) \gamma(y-z) dydz \\
    =& C \int_{0}^t\int_{0}^t |s-r|^{-\beta_0} drds \int_{\bR^{2d}} q_{\eps} (y) q_{\eps} (z) \gamma(X_s^{x,(j)}-X_r^{x,(k)}-y+z) dydz.
\end{align*}
By Taylor expansion, we get
\begin{equation}\label{eq-Eexp-inner}
\begin{aligned}
    &\bE \left[ \exp \left( \lambda \left\langle A^{\eps,\delta,(j)}_{t,x}, A^{\eps,\delta,(k)}_{t,x} \right\rangle_{\cH} \right) \right] \\
   & \le 1 + \sum_{m=1}^{\infty} \dfrac{\lambda^m C^m}{m!} \int_{[0,t]^{2m}}   d{\bf r}d{\bf s} \, \prod_{l=1}^m |s_l-r_l|^{-\beta_0}  \\
    & \qquad \times \int_{\bR^{2md}} \prod_{l=1}^m q_{\eps} (y_l)q_{\eps}(z_l) \bE \left[ \prod_{l=1}^m \gamma(X_{s_l}^{x,(j)}-X_{r_l}^{x,(k)}-y_l+z_l) \right] d{\bf y} d{\bf z}.
\end{aligned}
\end{equation}

To compute the expectation in \eqref{eq-Eexp-inner}, we apply Lemma \ref{Lem-prod-1} below to the  conditional expectation given $X^{(k)}$ and obtain that, for $0=s_0< s_1<\cdots, s_m<t$,
\begin{equation}\label{e:prod-gamma-X-X}
    \bE \left[ \prod_{l=1}^m \gamma(X_{s_l}^{x,(j)}-X_{r_l}^{x,(k)}-y_l+z_l) \right]
    \le C^m c_1^m \int_{\bR^{md}} \prod_{l=1}^m \exp \left(-c_2 (s_l-s_{l-1}) \Psi(\xi_l) \right) \mu(d\xi_1) \ldots \mu(d\xi_m).
\end{equation}
 Note that \eqref{e:prod-gamma-X-X} is independent of $y_1, \ldots, y_m, z_1, \ldots, z_m$ and $r_1,\ldots, r_m$, and  hence, 
 to estimate \eqref{eq-Eexp-inner}, we can integrate with respect to $dy_1\ldots dy_m dz_1 \ldots dz_m$ first then  to $dr_1\ldots dr_m$   using the fact 
\begin{equation}
\label{eq-integrate-s-t}
\int_{[0,t]^n} \prod_{j=1}^n |s_j-r_j|^{-\beta_0} d{\bf r}
    \le A_t^n,\, \text{ with } A_t := 2 \int_0^t |s|^{-\beta_0} ds<\infty,
    \end{equation}
 to obtain
\begin{align*}
    &\bE \left[ \exp \left( \lambda \left\langle A^{\eps,\delta,(j)}_{t,x}, A^{\eps,\delta,(k)}_{t,x} \right\rangle_{\cH} \right) \right] \\
    &\le 1 + \sum_{m=1}^{\infty} \dfrac{\lambda^m C^{m} c_1^m}{m!} \int_{[0,t]^{2m}}  d{\bf s}d{\bf r}\prod_{l=1}^m |s_l-r_l|^{-\beta_0}  \int_{\bR^{md}} \prod_{l=1}^m \exp \left(-c_2 (s_l-s_{l-1}) \Psi(\xi_l) \right)\mu(d\boldsymbol{\xi}) \\
    &\le 1 + \sum_{m=1}^{\infty} A_t^m \lambda^m C^{m} c_1^m  \int_{ [0,t]_<^m} d{\bf s}
    \int_{\bR^{md}} \prod_{l=1}^m \exp \left(-c_2 (s_l-s_{l-1}) \Psi(\xi_l) \right)\mu(d\boldsymbol{\xi}).
\end{align*}
Then we can use the same technique leading to \eqref{eq-exp} to prove the convergence of the series. Applying the change of variable $\tau_l=s_l-s_{l-1}$ for $l=1, \dots, m$ and  recalling the notation $\Sigma_t^m=[0<\tau_1+\cdots+\tau_m <t]\cap\R_+^m$, we have
\begin{align} \label{ineq-exp-nondiag}
&\bE \left[ \exp \left( \lambda \left\langle A^{(\eps,\delta,j)}_{t,x}, A^{(\eps,\delta,k)}_{t,x} \right\rangle_{\cH} \right) \right] \nonumber \\
    &= 1+  \sum_{m=1}^{\infty} A_t^m \lambda^m C^{m} c_1^m \int_{\Sigma_t^{m}} d{\boldsymbol{\tau}}  \int_{\R^{md}}  \boldsymbol{\mu}(d\boldsymbol{\xi})  \exp\left(-c_2\sum_{i=1}^m r_i \Psi(\xi_i) \right)  \nonumber   \\
&\le  1+ e^{c_2Mt } \sum_{m=1}^{\infty}A_t^m \lambda^m C^{m} c_1^m  \int_{\Sigma_t^{m}} d{\boldsymbol{\tau}}    \int_{\R^{md}}  \boldsymbol{\mu}(d\boldsymbol{\xi})  \exp\left(-c_2\sum_{i=1}^m \tau_i (M+\Psi(\xi_i)) \right) \nonumber \\
  &\le 1+ e^{c_2Mt} \sum_{m=1}^{\infty} A_t^m \lambda^m C^{m} c_1^m \left(\int_0^t d\tau   \int_{\R^{d}} \mu(d\xi)  \exp\left(-c_2 \tau (M+\Psi(\xi)) \right)    \right)^m \nonumber \\
   & \le1+ e^{c_2Mt} \sum_{m=1}^{\infty} A_t^m \lambda^m C^{m} c_1^m \left(  \int_{\R^{d}} \frac{1}{M+\Psi(\xi)} \mu(d\xi)\right)^m. 
\end{align}
Thus, by the Dalang's condition \eqref{dalang-SKo}, we can find $M$ sufficiently large such that the above series converges.  
\end{proof}

In the above proof, we have shown that the Skorohod solution $u(t,x)$ is a limit of $u^{\eps,\delta}$ given by \eqref{eq-FK-approx} which is obviously positive. This yields the non-negativity of $u(t,x)$ given non-negative initial condition as stated in Proposition~\ref{prop:pos2} below. We remark that under the weaker condition \eqref{dalang-SKo}, it is still not clear whether $u(t,x)>0$ a.s. (see Proposition~\ref{prop:pos1} for a comparison).
\begin{proposition}\label{prop:pos2}
Assume Assumption \ref{H} and the Dalang's condition \eqref{dalang-SKo}. Assume further that $u_0(x)$ is a non-negative function. Then  the  Skorohod solution $u(t,x)$ to \eqref{spde} is non-negative a.s. 
\end{proposition}

The following lemma has been used in the proof of Theorem \ref{Thm-FK-moment}.

\begin{lemma} \label{Lem-prod-1}
For  $0 =r_0< r_1 < \cdots < r_m < \infty$, we have 
\begin{align*}
\sup_{(a_1,\dots,a_m)\in\R^m} \bE \Bigg[ \prod_{i=1}^m \gamma \left( X_{r_i}^x - a_i \right) \Bigg]
    \le C^m c_1^m \int_{\bR^{md}} \prod_{j=1}^m \exp \Big(-c_2 (r_j-r_{j-1}) \Psi(\xi_j) \Big) \mu(d{\boldsymbol{\xi}}).
\end{align*}
\end{lemma}

\begin{proof}
The proof essentially follows the argument leading to \eqref{e:prod-gamma-X}. By Markov property of $X$, we have
\begin{align*} 
    \bE \left[ \prod_{i=1}^m \gamma \left( X_{r_i}^x - a_i \right) \right]
    =& \bE \left[ \prod_{i=1}^{m-1} \gamma \left( X_{r_i}^x - a_i \right) \bE \left[ \gamma \left( X_{r_m}^x - a_m \right) \big| \cF_{r_{m-1}} \right] \right] \nonumber \\
    =& \bE \left[ \prod_{i=1}^{m-1} \gamma \left( X_{r_i}^x - a_i \right) \bE \left[ \gamma \left( X_{r_m}^x - a_m \right) \big| X_{r_{m-1}}^x \right] \right] \nonumber \\
    =& \bE \left[ \prod_{i=1}^{m-1} \gamma \left( X_{r_i}^x - a_i \right) \int_{\bR^d} \gamma (y_m-a_m) p_{r_m-r_{m-1}}^{(X_{r_{m-1}}^x)}(y_m) dy_m \right]. 
\end{align*}
Applying assumption \ref{H}  and Parseval-Plancherel identity yields that 
    \begin{align*}
  \int_{\bR^d} \gamma (y_m-a_m) p_{r_m-r_{m-1}}^{(X_{r_{m-1}}^x)}(y_m) dy_m  \le& c_0 \int_{\bR^d} \gamma (y_m-a_m) P_{r_m-r_{m-1}}(y_m-X^x_{r_{m-1}}) dy_m  \\
    =&  c_0\int_{\bR^d} \exp (-i(X^x_{r_{m-1}}-a_m) \cdot \xi_m) \hat P_{r_m-r_{m-1}}(\xi_m) \mu(d\xi_m)  \\
    \le& C  \int_{\bR^d} \exp \big(-c_2 (r_m-r_{m-1}) \Psi(\xi_m) \big) \mu(d\xi_m).
\end{align*} 
Combining the above calculations, we get
\[ \bE \left[ \prod_{i=1}^m \gamma \left( X^x_{r_i} - a_i \right) \right]\le  C \bE \left[ \prod_{i=1}^{m-1} \gamma \left( X^x_{r_i} - a_i \right) \right] \int_{\bR^d} \exp \big(-c_2 (r_m-r_{m-1}) \Psi(\xi_m) \big) \mu(d\xi_m).\]
The proof is then concluded by repeating this procedure. 
\end{proof}
\subsection{Continuity and H\"older continuity}\label{sec:holder2}
Let $u(t,x)$ be the unique mild Skorohod solution to~\eqref{spde}. In this subsection, we study the continuity and H\"{o}lder continuity of $u(t,x)$. For simplicity, we assume that $u_0(x) \equiv 1$. In this case, the chaos decomposition \eqref{eq-expansion} of $u(t,x)$  holds with
\begin{align*}
    f_n(t_1,x_1,\ldots,t_n,x_n,t,x)
    = p^{(x_n)}_{t-t_n}(x) \ldots p^{(x_1)}_{t_2-t_1}(x_2) \times {\bf 1}_{\{0<t_1<\ldots<t_n<t\}}.
\end{align*}
In this section, we use the convention $p_t(x):=p_t^{(0)}(x)$.  We have the following maximal principle. \begin{lemma} \label{Lem-max}
For any $a \in \bR^d$, $t \in \bR_+$, we have
\begin{align*}
    \int_{\bR^d} \left| \cF p_t (\xi+a) \right|^2 \mu(d\xi)
    \le \int_{\bR^d} \left| \cF p_t (\xi) \right|^2 \mu(d\xi).
\end{align*}
\end{lemma}
\begin{proof}
By Parseval–Plancherel identity, we have 
\begin{align*}
     \int_{\bR^d} \left| \cF p_t (\xi+a) \right|^2 \mu(d\xi)&=\int_{\R^{2d}} p_t(x) e^{-\iota x\cdot a} p_t(y) e^{-\iota y\cdot a} \gamma(x-y) dx dy \\
     &\le \int_{\R^{2d}} p_t(x) p_t(y) \gamma(x-y) dx dy=\int_{\bR^d} \left| \cF p_t (\xi) \right|^2 \mu(d\xi).
\end{align*}
\end{proof}
\begin{theorem}\label{continuity of u}
Assume assumption \ref{H} and the Dalang's condition \eqref{dalang-SKo}. Then for $p\geq 2$, the Skorohod solution $u(t,x)$ to \eqref{spde} is $L^p(\Omega)$-continuous  in $t$ and $x$.
\end{theorem}
\begin{proof}For $p\geq 2$ and $(t,x)\in [0,T]\times \R^d$, by using Minkowski's inequality and hypercontractivity (see, e.g., \cite[Theorem 1.4.1]{Nualart}), we have 
\begin{align}\label{hyper}
    \|u(t,x)\|_{L^p}&\leq\sum_{n\geq 0}\|I_n(\tilde f_n(\cdot,t,x))\|_{L^p}\nonumber\\&\leq \sum_{n\geq 0}(p-1)^{\frac{n}{2}}\|I_n(\tilde f_n(\cdot,t,x))\|_{L^2}\\
    &=\sum_{n\geq 0}(p-1)^{\frac{n}{2}}(n!\|\tilde f_n(\cdot,t,x)\|_{\cH^{\otimes n}}^2)^{\frac 12}.\nonumber
\end{align}
 {Thus,  using \eqref{eq:tilde f}, we get
\begin{align*}
    \|u(t,x)\|_{L^p}
    &\le \sum_{n\geq 0} (p-1)^{\frac n2} e^{c_2Mt} (CA_t)^{\frac n2} \left(\int_{\R^d}\frac1{M+\Psi(\xi)}\mu(d\xi)\right)^{n/2}< \infty,
\end{align*}
if we choose $M$ sufficiently  large.} Then we have for any $T>0$, $p\geq 2$, there exists a constant $C_{T,p}$ such that
\begin{align*}
\sum_{n\geq 0}\sup_{(t,x)\in [0,T]\times \R^d}\|I_n(\tilde f_n(\cdot,t,x))\|_{L^p}\leq C_{T,p}<\infty.
\end{align*}

Hence, the sequence $\big\{u_n(t,x)=\sum_{k=0}^n I_k(\tilde f_k(\cdot,t,x))\big\}_{n\in\mathbb N}$ converges to $u(t,x)$ in $L^p(\Omega)$ as $n\to\infty$, uniformly in $(t,x)\in [0,T]\times \R^d$. By Lemma \ref{continuity of I} below, $I_n(\tilde f_n(\cdot,t,x))$ is $L^p(\Omega)$-continuous, and hence $u_n$
is $L^p(\Omega)$-continuous. Therefore,  $u$ is $L^p(\Omega)$-continuous.
\end{proof}

\begin{lemma}\label{continuity of I} Assume assumption \ref{H} and the Dalang's condition \eqref{dalang-SKo}.
For any $p\geq 2$, $n\geq 1$, $t>0$ and $x\in\R^d$,  we have
\begin{align}\label{spatial continuity}
\lim_{|z|\to0}\E|I_n(\tilde f_n(\cdot,t,x+z))-I_n(\tilde f_n(\cdot,t,x))|^p=0,
\end{align}
and 
\begin{align}\label{time continuity}
\lim_{h\to 0}\E|I_n(\tilde f_n(\cdot,t+h,x))-I_n(\tilde f_n(\cdot,t,x))|^p= 0.
\end{align}

\end{lemma}
\begin{proof}
For $x,z\in \R^d$, by the hypercontractivity used in \eqref{hyper}, we can obtain that
\begin{align*}
&\|I_n(\tilde f_n(\cdot,t,x+z))-I_n(\tilde f_n(\cdot,t,x))\|^2_{L^p}\\&\leq (p-1)^n\|I_n(\tilde f_n(\cdot,t,x+z))-I_n(\tilde f_n(\cdot,t,x))\|^2_{L^2}\\&=(p-1)^n n!\|\tilde f_n(\cdot,t,x+z)-\tilde f_n(\cdot,t,x)\|_{\cH^{\otimes n}}^2.
\end{align*}
Note that
\begin{align}\label{spatial-c}
&n!\|\tilde f_n(\cdot,t,x+z)-\tilde f_n(\cdot,t,x)\|_{\cH^{\otimes n}}^2\nonumber\\&=n!\int_{[0,t]^{2n}}\prod_{j=1}^n|t_j-s_j|^{-\beta_0}\int_{\R^{nd}}\cF[\tilde f_n(\mathbf{t},\cdot,t,x+z)-\tilde f_n(\mathbf{t},\cdot,t,x)](\xi_1,\cdots,\xi_n)\nonumber\\&\hspace{2em}\times\overline{\cF[\tilde f_n(\mathbf{s},\cdot,t,x+z)-\tilde f_n(\mathbf{s},\cdot,t,x)](\xi_1,\cdots,\xi_n)}\mu(d\xi_1)\cdots\mu(d\xi_n)d\mathbf{t}d\mathbf{s}\nonumber\\&\leq n!\int_{[0,t]^{2n}}\int_{\R^{nd}}\big|\cF[\tilde f_n(\mathbf{t},\cdot,t,x+z)-\tilde f_n(\mathbf{t},\cdot,t,x)](\xi_1,\cdots,\xi_n)\big|^2\mu(d\xi_1)\cdots\mu(d\xi_n)\nonumber\\&\hspace{6em}\times \prod_{j=1}^n|t_j-s_j|^{-\beta_0}d\mathbf{t}d\mathbf{s}\\&\leq n!A_t^n \int_{[0,t]^n} \int_{\R^{nd}}\big|\cF[\tilde f_n(\mathbf{t},\cdot,t,x+z)-\tilde f_n(\mathbf{t},\cdot,t,x)](\xi_1,\cdots,\xi_n)\big|^2\mu(d\xi_1)\cdots\mu(d\xi_n) d\mathbf{t}\nonumber\\&=\frac{1}{n!}A_t^n \int_{[0,t]^n}\int_{\R^{nd}}|\cF p_{t_{\tau(2)}-t_{\tau(1)}}(\xi_{\tau(1)})|^2\cdots|\cF p_{t_{\tau(n)}-t_{\tau(n-1)}}(\xi_{\tau(1)}+\cdots+\xi_{\tau(n-1)})|^2\nonumber\\&\hspace{6em}\times |\cF p_{t-t_{\tau(n)}}(\xi_{\tau(1)}+\cdots+\xi_{\tau(n)})|^2|1-e^{-\iota(\xi_1+\cdots+\xi_n)\cdot z}|^2\mu(d\xi_1)\cdots\mu(d\xi_n)d\mathbf{t}\nonumber\\&=A_t^n \int_{\{0<t_1<\cdots<t_n<t\}}\int_{\R^{nd}}|\cF p_{t_2-t_1}(\xi_1)|^2\cdots|\cF p_{t_n-t_{n-1}}(\xi_{1}+\cdots+\xi_{n-1})|^2\nonumber\\&\hspace{10em}\times |\cF p_{t-t_{n}}(\xi_{1}+\cdots+\xi_{n})|^2|1-e^{-\iota(\xi_1+\cdots+\xi_n)\cdot z}|^2\mu(d\xi_1)\cdots\mu(d\xi_n)d\mathbf{t},\nonumber
\end{align}
where $\tau \in S_n$ is the permutation such that $0<t_{\tau(1)}<\ldots<t_{\tau(n)}<t$ and $A_t$ is given in \eqref{eq-integrate-s-t}.  By the maximum principle in Lemma \ref{Lem-max} and Assumption~\ref{H},  the right-hand side of \eqref{spatial-c} is bounded by
 \begin{align*}
&A_t^n \int_{\{0<t_1<\cdots<t_n<t\}} \int_{\R^{nd}}\prod_{j=1}^{n}|\cF p_{t_{j+1}-t_j}(\xi_j)|^2\mu(d\xi_1)\cdots\mu(d\xi_n)d\mathbf{t}\\&=A_t^n \int_{\{0<t_1<\cdots<t_n<t\}} \int_{\R^{2nd}}\prod_{j=1}^{n}p_{t_{j+1}-t_j}(y_j)p_{t_{j+1}-t_j}(z_j)\gamma(y_j-z_j)d\mathbf y d\mathbf zd\mathbf{t}\\&\leq c_0^{2n}A_t^n \int_{\{0<t_1<\cdots<t_n<t\}} \int_{\R^{nd}}\prod_{j=1}^{n}|\cF P_{t_{j+1}-t_j}(\xi_j)|^2\mu(d\xi_1)\cdots\mu(d\xi_n)d\mathbf{t}\\& \leq c_0^{2n}c_1^{2n} A_t^n \int_{\{0<t_1<\cdots<t_n<t\}} \int_{\R^{nd}}\prod_{j=1}^{n}\exp\big(-c_2(t_{j+1}-t_{j})\Psi(\xi_{j})\big)\mu(d\xi_1)\cdots\mu(d\xi_n)d\mathbf{t}<\infty,
\end{align*}
with $t_{n+1}=t$, where the last inequality follows  from  \cite[Proposition 3.5]{Song17} with $\beta_0=0$. 
Noting that on the right-hand side of \eqref{spatial-c},  $|1-e^{-\iota(\xi_1+\cdots+\xi_n)\cdot z}|\to 0$ as $|z|\to 0$,  the dominated convergence theorem yields the desired  \eqref{spatial continuity}.

Now we prove \eqref{time continuity}.
For $h\in(0,1)$, by hypercontractivity we get
\begin{align*}
&\|I_n(\tilde f_n(\cdot,t+h,x))-I_n(\tilde f_n(\cdot,t,x))\|_{L^p}^2\\&\leq (p-1)^n\|I_n(\tilde f_n(\cdot,t+h,x))-I_n(\tilde f_n(\cdot,t,x))\|_{L^2}^2\\&=(p-1)^n n!\|\tilde f_n(\cdot,t+h,x)-\tilde f_n(\cdot,t,x)\|_{\cH^{\otimes n}}^2\\&\leq 2(p-1)^n(A_n(t,h)+B_n(t,h)),
\end{align*}
where 
\begin{align}
&A_n(t,,h)=n!\|\tilde f_n(\cdot,t+h,x) {\bf 1}_{[0,t]^n}-\tilde f_n(\cdot,t,x)\|_{\cH^{\otimes n}}^2,\label{A_n}\\
&B_n(t,,h)=n!\|\tilde f_n(\cdot,t+h,x) {\bf 1}_{[0,t+h]^n\setminus[0,t]^n}\|_{\cH^{\otimes n}}^2.\label{B_n}
\end{align}
Firstly, we estimate $A_n(t,h)$,\begin{align*}
A_n(t,h)
&=n!\int_{[0,t]^{2n}}\prod_{j=1}^n|t_j-s_j|^{-\beta_0}\int_{\R^{nd}}\cF[\tilde f_n(\mathbf{t},\cdot,t+h,x)-\tilde f_n(\mathbf{t},\cdot,t,x)](\xi_1,\cdots,\xi_n)\\&\hspace{2em}\times\overline{\cF[\tilde f_n(\mathbf{s},\cdot,t+h,x)-\tilde f_n(\mathbf{s},\cdot,t,x)](\xi_1,\cdots,\xi_n)}\mu(d\xi_1)\cdots\mu(d\xi_n)d\mathbf{t}d\mathbf{s}\\&\leq n!\int_{[0,t]^{2n}}\int_{\R^{nd}}\big|\cF[\tilde f_n(\mathbf{t},\cdot,t+h,x)-\tilde f_n(\mathbf{t},\cdot,t,x)](\xi_1,\cdots,\xi_n)\big|^2\mu(d\xi_1)\cdots\mu(d\xi_n)\\&\hspace{6em}\times \prod_{j=1}^n|t_j-s_j|^{-\beta_0}d\mathbf{t}d\mathbf{s}\\&\leq n!A_t^n \int_{[0,t]^n} \int_{\R^{nd}}\big|\cF[\tilde f_n(\mathbf{t},\cdot,t+h,x)-\tilde f_n(\mathbf{t},\cdot,t,x)](\xi_1,\cdots,\xi_n)\big|^2\mu(d\xi_1)\cdots\mu(d\xi_n) d\mathbf{t}\\
&=\frac{1}{n!}A_t^n \int_{[0,t]^n}\int_{\R^{nd}}|\cF p_{t_{\tau(2)}-t_{\tau(1)}}(\xi_{\tau(1)})|^2\cdots|\cF p_{t_{\tau(n)}-t_{\tau(n-1)}}(\xi_{\tau(1)}+\cdots+\xi_{\tau(n-1)})|^2\\&\hspace{6em}\times |\cF[ p_{t+h-t_{\tau(n)}}-p_{t-t_{\tau(n)}}](\xi_{\tau(1)}+\cdots+\xi_{\tau(n)})|^2\mu(d\xi_1)\cdots\mu(d\xi_n)d\mathbf{t}\\&=A_t^n \int_{\{0<t_1<\cdots<t_n<t\}}\int_{\R^{nd}}|\cF p_{t_2-t_1}(\xi_1)|^2\cdots|\cF p_{t_n-t_{n-1}}(\xi_{1}+\cdots+\xi_{n-1})|^2\\&\hspace{10em}\times |\cF[ p_{t+h-t_n}-p_{t-t_{n}}](\xi_{1}+\cdots+\xi_{n})|^2\mu(d\xi_1)\cdots\mu(d\xi_n)d\mathbf{t}.
\end{align*}
By Lemma \ref{Lem-max} and  Assumption \ref{H}, we get
\begin{align*}
A_n(t,h)&\leq A_t^n \int_{\{0<t_1<\cdots<t_n<t\}} \prod_{j=1}^{n-1}\int_{\R^d}|\cF p_{t_{j+1}-t_j}(\xi_j)|^2\mu(d\xi_j)\\&\hspace{4em}\times \int_{\R^d}(|\cF p_{t+h-t_n}(\xi_n)|^2+|\cF p_{t-t_{n}}(\xi_{n})|^2)\mu(d\xi_n)d\mathbf{t}\\
&\leq 2 A_t^n \int_{\{0<t_1<\cdots<t_n<t\}} \int_{\R^{nd}}\prod_{j=1}^{n}\exp\big(-c_2(t_{j+1}-t_{j})\Psi(\xi_{j})\big)\mu(d\xi_1)\cdots\mu(d\xi_n)d\mathbf{t}<\infty,
\end{align*}
with $t_{n+1}=t$.  By the continuity of the function $t\to \cF p_t(\xi)$ and the dominated convergence theorem, we can deduce that $A_n(t,h)\to 0$ as $h\to 0$.

As for the term $B_n(t,h)$,
\begin{align*}
B_n(t,h)&=n!\int_{[0,t+h]^{2n}}\prod_{j=1}^n|t_j-s_j|^{-\beta_0}\int_{\R^{nd}}\cF \tilde f_n(\mathbf{t},\cdot,t+h,x)(\xi_1,\cdots,\xi_n){\bf 1}_{[0,t+h]^n\setminus[0,t]^n}(\mathbf{t})\\&\hspace{2em}\times\overline{\cF \tilde f_n(\mathbf{s},\cdot,t+h,x)(\xi_1,\cdots,\xi_n)}{\bf 1}_{[0,t+h]^n\setminus[0,t]^n}(\mathbf{s})\mu(d\xi_1)\cdots\mu(d\xi_n)d\mathbf{t}d\mathbf{s}\\&\leq n!\int_{[0,t+h]^{2n}}\int_{\R^{nd}}|\cF \tilde f_n(\mathbf{t},\cdot,t+h,x)(\xi_1,\cdots,\xi_n)|^2{\bf 1}_{[0,t+h]^n\setminus[0,t]^n}(\mathbf{t})\\&\hspace{3em}\times \prod_{j=1}^n|t_j-s_j|^{-\beta_0}\mu(d\xi_1)\cdots\mu(d\xi_n)d\mathbf{t}d\mathbf{s}\\&\leq n!A_{t+h}^n \int_{[0,t+h]^{n}}\int_{\R^{nd}}|\cF \tilde f_n(\mathbf{t},\cdot,t+h,x)](\xi_1,\cdots,\xi_n)|^2{\bf 1}_{[0,t+h]^n\setminus[0,t]^n}(\mathbf{t})\mu(d\xi_1)\cdots\mu(d\xi_n)d\mathbf{t}\\&=\frac{1}{n!}A_{t+h}^n \int_{[0,t+h]^n}\int_{\R^{nd}}|\cF p_{t_{\tau(2)}-t_{\tau(1)}}(\xi_{\tau(1)})|^2\cdots|\cF p_{t_{\tau(n)}-t_{\tau(n-1)}}(\xi_{\tau(1)}+\cdots+\xi_{\tau(n-1)})|^2\\&\hspace{6em}\times |\cF p_{t+h-t_{\tau(n)}}(\xi_{\tau(1)}+\cdots+\xi_{\tau(n)})|^2\mu(d\xi_1)\cdots\mu(d\xi_n){\bf 1}_{[0,t+h]^n\setminus[0,t]^n}(\mathbf{t})d\mathbf{t}\\&=A_{t+h}^n \int_{\{0<t_1<\cdots<t_n<t+h\}}\int_{\R^{nd}}|\cF p_{t_2-t_1}(\xi_1)|^2\cdots|\cF p_{t_n-t_{n-1}}(\xi_{1}+\cdots+\xi_{n-1})|^2\\&\hspace{10em}\times |\cF p_{t+h-t_n}(\xi_{1}+\cdots+\xi_{n})|^2\mu(d\xi_1)\cdots\mu(d\xi_n){\bf 1}_{[0,t+h]^n\setminus[0,t]^n}(\mathbf{t})d\mathbf{t}\\&\leq A_{t+h}^n\int_{\{0<t_1<\cdots<t_n<t+h\}} \int_{\R^{nd}}\prod_{j=1}^{n}\exp\big(-c_2(t_{j+1}-t_{j})\Psi(\xi_{j})\big)\mu(d\xi_1)\cdots\mu(d\xi_n)\\&\hspace{4em}\times {\bf 1}_{[0,t+h]^n\setminus[0,t]^n}(\mathbf{t})d\mathbf{t}<\infty,
\end{align*}
with $t_{n+1}=t+h$. By the dominated convergence theorem and the fact that  ${\bf 1}_{[0,t+h]^n\setminus[0,t]^n}(\mathbf{t})\to 0$ as $h\to 0$,  we have $B_n(t,h)\to 0$ as $h\to 0$. Then \eqref{time continuity} follows.
\end{proof}
The following result follows from direct calculation and can be found, e.g., in \cite[Lemma 3.10]{Song17}. 
\begin{lemma}\label{useful equality}
Suppose $\alpha_i<1$ for $i=1,\cdots,n$
 and let $\alpha=\sum_{i=1}^n\alpha_i$. Then,
 \begin{align*}
 \int_{\{0<t_1<\cdots<t_n<t_{n+1}=t\}}\prod_{i=1}^n (t_{i+1}-t_i)^{-\alpha_i}d\mathbf{t}=\frac{\prod_{i=1}^n\Gamma(1-\alpha_i)}{\Gamma(n-\alpha+1)}t^{n-\alpha}.
 \end{align*}
 
 \end{lemma}

 Now, we are ready to obtain the H\"{o}lder continuity of the Skorohod solution under some additional conditions.
\begin{theorem} \label{Thm-Skorohod-Holder-space}
Assume \ref{H} and condition \eqref{dalang-SKo}. Furthermore,  suppose that there exist $\alpha_1, \alpha_2, \beta \in (0,1)$ and $C>0$  such that for all $T>0$ and $z \in \bR^d$, the following conditions  hold:
\begin{align} \label{cond-space-Holder}
    \sup_{t\in[0,T]}\sup_{\eta \in \bR^d} \int_{\bR^d} \left| 1 - e^{-\iota (\xi+\eta)\cdot z} \right|^2 \left| \cF p_t(\xi+\eta) \right|^2 \mu(d\xi) 
    \le C |z|^{2\alpha_1};
\end{align}
\begin{align}\label{cond-time-Holder}
\sup_{t\in[0,T\wedge(T-h)]}\sup_{\eta\in \R^d}\int_{\R^d}|\cF p_{t+h}(\xi+\eta)-\cF p_{t}(\xi+\eta)|^2\mu(d\xi)\leq C|h|^{2\alpha_2};
\end{align}
\begin{align}\label{cond-time-spatial-Holder}
\sup_{\eta\in \R^d}\int_{\R^d}|\cF p_{t}(\xi+\eta)|^2\mu(d\xi)\leq Ct^{-\beta}, \text{ for } t\in[0,T].
\end{align}
Then, the Skorohod solution $u(t,x)$ has a version that is $\varkappa_1$-H\"{o}lder continuous in $x$ with $\varkappa_1<\alpha_1$ and $\varkappa_2$-H\"{o}lder continuous in $t$ with $\varkappa_2<(\alpha_2\wedge \frac {1-\beta}{2})$ on any compact set of $[0,\infty) \times \bR^d$.   
\end{theorem}
\begin{remark}
 Under Assumption \ref{H}, applying Lemma \ref{Lem-max} and \cite[Lemma 3.9]{Song17}, we can deduce that a sufficient condition for \eqref{cond-time-spatial-Holder}  is 
\begin{equation*}\label{e:cond-BQS}
  \int_{\R^d} \left(\frac1{1+\Psi(\xi)} \right)^{\beta}\mu(d\xi)<\infty,
 \end{equation*}
 which is the condition (T2) in \cite{Song17} with $\beta=1-\alpha_2$ and the condition (1.6) in \cite{BQS2019} with $\beta=\eta$. Condition \eqref{cond-space-Holder} is a consequence of condition (S2) in \cite{Song17}. Furthermore, 
 conditions \eqref{cond-space-Holder} and \eqref{cond-time-Holder} imply the estimates (3.1) and (3.2) in \cite[Proposition 3.1]{BQS2019}, and thus  Theorem \ref{Thm-Skorohod-Holder-space} is consistent with \cite[Theorem 3.2]{BQS2019} which improved the similar result obtained in \cite[Theorem 5.9]{Song17}.
\end{remark}

\begin{proof}
Let $t\in [0,T]$ and $x, z\in \R^d$. By Minkowski's inequality and \eqref{hyper}, we have
\begin{align}\label{spatial-hc}
&\|u(t,x+z)-u(t,x)\|_{L^p}\leq \sum_{n\geq 0}(p-1)^{\frac{n}{2}}\|I_n(\tilde f_n(\cdot,t,x+z))-I_n(\tilde f_n(\cdot,t,x))\|_{L^2}\nonumber\\&=\sum_{n\geq 0}(p-1)^{\frac{n}{2}} (n!\|\tilde f_n(\cdot,t,x+z)-\tilde f_n(\cdot,t,x)\|_{\cH^{\otimes n}}^2)^{\frac 12}.
\end{align}
Note that by
\eqref{spatial-c}, \eqref{cond-space-Holder}, \eqref{cond-time-spatial-Holder} and Lemma \ref{useful equality},
\begin{align*}
&n!\|\tilde f_n(\cdot,t,x+z)-\tilde f_n(\cdot,t,x)\|_{\cH^{\otimes n}}^2 \\&\leq C^n A_t^n \int_{\{0<t_1<\cdots<t_n<t\}}\prod_{j=1}^{n-1}(t_{j+1}-t_j)^{-\beta}|z|^{2\alpha_1}dt_1\cdots dt_n\\&\leq C^n A_t^n |z|^{2\alpha_1}\frac{1}{((n-1)!)^{1-\beta}}t^{(n-1)(1-\beta)+1}.
\end{align*}
The spatial H\"{o}lder continuity of $u(t,x)$ follows from the  Kolmogorov's continuity theorem.

Assume $h>0$ (the case of $h<0$ is similar). By Minkowski's inequality and \eqref{hyper},
\begin{align}\label{time-hc}
&\|u(t+h,x)-u(t,x)\|_{L^p}\leq \sum_{n\geq 0}(p-1)^{\frac{n}{2}}\|I_n(\tilde f_n(\cdot,t+h,x))-I_n(\tilde f_n(\cdot,t,x))\|_{L^2}\nonumber\\&=\sum_{n\geq 0}(p-1)^{\frac{n}{2}} (n!\|\tilde f_n(\cdot,t+h,x)-\tilde f_n(\cdot,t,x)\|_{\cH^{\otimes n}}^2)^{\frac 12}\\&\leq \sum_{n\geq 0}(p-1)^{\frac{n}{2}}2^{\frac 12}(A_n(t,h)+B_n(t,h))^{\frac 12},\nonumber
\end{align}
where $A_n(t,h)$ and $B_n(t,h)$ are given by \eqref{A_n} and \eqref{B_n}, respectively. By   \eqref{cond-time-Holder}, \eqref{cond-time-spatial-Holder} and Lemma \ref{useful equality}, we get
\begin{align}\label{A-hc}
A_n(t,h)&\leq C^n A_t^n\int_{\{0<t_1<\cdots<t_n<t\}}\prod_{j=1}^{n-1}(t_{j+1}-t_j)^{-\beta}h^{2\alpha_2}dt_1\cdots dt_n\nonumber\\&\leq C^n A_t^n h^{2\alpha_2}\frac{1}{((n-1)!)^{1-\beta}}t^{(n-1)(1-\beta)+1},
\end{align}
and 
\begin{align}\label{B-hc}
B_n(t,h)&\leq C^n A_{t+h}^n \int_{\{0<t_1<\cdots<t_n<t+h\}}\prod_{j=1}^{n-1}(t_{j+1}-t_j)^{-\beta}(t+h-t_n)^{-\beta}{\bf 1}_{[0,t+h]^n\setminus[0,t]^n}(\mathbf{t})dt_1\cdots dt_n\nonumber\\&=C^n A_{t+h}^n\frac{\Gamma(1-\beta)^{n}}{\Gamma((n-1)(1-\beta)+1)}\int_t^{t+h}t_n^{(n-1)(1-\beta)}(t+h-t_n)^{-\beta}dt_n\nonumber\\&=C^n A_{t+h}^n\frac{\Gamma(1-\beta)^{n}}{\Gamma((n-1)(1-\beta)+1)}\int_0^{h}(t+h-t_n)^{(n-1)(1-\beta)}t_n^{-\beta}dt_n\nonumber\\&\leq C^n A_{t+h}^n\frac{\Gamma(1-\beta)^{n}}{\Gamma((n-1)(1-\beta)+1)} T^{(n-1)(1-\beta)} \frac{1}{1-\beta}h^{1-\beta}\\&\leq C^n A_{t+h}^n\frac{1}{((n-1)!)^{1-\beta}} T^{(n-1)(1-\beta)}h^{1-\beta}.\nonumber
\end{align}
Then the H\"{o}lder continuity of $u(t,x)$ in time variable follows from \eqref{time-hc}, \eqref{A-hc}, \eqref{B-hc} and Kolmogorov's continuity theorem.
\end{proof}

\subsection{Regularity of the law}\label{sec:law2}

In this subsection, we study the regularity of the law of the Skorohod solution to \eqref{spde}. 

In light of Theorem \ref{Thm-FK-Skorohod}, using a similar argument as in the proof of Theorem \ref{thm:existence-densitiy} we can prove the existence of probability density under the stronger Dalang's condition \eqref{dalang-stra}. 

\begin{theorem}\label{thm:existence-densitiy'}
Assume Assumption \ref{H} and condition \eqref{dalang-stra}. Furthermore, suppose that $u_0(x)>0$ almost everywhere and \eqref{e:existence-density} is satisfied.   Then, the law of $u(t,x)$ given by \eqref{e:FK1} has a probability density. 
\end{theorem}

Still assuming \eqref{dalang-stra}, we can prove the smoothness of the density with the help of Feynman--Kac formula given in Theorem \ref{Thm-FK-Skorohod} under some proper conditions. 
\begin{theorem} \label{Thm-density-Skorohod}
Assume  the same conditions as in Theorem \ref{thm:smooth-density1}. 
Then the probability density of the mild Skorohod solution $u(t,x)$ to \eqref{spde} exists and is smooth.
\end{theorem}

\begin{proof}

The proof parallels that of Theorem \ref{thm:smooth-density1}, and a concise outline is presented below. By Theorem \ref{Thm-FK-Skorohod}, we have
\begin{align*}
    u(t,x) = \bE_X \left[ u_0(X_t^x) \exp \left( V_{t,x} - U_{t,x} \right) \right],
\end{align*}
where $V_{t,x} = W(\boldsymbol{\delta}(X_{t-\cdot}^x - \cdot))$ is given in \eqref{V_{t,x}} and 
\begin{align*}
    U_{t,x} = \dfrac{1}{2} \int_0^t \int_0^t |r-s|^{-\beta_0} \gamma(X_r^x-X_s^x) drds.
\end{align*}
The Malliavin derivative of $u(t,x)$ is given by
\begin{align*}
    D_{s,y} u(t,x)
    = &\bE_X \left[ u_0(X_t^x) \exp \left( V_{t,x} - U_{t,x} \right) \boldsymbol{\delta}(X_{t-s}^x-y) \right],
\end{align*}
and thus
\begin{align*}
    \left\| D u(t,x) \right\|_{\cH}^2
    = \bE_{X, \tilde X} \left[ u_0(X_t^x) u_0(\tilde X_t^x) \exp \left( V_{t,x} - U_{t,x} + \tilde V_{t,x} - \tilde U_{t,x} \right) \left\langle \boldsymbol{\delta}(X_{t-\cdot}^x-\cdot), \boldsymbol{\delta}(\tilde X_{t-\cdot}^x-\cdot) \right\rangle_{\cH} \right],
\end{align*}
where $\tilde X$ is an independent copy of $X$, and $\tilde V_{t,x}, \tilde U_{t,x}$ are the corresponding quantities of $V_{t,x}, U_{t,x}$ with $X$ replaced by $\tilde X$. Then by Jensen's inequality and H\"{o}lder inequality, we have
\begin{align} \label{eq-negative moment}
    & \E\left[\left\| D u(t,x) \right\|_{\cH}^{-2p} \right]\nonumber \\
   & \le \bE \left[ \left( u_0(X_t^x) u_0(\tilde X_t^x) \right)^{-p} \exp \left( -p \left( V_{t,x} - U_{t,x} + \tilde V_{t,x} - \tilde U_{t,x} \right) \right) \left| \left\langle \boldsymbol{\delta}(X_{t-\cdot}^x-\cdot), \boldsymbol{\delta}(\tilde X_{t-\cdot}^x-\cdot) \right\rangle_{\cH} \right|^{-p} \right] \nonumber \\
    &\le \left( \bE \left[ \left( u_0(X_t^x) u_0(\tilde X_t^x) \right)^{-pp_1} \right] \right)^{1/p_1}
    \left( \bE \left[ \exp \left( -pp_2 \left( V_{t,x} + \tilde V_{t,x} \right) \right) \right] \right)^{1/p_2} \nonumber \\
    &\hspace{1em} \times \left( \bE \left[ \exp \left( pp_3 \left( U_{t,x} + \tilde U_{t,x} \right) \right) \right] \right)^{1/p_3}
    \left( \bE \left[ \left| \left\langle \boldsymbol{\delta}(X_{t-\cdot}^x-\cdot), \boldsymbol{\delta}(\tilde X_{t-\cdot}^x-\cdot) \right\rangle_{\cH} \right|^{-pp_4} \right] \right)^{1/p_4},
\end{align}
for any  positive numbers $p_1,p_2,p_3,p_4$ satisfying $\frac{1}{p_1}+\frac{1}{p_2}+\frac{1}{p_2}+\frac{1}{p_4}=1$.

Utilizing the condition  \eqref{e:con-smooth}, we have
\begin{align} \label{eq-negative moment-1}
    \left( \bE \left[ \left( u_0(X_t^x) u_0(\tilde X_t^x) \right)^{-pp_1} \right] \right)^{1/p_1}
    = \left( \bE \left[ \left( u_0(X_t^x) \right)^{-pp_1} \right] \right)^{2/p_1} < \infty.
\end{align}
Since $X$ and $\tilde X$ have identical distribution, by Cauchy-Schwarz inequality and Theorem \ref{thm:exp}, we have
\begin{align} \label{eq-negative moment-2}
    \left( \bE \left[ \exp \left( -pp_2 \left( V_{t,x} + \tilde V_{t,x} \right) \right) \right] \right)^{1/p_2}
    \le \left( \bE \left[ \exp \Big( -2pp_2 V_{t,x} \Big) \right] \right)^{1/p_2} < \infty.
\end{align}
By independence and Theorem \ref{exp'}, we have
\begin{align} \label{eq-negative moment-3}
    \left( \bE \left[ \exp \left( pp_3 \left( U_{t,x} + \tilde U_{t,x} \right) \right) \right] \right)^{1/p_3}
    = \left( \bE \left[ \exp \Big( pp_3 U_{t,x} \Big) \right] \right)^{2/p_3} < \infty.
\end{align}
Condition \eqref{e:con-smooth} together with the fact  $|r-s|^{-\beta_0} \ge t^{-\beta_0}$ yields that
\begin{align} \label{eq-negative moment-4}
    & \left( \bE \left[ \left| \left\langle \boldsymbol{\delta}(X_{t-\cdot}^x-\cdot), \boldsymbol{\delta}(\tilde X_{t-\cdot}^x-\cdot) \right\rangle_{\cH} \right|^{-pp_4} \right] \right)^{1/p_4} \nonumber \\
    =& \bigg( \bE \bigg[ \left| \int_0^t\int_0^t |r-s|^{-\beta_0} \gamma (X_r^x - \tilde X_s^x) drds \right|^{-pp_4} \bigg] \bigg)^{1/p_4} \nonumber \\
    \le& t^{\beta_0p} \bigg( \bE \bigg[ \left| \int_0^t\int_0^t \gamma (X_r^x - \tilde X_s^x) drds \right|^{-pp_4} \bigg] \bigg)^{1/p_4}
    < \infty.
\end{align}
Substituting \eqref{eq-negative moment-1}, \eqref{eq-negative moment-2}, \eqref{eq-negative moment-3} and \eqref{eq-negative moment-4} to \eqref{eq-negative moment}, we have for all $p>0$,
\begin{align}\label{e:Du-p}
    \E\left[\left\| D u(t,x) \right\|_{\cH}^{-2p}\right] < \infty.
\end{align}
\end{proof}

In the proof of Theorem \ref{Thm-density-Skorohod}, in order to prove the desired inequality  \eqref{e:Du-p}, we employ the Feynman--Kac formula \eqref{e:FK-sko} of the Skorohod solution to \eqref{spde}, which is valid only under  the stronger Dalang's condition \eqref{dalang-stra}. Under  the weaker condition~\eqref{dalang-SKo},  a unique square integrable Skorohod solution exists (see   Theorem \ref{Thm-solution}) but  in general does not admit any Feynman--Kac type formula. In this situation, the finiteness of $\E \big[\|Du(t,x)\|_{\mathcal H}^{-2p}\big]$  is closely related to that of $\E \left[|u(t,x)|^{-2p}\right]$ which can be obtained by studying the small ball probability of $u(t,x)$ (see e.g.~\cite{MN08} which concerns semilinear heat equations with space-time white noise).

In the rest of the section, we shall prove the existence of the probability density for the Skorohod solution under the weaker Dalang's condition \eqref{dalang-SKo}. Our method is inspired by  \cite{NQ07,BQS2019EJP}. 


\begin{assumptionp}{(H')}\label{H'}
  We assume $p_t^{(x)}(y) \geq  
\tilde c_0 \tilde P_t(y-x)$ for a non-negative function $\tilde P_t(x)$ satisfying 
     \begin{equation}\label{e:P-bound-2}
     \cF{\tilde  P}_t(\xi) \geq  \tilde c_1 \exp\left(-\tilde c_2 t \tilde \Psi(\xi)\right),
     \end{equation}
    where $\tilde c_0,\tilde c_1,\tilde c_2$ are positive constants and $\tilde \Psi(\xi)=\tilde \Psi(|\xi|)$ is a non-negative measurable function. 
\end{assumptionp}

{
The following localization lemma is taken from \cite[Lemma 6.3]{BQS2019EJP}, which will be used to prove the regularity of the law of the Skorohod solution under the weaker condition~\eqref{dalang-SKo}. 
\begin{lemma}\label{lem:local-absolute-continuity}
Let $(\Gamma_m)_{m\geq 1}$ be a sequence of open sets in $\mathbb R$ such that
$0\notin \Gamma_m$ and $\Gamma_m\subset \Gamma_{m+1}$ for all $m\geq 1$.
Let
   $ \Gamma=\bigcup_{m\geq 1}\Gamma_m .$
Let $F\in \mathbb D^{2,p}$ for some $p>1$ be such that, for all $m\geq 1$,
\[
    \|DF\|_{\mathcal H}>0
    \quad \text{a.s. on } \{F\in \Gamma_m\}.
\]
Then the restriction to $\Gamma$ of the law of
$F\mathbf 1_{\{F\in \Gamma\}}$ is absolutely continuous with respect to
the Lebesgue measure on $\Gamma$.
\end{lemma}

}

\begin{theorem} \label{Thm-density-Skorohod-1}
Assume  \ref{H}, \ref{H'},  \eqref{dalang-SKo} and $\mu\big(\{\xi\in\R^d:\tilde\Psi(\xi)<+\infty\}\big)>0.$
 Furthermore, suppose  in Assumption \ref{H} $P_t(x) = t^{-d/2} \exp\left(-c|x|^2/t\right)$ for some positive constant $c$ and $\Psi(\xi)=|\xi|^2$ .

Then the restriction of the law of the random variable $u(t,x)1_{\{u(t,x) \not=0\}}$ to the set $\bR \setminus \{0\}$ is absolutely continuous with respect to the Lebesgue measure on $\bR \setminus \{0\}$.
\end{theorem}

\begin{remark}
    A typical example of the Feller process $X$ satisfying the conditions in Theorem~\ref{Thm-density-Skorohod-1} is the diffusion process governed by \eqref{x}. 
\end{remark}

 \begin{proof}
 {
Utilizing Lemma~\ref{lem:local-absolute-continuity} by letting $F=u(t,x)$ and $\Gamma_m=\{y\in \bR; |y|>\frac1m\}$, we only need to show $\|D u(t,x)\|_{\cH}>0$ a.s. on $\Omega_m:=\{|u(t,x)|>\frac{1}{m}\}$ for  $m\in\mathbb N$. 
}

Denote by $\tilde L^2({\R^d})$ the class of   functions $f$ satisfying 
\begin{align*}
\|f\|^2_{\tilde L^2(\R^d)}:=\int_{\R^d}|\mathcal F f(\xi)|^2  \mu(d\xi)<\infty.
\end{align*}
Then it suffices to prove 
 \begin{equation} \label{integral-positive}
\int_{0}^t \|D_{r,\cdot} u(t,x)\|_{\tilde L^2(\R^d)}^2  dr >0  \quad  \mbox{a.s. on} \ \Omega_m.
\end{equation}
Indeed, if $\|D u(t,x)\|_{\cH}=0$, by Lemma \ref{e:tilde} below we have
 $\|D_{r,\cdot} u(t,x)\|_{\tilde L^2(\R^d)}^2 = 0$ for almost all $r \in \bR_+$, which contradicts with $\eqref{integral-positive}$. The proof of \eqref{integral-positive} is split into seven steps.
 
{\bf Step 1.} We show that $u \in \mathbb D^{2,p}$ for $p>1$. It follows from 
\cite[Equation (3.1)]{BQS2019EJP} (see also \cite{Nualart}) that  \[u\in \mathbb D^{k,p}
\quad \text{if} \quad
\sum_{n\geq 1}
n^{k/2}(p-1)^{n/2}
\big(\mathbb E |I_n( f_n(\cdot,t,x)|^2\big)^{1/2}
<\infty.\] Hence, we only need to show
\begin{align*}
    \sum_{n\geq 1}n(p-1)^{\frac{n}{2}}\big(\E|I_n(\tilde f_n(\cdot,t,x))|^2\big)^{\frac12}
    =\sum_{n\geq 1}n(p-1)^{\frac{n}{2}}\left(n!\|\tilde f_n(\cdot,t,x)\|_{\mathcal H^{\otimes n}}^2\right)^\frac12<\infty.
\end{align*}
{
By \eqref{eq:tilde f}, we have
\begin{align*}
&\sum_{n\geq 1}n(p-1)^{\frac n2}
\left(n!\|\tilde f_n(\cdot,t,x)\|_{\mathcal H^{\otimes n}}^2\right)^{\frac12}\\
&\leq 
e^{c_2Mt}
\sum_{n\geq 1}n
\left[
(p-1)CA_t
\int_{\R^d}\frac1{M+\Psi(\xi)}\mu(d\xi)
\right]^{\frac n2}.
\end{align*}
We may choose $M$ sufficiently large such that
$\left[
(p-1)CA_t
\int_{\R^d}\tfrac1{M+\Psi(\xi)}\mu(d\xi)
\right]^{\frac12}<1.$
Therefore,
\[
\sum_{n\geq 1}n
\left[
(p-1)CA_t
\int_{\R^d}\frac1{M+\Psi(\xi)}\mu(d\xi)
\right]^{\frac n2}<\infty,
\]
and consequently,
\[
\sum_{n\geq 1}n(p-1)^{\frac n2}
\left(n!\|\tilde f_n(\cdot,t,x)\|_{\mathcal H^{\otimes n}}^2\right)^{\frac12}
<\infty.
\]
}
{\bf Step 2.}
 For $f\in\mathcal H$, let $\tilde \cH$ be the space  of functions satisfying 
\begin{align*}
\|f\|_{\tilde {\cH}}^2:=\int_{\R_+}\int_{\R^d} \left|\cF f(s,\xi)\right|^2\mu(d\xi)ds=\int_{\R_+}\int_{\R^{2d}} f(s,x) f(s,y) \gamma(x-y) dxdyds<\infty.
\end{align*}
{
Indeed, applying \cite[Lemma 5.1.1]{Nualart} with $H=1-\beta_0/2$ to
$\mathcal F f(\cdot,\xi)$ and using
$L^2(0,T)\subset L^{\frac{1}{1-\beta_0/2}}(0,T)$, we get
\[
\int_0^T\int_0^T
|\mathcal F f(t,\xi)|\,|\mathcal F f(s,\xi)|
|t-s|^{-\beta_0}\,dsdt
\leq C_T\int_0^T |\mathcal F f(s,\xi)|^2\,ds .
\]
After integrating with respect to $\mu(d\xi)$, this gives
$\|f\|_{\mathcal H}\leq C_T\|f\|_{\tilde{\mathcal H}}$ and hence $\tilde\cH \subset \mathcal H$.} In this step, we will show
that for any $T>0$, there exist some positive constants $C_T$ and $C_T'$ such that

 \begin{align}\label{estimate Du}
 \sup_{(t,x)\in [0,T]\times \bR^d}\sup_{r\in[0,t]} \bE \left[\| D_{r,\cdot} u(t,x)\|^2_{\tilde L^2(\R^d)}  \right]
    \le C_T,
\end{align}
and
\begin{align}\label{estimate DDu}
\sup_{(t,x) \in [0,T]\times\bR^d}\sup_{r\in[0,t]} \bE \left[  \big\| D^2_{\cdot,(r,\cdot)} u(t,x) \big\|_{\cH\times \tilde L^2(\R^d)}^2 \right]
    \le C_T'.
\end{align}

We approximate the Skorohod solution to \eqref{spde} using \eqref{approximate1} as in the proof of Theorem \ref{thm:FK}. Note that the solution $u^{\eps,\delta}(t,x)$ to \eqref{approximate1} is given by the Feynman--Kac formula \eqref{eq-FK-approx}. We will show that $D_{r,\cdot} u^{\eps,\delta}(t,x)$ converges in $ \tilde L^2(\bR^d)$ by using the same approach of Step 1 in the proof of Theorem~\ref{thm:FK}. For $r\in[0,t]$ and $z\in\R^d$,
\begin{align} \label{eq-Du}
    D_{r,z} u^{\eps,\delta}(t,x)
    =& \bE_X \left[ u_0(X_t^x) \exp \left( W(A_{t,x}^{\eps,\delta}) - \dfrac{1}{2} \left\| A_{t,x}^{\eps,\delta} \right\|_{\cH}^2 \right) D_{r,z} \left( W(A_{t,x}^{\eps,\delta}) - \dfrac{1}{2} \left\| A_{t,x}^{\eps,\delta} \right\|_{\cH}^2 \right) \right] \nonumber \\
    =& \bE_X \left[ u_0(X_t^x) \exp \left( W(A_{t,x}^{\eps,\delta}) - \dfrac{1}{2} \left\| A_{t,x}^{\eps,\delta} \right\|_{\cH}^2 \right) A_{t,x}^{\eps,\delta} (r,z) \right],
\end{align}
where we recall that $A_{t,x}^{\eps, \delta}$ is given in \eqref{A}. Thus, we have
\begin{align}\label{eq-Dutilde}
&\E\left[\la D_{r,\cdot}u^{\varepsilon,\delta}(t,x),D_{r,\cdot}u^{\varepsilon',\delta'}(t,x)\ra_{\tilde L^2(\R^d)}\right]=\E\left[\int_{\R^d}\cF D_{r,\cdot}u^{\varepsilon,\delta}(t,x)(\xi)\overline{\cF D_{r,\cdot}u^{\varepsilon',\delta'}(t,x)(\xi)}\mu(d\xi)\right] \nonumber\\&=\E\left[\int_{\R^{2d}}D_{r,y}u^{\varepsilon,\delta}(t,x)D_{r,z}u^{\varepsilon',\delta'}(t,x)\gamma(y-z)dydz\right]\nonumber\\ 
  & = \bE \bigg[  u_0(X_t^x) u_0(\tilde X_t^x) \exp \left( W \left(A_{t,x}^{\eps,\delta} + \tilde A_{t,x}^{\eps',\delta'} \right) - \dfrac{1}{2} \left\| A^{\epsilon,\delta}_{t,x} \right\|_{\cH}^2 - \dfrac{1}{2} \left\| \tilde A_{t,x}^{\eps',\delta'} \right\|_{\cH}^2 \right) \\
    & \hspace{2em} \times \int_0^t\int_0^t \int_{\R^{2d}}\varphi_\delta(t-s_1-r) \varphi_{\delta'}(t-s_2-r) q_\eps(X_{s_1}^x-y) q_{\eps'}(\tilde X_{s_2}^x-z) ds_1ds_2\gamma(y-z)dy dz \bigg].\nonumber
\end{align}
By applying Parseval-Plancherel identity, it follows that
\begin{align}\label{eq-apply-Parseval}
\int_{\R^{2d}} q_\eps(X_{s_1}^x-y) q_{\eps'}(\tilde X_{s_2}^x-z) \gamma(y-z)dy dz  \leq C\int_{\R^d}\cF q_{\varepsilon}(\xi)\cF q_{\varepsilon'}(\xi)\mu(d\xi)<\infty.
\end{align}
Combining the fact  
\begin{equation}\label{e:varphi-bound}
\int_0^t\int_0^t\varphi_\delta(r-s_1)\varphi_{\delta'}(r-s_2)ds_1ds_2\leq 1,
\end{equation}
 we have
\begin{align}\label{eq:DD}
&\E\left[\la D_{r,\cdot}u^{\varepsilon,\delta}(t,x),D_{r,\cdot}u^{\varepsilon',\delta'}(t,x)\ra_{\tilde L^2(\R^d)}\right]\nonumber\\ & \leq C\bE \bigg[  u_0(X_t^x) u_0(\tilde X_t^x) \exp \left( W \left(A_{t,x}^{\eps,\delta} + \tilde A_{t,x}^{\eps',\delta'} \right) - \dfrac{1}{2} \left\| A^{\epsilon,\delta}_{t,x} \right\|_{\cH}^2 - \dfrac{1}{2} \left\| \tilde A_{t,x}^{\eps',\delta'} \right\|_{\cH}^2 \right)\bigg].
\end{align}
In view of  \eqref{eq:DD}, the boundedness of $u_0$ and \eqref{ineq-exp-nondiag}  yield that 
\begin{align}\label{eq-Dutilde-finiteness}
\sup_{r\in[0,t]}\sup_{\varepsilon,\delta,\varepsilon',\delta'>0}\E\left[\la D_{r,\cdot}u^{\varepsilon,\delta}(t,x),D_{r,\cdot}u^{\varepsilon',\delta'}(t,x)\ra_{\tilde L^2(\R^d)}\right]\leq C_t,
\end{align}
where $C_t$ is a constant increasing in $t$. 
By \eqref{eq-Dutilde} and \eqref{eq-convergence-inner prod},  we have 
\begin{align*}
    & \lim_{\eps,\delta,\eps',\delta' \to 0} \E\bigg[\la D_{r,\cdot}u^{\varepsilon,\delta}(t,x),D_{r,\cdot}u^{\varepsilon',\delta'}(t,x)\ra_{\tilde L^2(\R^d)}\bigg] \nonumber \\
   & =  \bE \bigg[ u_0(X_t^x) u_0(\tilde X_t^x) \exp \left( \int_0^t\int_0^t |s_1-s_2|^{-\beta_0} \gamma(X_{s_1}^x-\tilde X_{s_2}^x) ds_1ds_2\right) \int_0^t \gamma \left( X_s^{x} - \tilde X_s^{x} \right)  ds \bigg],
\end{align*}
which is  finite due to \eqref{eq-Dutilde-finiteness}.
Thus, we get that
\begin{align*}
    & \lim_{\eps,\delta,\eps',\delta' \to 0} \bE\left[ \left\|D_{r,\cdot}u^{\varepsilon,\delta}(t,x)-D_{r,\cdot}u^{\varepsilon',\delta'}(t,x) \right\|_{\tilde L^2(\R^d)}^2\right]
    = 0,
\end{align*}
which also yields, noting that $C_t$ in \eqref{eq-Dutilde-finiteness} is finite,  \begin{align*}
\lim_{\eps,\delta,\eps',\delta' \to 0} \bE \left[\left\|Du^{\varepsilon,\delta}(t,x)-Du^{\varepsilon',\delta'}(t,x) \right\|_{\tilde {\mathcal H}}^2\right]
    = 0.
\end{align*}
Hence, $\{D u^{\eps,\delta}(t,x)\}_{\varepsilon,\delta>0}$ is a Cauchy sequence  in $\tilde {\cH}$ with the limit denoted by $\tilde D u(t,x)$. 
It has been proved that $\lim_{\eps,\delta\to 0}\E\left[\|D u^{\eps,\delta}(t,x)-D u(t,x)\|_\mathcal H^2\right]=0$ in the proof of Theorem \ref{thm:FK}. Since $\tilde \cH\subset \cH $, we have  $\tilde D u(t,x) = D u(t,x)$. It follows that \begin{align*}
\lim_{\eps,\delta\to 0} \bE \left[\left\|D_{r,\cdot}u^{\varepsilon,\delta}(t,x)-D_{r,\cdot}u(t,x) \right\|_{\tilde L^2{(\R^d)}}^2\right]
    = 0.
\end{align*}
Therefore, we have
\begin{align*}
    \sup_{(t,x) \in [0,T]\times\bR^d}\sup_{r\in[0,t]} \bE \left[\|D_{r,\cdot}u(t,x)\|_{\tilde L^2(\R^d)} ^2\right]
    &\le \sup_{(t,x) \in [0,T]\times\bR^d} \sup_{r\in[0,t]} \sup_{\eps,\delta>0} \bE \left[\|D_{r,\cdot}u^{\eps,\delta}(t,x)\|_{\tilde L^2(\R^d)} ^2\right]\\
    &\le \sup_{(t,x) \in [0,T]\times\bR^d}\sup_{r\in[0,t]}  C_t
    = C_T,
\end{align*}
which proves \eqref{estimate Du}.

Now we prove \eqref{estimate DDu}. The second order Malliavin derivative can be written as
\begin{align*} 
    D^2_{\cdot,(r,z)} u^{\eps,\delta}(t,x)
  &  = \bE_X \left[ u_0(X_t^x) \exp \left( W(A_{t,x}^{\eps,\delta}) - \dfrac{1}{2} \left\| A_{t,x}^{\eps,\delta} \right\|_{\cH}^2 \right)A_{t,x}^{\eps,\delta}(r,z) A_{t,x}^{\eps,\delta}(\cdot) \right],
    \end{align*}
and hence
\begin{align} \label{eq-4.86''}
  &\bE \left[  \Big< D^2_{\cdot,(r,\cdot)} u^{\eps,\delta}(t,x), D^2_{\cdot,(r,\cdot)} u^{\epsilon',\delta'}(t,x)\Big>_{\cH\times \tilde L^2(\R^d)} \right]
   \nonumber \\&=\bE \left[ \int_{\bR^{2d}} \la D^2_{\cdot,(r,y)} u^{\eps,\delta}(t,x), D^2_{\cdot,(r,z)} u^{\epsilon',\delta'}(t,x)\ra_\mathcal H \gamma(y-z) dydz \right] \nonumber \\
  & = \bE \bigg[  u_0(X_t^x) u_0(\tilde X_t^x) \exp \left( W \left(A_{t,x}^{\eps,\delta} + \tilde A_{t,x}^{\eps',\delta'} \right) - \dfrac{1}{2} \left\| A^{\epsilon,\delta}_{t,x} \right\|_{\cH}^2 - \dfrac{1}{2} \left\| \tilde A_{t,x}^{\eps,\delta} \right\|_{\cH}^2 \right) \la A^{\epsilon,\delta}_{t,x},\tilde A^{\epsilon',\delta'}_{t,x}\ra_\mathcal H  \nonumber\\
    &\qquad \qquad \times \int_0^t\int_0^t \int_{\bR^{2d}} \varphi_\delta(t-s_1-r) \varphi_{\delta'}(t-s_2-r) q_\eps(X_{s_1}^x-y) q_{\eps'}(\tilde X_{s_2}^x-z) ds_1ds_2\gamma(y-z) dydz \bigg] \nonumber \\
    & = \bE \bigg[ u_0(X_t^x) u_0(\tilde X_t^x) \exp \left( \left\langle A_{t,x}^{\eps,\delta}, \tilde A_{t,x}^{\eps',\delta'} \right\rangle_{\cH} \right)\left\langle A_{t,x}^{\eps,\delta}, \tilde A_{t,x}^{\eps',\delta'} \right\rangle_{\cH} \nonumber \\
    &\hspace{4em}\times \int_0^t\int_0^t\int_{\bR^{2d}} \varphi_\delta(r-s_1) \varphi_{\delta'}(r-s_2) q_\eps(X_{s_1}^x-y) q_{\eps'}(\tilde X_{s_2}^x-z) ds_1ds_2\gamma(y-z) dydz \bigg],   
\end{align} 
where we have applied 
the change of variable $r \to t-r$, and compute $\bE_W$ first in the second equality.
By H\"{o}lder's inequality and \eqref{eq-apply-Parseval}, we have
\begin{align} \label{eq-4.86}
    &\bE \left[  \Big< D^2_{\cdot,(r,\cdot)} u^{\eps,\delta}(t,x), D^2_{\cdot,(r,\cdot)} u^{\epsilon',\delta'}(t,x)\Big>_{\cH\times \tilde L^2(\R^d)} \right] \nonumber \\
    & \le \left( \bE \bigg[ \left| u_0(X_t^x) u_0(\tilde X_t^x) \right|^4 \bigg] \right)^{\frac{1}{4}} \bigg(\E \left[ \exp \left( 4 \left\langle A_{t,x}^{\eps,\delta}, \tilde A_{t,x}^{\eps',\delta'} \right\rangle_{\cH} \right) \left| \left\langle A_{t,x}^{\eps,\delta}, \tilde A_{t,x}^{\eps',\delta'} \right\rangle_{\cH} \right|^4 \right]\bigg)^{\frac{1}{4}}\nonumber \\
    & \hspace{2em}\times \left( \bE \bigg[ \left( \int_0^t\int_0^t\varphi_\delta(r-s_1) \varphi_{\delta'}(r-s_2) \int_{\R^d}\cF q_{\varepsilon}(\xi)\cF q_{\varepsilon'}(\xi)\mu(d\xi)  \right)^2 \bigg] \right)^{1/2}.
\end{align}

By \eqref{A} and the non-negativity of $\gamma$, we can deduce that $\left\langle A_{t,x}^{\eps,\delta}, \tilde A_{t,x}^{\eps',\delta'} \right\rangle_{\cH} \ge 0$. Using the inequality $e^{4x} \ge x^4$ for $x>0$, we have
\begin{align} \label{eq-exp(4x)x^4}
    \exp \left( 4 \left\langle A_{t,x}^{\eps,\delta}, \tilde A_{t,x}^{\eps',\delta'} \right\rangle_{\cH} \right) \left| \left\langle A_{t,x}^{\eps,\delta}, \tilde A_{t,x}^{\eps',\delta'} \right\rangle_{\cH} \right|^4
    \le \exp \left( 8 \left\langle A_{t,x}^{\eps,\delta}, \tilde A_{t,x}^{\eps',\delta'} \right\rangle_{\cH} \right).
\end{align}
Thus, by the boundedness of $u_0$, \eqref{eq-exp(4x)x^4}, \eqref{ineq-exp-nondiag},  \eqref{eq-apply-Parseval} and \eqref{e:varphi-bound}, one can deduce from \eqref{eq-4.86} that
\begin{align} \label{eq-4.89}
\sup_{r\in[0,t]}\sup_{\eps,\delta,\eps',\delta'>0} \bE \left[  \Big< D^2_{\cdot,(r,\cdot)} u^{\eps,\delta}(t,x), D^2_{\cdot,(r,\cdot)} u^{\epsilon',\delta'}(t,x)\Big>_{\cH\times \tilde L^2(\R^d)} \right]
    \le C_t'.
\end{align}
Moreover, by \eqref{eq-convergence-inner prod} and \eqref{eq-4.86''}, we have the following convergence 
\begin{align*}
&\lim\limits_{\varepsilon,\delta,\varepsilon',\delta'\to 0}\bE \left[  \Big< D^2_{\cdot,(r,\cdot)} u^{\eps,\delta}(t,x), D^2_{\cdot,(r,\cdot)} u^{\epsilon',\delta'}(t,x)\Big>_{\cH\times \tilde L^2(\R^d)} \right]\\&=\bE \bigg[ u_0(X_t^x) u_0(\tilde X_t^x) \exp \left( \int_0^t\int_0^t |s_1-s_2|^{-\beta_0} \gamma(X_{s_1}^x-\tilde X_{s_2}^x) drds\right) \int_0^t \gamma \left( X_s^{x} - \tilde X_s^{x} \right)  ds\\&\hspace{4em} \times\int_0^t\int_0^t |s_1-s_2|^{-\beta_0} \gamma(X_{s_1}^x-\tilde X_{s_2}^x) drds\bigg].
\end{align*}

Due to \eqref{eq-4.89}, the limit above is finite. Similar to the estimate of the first order derivative case, we can conclude that 
\begin{align*}
\lim_{\varepsilon,\delta \to 0}\E\Big[\|D^2_{\cdot,(r,\cdot)}u^{\varepsilon,\delta}(t,x)-D^2_{\cdot,(r,\cdot)}u(t,x)\|_{\cH\times \tilde L^2(\R^d)}\Big]=0.
\end{align*}
Therefore, we have
\begin{align*}
    &\sup_{(t,x) \in [0,T]\times\bR^d}\sup_{r\in[0,t]} \bE \left[\|D^2_{\cdot,(r,\cdot)}u(t,x)\|_{\cH\times \tilde L^2{(\R^d)}} ^2\right] 
   \\ &\le \sup_{(t,x) \in [0,T]\times\bR^d} \sup_{r\in[0,t]}\sup_{\eps,\delta>0} \bE\left[ \|D^2_{\cdot,(r,\cdot)}u^{\varepsilon,\delta}(t,x)\|_{\cH\times \tilde L^2(\R^d)} ^2 \right]\\
    &\le \sup_{(t,x) \in [0,T]\times\bR^d}\sup_{r\in[0,t]} C_t'
    = C_T',
\end{align*}
and \eqref{estimate DDu} is proved.

{\bf Step 3.} In this step, we shall estimate the following functions:
\begin{align}\label{eq-phi}
    \phi(t) = \int_0^t\int_0^t \int_{\bR^{2d}} p_{s_1}^{(x)}(y_1) p_{s_2}^{(x)}(y_2) |s_1-s_2|^{-\beta_0} \gamma(y_1-y_2) dy_1dy_2 ds_1ds_2,
\end{align}
and
 \begin{align*} 
    \psi(t) = \int_0^t \|p_{s}^{(x)}(\cdot)\|_{\tilde L^2(\R^d)}^2ds.
\end{align*}

By Fourier transform we have
\[\psi(t)=\int_0^t\int_{\R^{2d}}p_s^{(x)}(y)p_s^{(x)}(z)\gamma(y-z)dydzds=\int_0^t \int_{\R^d} |(\mathcal 
 Fp_s^{(x)})(\xi)|^2\mu(d\xi)ds.\]
Thus we have that $\psi(t)<\infty$ by Assumption \ref{H} and \eqref{dalang-SKo}. 
{
Moreover, by Cauchy's inequality,
\[
\int_{\R^{2d}}p_{s_1}^{(x)}(y_1)p_{s_2}^{(x)}(y_2)\gamma(y_1-y_2)dy_1dy_2
\le
\|p_{s_1}^{(x)}(\cdot)\|_{\tilde L^2(\R^d)}
\|p_{s_2}^{(x)}(\cdot)\|_{\tilde L^2(\R^d)} .
\]
Hence, using $2ab\le a^2+b^2$, we obtain
\begin{equation}\label{phipsi}
\phi(t)\le A_t\psi(t)<\infty,
\end{equation}
where we recall $A_t=2\int_0^t s^{-\beta_0}ds<\infty$.}

On the other hand, by Assumption \ref{H'}  we get
 \begin{align*}
\psi(t)&=\int_0^t\int_{\R^{2d}}p_s^{(x)}(y)p_s^{(x)}(z)\gamma(y-z)dydzds\\&\geq \tilde c_0^2\int_0^t\int_{\R^{2d}}\tilde P_s(y-x)\tilde P_s(z-x)\gamma(y-z)dydzds\\&=\tilde c_0^2\int_0^t\int_{\R^{d}}|\cF \tilde P_s(\xi)|^2 \mu(d\xi)ds\\&\geq \tilde c_0^2\tilde c_1^2\int_0^t\int_{\R^d}\exp\left(-2\tilde c_2 s \tilde \Psi(\xi)\right)\mu(d\xi)ds.
\end{align*}
Thus, we have 
\begin{align} \label{ineq-psi}
    \psi(t) \ge \tilde C_0 t,
\end{align}
where 
$\tilde C_0:=\tilde c_0^2\tilde c_1^2 \int_{\R^d}\exp\left(-2T \tilde c_2  \tilde \Psi(\xi)\right)\mu(d\xi)>0$ due to the assumption $\mu\big(\{\xi\in\R^d:\tilde\Psi(\xi)<+\infty\}\big)>0$.



  {\bf Step 4.} 
Denote
\begin{align}\label{e:Kr}
K^{(r)}((s,y),\cdot)=\mathbf {1}_{[0,t]}(s)p_{t-s}^{(x)}(y)D_{r,\cdot}u(s,y).
\end{align}
{
For fixed $r\in[0,t]$, using the
Cauchy--Schwarz inequality in $\tilde L^2(\R^d)$, we get
\begin{align*}
\E\|K^{(r)}\|^2_{|\mathcal H|\otimes \tilde L^2(\R^d)} &\leq
\int_0^t\int_0^t\int_{\R^{2d}}
p_{t-s_1}^{(x)}(y_1)p_{t-s_2}^{(x)}(y_2)
|s_1-s_2|^{-\beta_0}\gamma(y_1-y_2)  \\
&\qquad\qquad\times
\E\Big[
\|D_{r,\cdot}u(s_1,y_1)\|_{\tilde L^2(\R^d)}
\|D_{r,\cdot}u(s_2,y_2)\|_{\tilde L^2(\R^d)}
\Big]dy_1dy_2ds_1ds_2 .
\end{align*}
The estimate \eqref{estimate Du} and the Cauchy--Schwarz inequality yields
\[
\E\Big[
\|D_{r,\cdot}u(s_1,y_1)\|_{\tilde L^2(\R^d)}
\|D_{r,\cdot}u(s_2,y_2)\|_{\tilde L^2(\R^d)}
\Big]
\leq C_t .
\]
Hence, by the change of variables $s_i\to t-s_i$, $i=1,2$, we obtain
\begin{equation}\label{e:Kr'}
\E\|K^{(r)}\|^2_{|\mathcal H|\otimes \tilde L^2(\R^d)}
\leq C_t\phi(t)<\infty .
\end{equation}
Similarly, noting
\[
DK^{(r)}((s,y),\cdot)
=
\mathbf 1_{[0,t]}(s)p_{t-s}^{(x)}(y)
D^2_{\cdot,(r,\cdot)}u(s,y),
\]
we have
\begin{align*}
&\E\|DK^{(r)}\|^2_{|\mathcal H|\otimes \mathcal H\otimes \tilde L^2(\R^d)} \\
&\leq
\int_0^t\int_0^t\int_{\R^{2d}}
p_{t-s_1}^{(x)}(y_1)p_{t-s_2}^{(x)}(y_2)
|s_1-s_2|^{-\beta_0}\gamma(y_1-y_2) \\
&\quad \times
\E\Big[
\|D^2_{\cdot,(r,\cdot)}u(s_1,y_1)\|_{\mathcal H\otimes \tilde L^2(\R^d)}
\|D^2_{\cdot,(r,\cdot)}u(s_2,y_2)\|_{\mathcal H\otimes \tilde L^2(\R^d)}
\Big]\,dy_1dy_2ds_1ds_2 .
\end{align*}
By the Cauchy--Schwarz inequality and \eqref{estimate DDu}, we get
\[
\E\Big[
\|D^2_{\cdot,(r,\cdot)}u(s_1,y_1)\|_{\mathcal H\otimes \tilde L^2(\R^d)}
\|D^2_{\cdot,(r,\cdot)}u(s_2,y_2)\|_{\mathcal H\otimes \tilde L^2(\R^d)}
\Big]
\leq C_t' ,
\]
and consequently \begin{equation}\label{e:DKr}
\E\|DK^{(r)}\|^2_{|\mathcal H|\otimes \mathcal H\otimes \tilde L^2(\R^d)}
\leq C_t'\phi(t)<\infty .
\end{equation}
}
It follows  from \eqref{e:Kr'} and \eqref{e:DKr} that for $r\in[0,t]$,
$K^{(r)}\in\mathbb D^{1,2}(\mathcal H\otimes \tilde L^2{(\R^d}))$. Thus  $K^{(r)}\in \text{Dom}(\tilde{\boldsymbol{\delta}})$, where $\tilde{\boldsymbol{\delta}}$ is the $\tilde L^2{(\R^d})$-valued Skorohod integral (see \cite[Section 3]{NQ07} or \cite[Section 6]{Balan12} for the definition of Hilbert-space-valued Skorohod integral).  Therefore, $\tilde{\boldsymbol{\delta}}(K^{(r)})= \int_0^t \int_{\bR^d} p_{t-s}^{(x)}(y) D_{r,\cdot} u(s,y) W(d s, dy)$ is a well-defined $\tilde L^2{(\R^d})$-valued Skorohod integral.

By  \cite[Proposition 1.3.8]{Nualart}, we know that Malliavin derivative $Du$ of the solution satisfies the following equation in $L^2(\Omega;\cH)$: 
\begin{align} \label{eq-Du-H}
    D_{\cdot}u(t,x)
    =p_{t-\cdot}^{(x)}(\cdot) u(\cdot) + \int_0^t \int_{\bR^d} p_{t-s}^{(x)}(y) D_{\cdot} u(s,y) W(d s, dy),
\end{align}
which also holds in $L^2(\Omega;\tilde H)$ by Lemma \ref{e:tilde}. Thus, for almost all $r \in \bR_+$, the following equality holds in $L^2(\Omega;\tilde L^2(\bR^d))$:
\begin{align} \label{eq-Du-L2}
    D_{r,\cdot}u(t,x)
    =p_{t-r}^{(x)}(\cdot) u(r,\cdot) + \int_0^t \int_{\bR^d} p_{t-s}^{(x)}(y) D_{r,\cdot} u(s,y) W(d s, dy).
\end{align}
Then  the continuity in time of  $D_{r,\cdot}u(t,x)$ yields that \eqref{eq-Du-L2} holds in $L^2(\Omega;\tilde L^2(\bR^d))$  for all $r \in [0,t]$. 

Using the inequality 
$\|a+b\|^2\geq \frac 12\|a\|^2-\|b\|^2$,
 it follows that
\begin{equation}\label{eq:33}
    \begin{aligned} 
    \int_{0}^t \| D_{r,\cdot} u(t,x)\|^2_{\tilde L^2(\R^d)}   dr
    \ge& \int_{t-\delta}^t \| D_{r,\cdot} u(t,x)\|^2_{\tilde L^2(\R^d)}  dr  \\
    \geq& \frac{1}{2} \int_{t-\delta}^t \| p_{t-r}^{(x)}(\cdot)u(r,\cdot)\| ^2_{\tilde L^2(\R^d)} dr - I(\delta),
\end{aligned}
\end{equation}
where { $\delta\in(0,1)$ and}
\begin{equation}\label{e:I-delta}
\begin{aligned}
    I(\delta) & = \int_{t-\delta}^t \left\| \int_0^t \int_{\bR^d} p_{t-s}^{(x)}(y) D_{r,\cdot} u(s,y) W(d s, dy) \right\|^2 _{\tilde L^2(\R^d)} dr \\
    & = \int_{t-\delta}^t \left\| \int_{t-\delta}^t \int_{\bR^d} p_{t-s}^{(x)}(y) D_{r,\cdot} u(s,y) W(d s, dy) \right\|^2 _{\tilde L^2(\R^d)} dr.
\end{aligned}
\end{equation}
Here, the second equality is due to the fact that $D_{r,\cdot}u(s,y)=0$ if $r>s$.

By the inequality $\|a-b\|^2 \ge \frac{1}{2} \|a\|^2 - \|b\|^2$, we have that on $\Omega_m:=\{|u(t,x)|>\frac{1}{m}\}$,
\begin{align*}
   & \int_{t-\delta}^t \| p_{t-r}^{(x)}(\cdot) u(r,\cdot)\|_{\tilde L^2(\R^d)}^2  dr=\int_{t-\delta}^t \|p_{t-r}^{(x)}(\cdot)u(t,x)- p_{t-r}^{(x)}(\cdot) (u(t,x)-u(r,\cdot))\|_{\tilde L^2(\R^d)}^2  dr\\ &\geq\int_{t-\delta}^t\Big( \frac 12 \|p_{t-r}^{(x)}(\cdot)u(t,x)\|_{\tilde L^2(\R^d)}^2-\|p_{t-r}^{(x)}(\cdot) (u(t,x)-u(r,\cdot))\|_{\tilde L^2(\R^d)}^2\Big)  dr\\& =\frac 12\int_{t-\delta}^t \int_{\R^{2d}}p_{t-r}^{(x)}(z_1)p_{t-r}^{(x)}(z_2)\gamma(z_1-z_2)(u(t,x))^2dz_1dz_2dr\\&\hspace{1em}-\int_{t-\delta}^t \int_{\R^{2d}}p_{t-r}^{(x)}(z_1)p_{t-r}^{(x)}(z_2)\gamma(z_1-z_2)(u(t,x)-u(r,z_1))(u(t,x)-u(r,z_2))dz_1dz_2dr\\
    & \geq \frac{1}{2 m^2} \psi(\delta) -\frac 12\int_{t-\delta}^t \int_{\R^{2d}}p_{t-r}^{(x)}(z_1)p_{t-r}^{(x)}(z_2)\gamma(z_1-z_2)\big[(u(t,x)-u(r,z_1))^2\\&\hspace{12em}+(u(t,x)-u(r,z_2))^2\big]dz_1dz_2dr\\&\geq \frac{1}{2 m^2} \psi(\delta)-J(\delta),
\end{align*}
where
\begin{align*}
    J(\delta) &:= \int_{t-\delta}^t \int_{\R^{2d}}p_{t-r}^{(x)}(z_1)p_{t-r}^{(x)}(z_2)\gamma(z_1-z_2)(u(t,x)-u(r,z_1))^2dz_1dz_2dr.
\end{align*}
So, by \eqref{eq:33}, we have that
\begin{equation} \label{LB-D}
    \int_{0}^t \| D_{r,\cdot} u(t,x)\|^2_{\tilde L^2(\R^d)}   dr
    \geq \frac{1}{4 m^2} \psi(\delta) - \frac12 J(\delta) - I(\delta).
\end{equation}

{\bf Step 5.}
Let us now estimate the first moment of $J(\delta)$. To start with, we have
\begin{align*}
    \E[|J(\delta)|]& =\int_{t-\delta}^t \int_{\R^{2d}}p_{t-r}^{(x)}(z_1)p_{t-r}^{(x)}(z_2)\gamma(z_1-z_2)\E\big[|u(t,x)-u(r,z_1)|^2\big]dz_1dz_2dr\\&=\int_{t-\delta}^t \int_{\R^{d}}\int_{B(x;\delta^{\frac {1}{3}})}p_{t-r}^{(x)}(z_1)p_{t-r}^{(x)}(z_2)\gamma(z_1-z_2)\E\big[|u(t,x)-u(r,z_1)|^2\big]dz_1dz_2dr\\&\hspace{1em}+\int_{t-\delta}^t \int_{\R^{d}}\int_{\R^d\setminus B(x;\delta^{\frac {1}{3}})}p_{t-r}^{(x)}(z_1)p_{t-r}^{(x)}(z_2)\gamma(z_1-z_2)\E\big[|u(t,x)-u(r,z_1)|^2\big]dz_1dz_2dr\\&=:J_1(\delta)+J_2(\delta),
\end{align*}
where $B(x;\delta^{\frac 13})=\{y\in\R^d:|x-y|<\delta^{\frac 13}\}$.
For the first term, we have that
\begin{align}\label{e:J1} \nonumber
    J_1(\delta)& \leq \int_{t-\delta}^t \int_{\R^{d}}\int_{B(x;\delta^{\frac {1}{3}})}p_{t-r}^{(x)}(z_1)p_{t-r}^{(x)}(z_2)\gamma(z_1-z_2)\sup_{t-\delta<s<t} \sup_{|x-y|<\delta^{\frac {1}{3}} }\left(\E\big[|u(t,x)-u(s,y)|^2\big]\right)dz_1dz_2dr \\
    & \leq  g_{t,x}(\delta) \psi(\delta),
\end{align}
with
\[
g_{t,x}(\delta):= \sup_{|t-s|<\delta }\sup_{|x-y|<\delta^{\frac {1}{3}} } \left(\E
\big[|u(t,x)-u(s,y)|^2\big]\right).
\]
According to Theorem \ref{continuity of u}, $u$ is
continuous in $L^2(\Omega)$ in both time and space. Hence, 
\begin{equation}\label{e:gtx0}
\lim_{\delta \to 0}g_{t,x}(\delta) =0.
\end{equation}
{ For the second term,  we have
\begin{align*}
J_2(\delta)&\leq  4\sup_{(s,y)\in[0,t]\times \R^d}\E\big[|u(s,y)|^2\big]\int_{0}^\delta \int_{\R^d\setminus B(x;\delta^{\frac 13})}P_{r}(z_1-x)\gamma(z_1-z_2)dz_1\int_{\R^d}p_{r}^{(x)}(z_2)dz_2 dr\nonumber\\
&\leq C \delta^{d/2} \int_{0}^\delta \int_{\R^d\setminus B(x;\delta^{\frac 13})}P_{2r}(\delta^{\frac 13})P_{2r}(z_1-x)\gamma(z_1-z_2)dz_1\int_{\R^d}p_{r}^{(x)}(z_2)dz_2 dr\nonumber\\
&\leq C \delta^{d/2} \int_{0}^\delta P_{2r}(\delta^{\frac 13})dr\int_{\bR^{2d}} P_{2r}(z_1-x) P_r(z_2-x) \gamma(z_1-z_2)dz_1dz_2 \\
&\leq C \delta^{d/2} \int_{0}^\delta P_{2r}(\delta^{\frac 13})dr  \nonumber\\
&= C \delta^{\frac d2+1} { P_{2\varepsilon}(\delta^\frac13)},\nonumber
\end{align*}}
where  in the second step we use the identity $P_r(x)=2^d r^{d/2} \big(P_{2r}(x)\big)^2$ recalling that by assumption $P_t(x) = t^{-d/2} \exp\left(-c|x|^2/t\right)$ and in the last step $\varepsilon\in(0,\delta)$. Noting that $\lim_{
\delta\to0}P_{2\varepsilon}(\delta^{\frac13})=0$, 
we have $J_2(\delta)=o(\delta)$. Combining this with  \eqref{e:J1}, we get 
\begin{align}\label{e:J}
\E[J(\delta)]\leq g_{t,x}(\delta)\psi(\delta)+o(\delta).
\end{align}

 {\bf Step 6.} Now we proceed to bound $\E[I(\delta)]$ for the term $I(\delta)$ on the right-hand side of \eqref{LB-D}. Recall that $I(\delta)$ is given in \eqref{e:I-delta}. By Fubini's theorem,
\begin{align*}
&\E[I(\delta)]
    =\int_{t-\delta}^t\E \left\| \int_{t-\delta}^t \int_{\bR^d} p_{t-s}^{(x)}(y) D_{r,\cdot} u(s,y) W(d s, dy) \right\|^2 _{\tilde L^2(\R^d)} dr.
\end{align*}
Similar to $K^{(r)}$ given in \eqref{e:Kr}, we denote
\begin{align*}
    K_{\delta}^{(r)}((s,y),\cdot)={\mathbf 1}_{[t-\delta,t]}(s) p_{t-s}^{(x)}(y)  D_{r,\cdot}u(s,y).
\end{align*}
 Clearly $K_{\delta}^{(r)} \in \text{Dom}(\tilde{\boldsymbol{\delta}})$, and  we have  by \eqref{e:formula1},
$$\E\left[\|\tilde{\boldsymbol{\delta}}(K_{\delta}^{(r)})\|^2_{\tilde L^2(\bR^d)}\right] \leq \E\left[\|K_{\delta}^{(r)}\|_{\cH \otimes  \tilde L^2(\bR^d)}^2\right]+\E\left[\|DK_{\delta}^{(r)}\|^2_{\cH \otimes \cH\otimes  \tilde L^2(\bR^d)}\right].$$
Hence
\begin{align} \label{ineq-E[Id(elta)]}
    \E[I(\delta)] \leq I_1(\delta)+I_2(\delta),
\end{align}
where
$$I_1(\delta)=\int_{t-\delta}^t \E\left[\|K_{\delta}^{(r)}\|_{\cH \otimes \tilde L^2(\bR^d)}^2\right]  dr \quad \mbox{and} \quad I_2(\delta)=\int_{t-\delta}^t \E\left[\|DK_{\delta}^{(r)}\|_{\cH \otimes \cH \otimes \tilde L^2(\bR^d)}^2\right]  dr.$$ 
 By \eqref{estimate Du} in Step 2 and using \eqref{eq-phi}, we get
\begin{align*}
    & \bE\left[\|K_{\delta}^{(r)}\|_{|\cH|\otimes\tilde L^2(\bR^d)}^2\right] \\
    & \leq \int_{t-\delta} ^t\int_{t-\delta} ^t\int_{\bR^{2d}} p_{t-s}^{(x)}(y) p_{t-s'}^{(x)}(y') \Big( \bE\left[ \| D_{r,\cdot}u(s,y)\|_{\tilde L^2(\bR^d)}^2\right] \Big)^{1/2}  \\
    & \qquad \qquad \times \Big(\E \left[\| D_{r,\cdot}u(s',y')\|_{\tilde L^2(\bR^d)}^2 \right]\Big)^{1/2} |s-s'|^{-\beta_0} \gamma(y-y') dydy'dsds' \\
    & \leq C_t \phi(\delta).
\end{align*}
Similar,  using the estimate \eqref{estimate DDu} of $\bE \left[\| D^2_ru(t,x)\|_{\tilde L^2(\bR^d)}\right]$ proven in Step 2, we obtain
\begin{align*}
    & \bE \left[\|DK_{\delta}^{(r)}\|_{|\cH|\otimes \cH \otimes \tilde L^2(\bR^d)}^2\right] \\
\nonumber
    & \leq \int_{t-\delta} ^t \int_{t-\delta} ^t \int_{\bR^{2d}} p_{t-s}^{(x)}(y) p_{t-s'}^{(x)}(y') \Big(\E \left[\| D^2_{\cdot,(r,\cdot)} u(s,y)\|_{\cH \otimes \tilde L^2(\bR^d)}^2\right] \Big)^{1/2} \nonumber \\
    & \qquad \qquad \times  \Big(\E \left[\| D_{\cdot,(r,\cdot)}^2 u(s',y')\|_{\cH \otimes \tilde L^2(\bR^d)}^2 \right]\Big)^{1/2} |s-s'|^{-\beta_0} \gamma(y-y')dydy'dsds' \\
    & \leq C_t' \phi(\delta).
\end{align*}
Thus, by \eqref{ineq-E[Id(elta)]} we have,
\begin{align} \label{estimate-I}
    \bE[I(\delta)]
    \le (C_t + C_t') \phi(\delta).
\end{align}
{\bf Step 7.}
Combining  \eqref{LB-D},  \eqref{e:J} and \eqref{estimate-I} and using  Markov's inequality, we get for any $n\geq 1$, 
\begin{align*}
& \bP\left(\Big\{\int_{0}^t  \|D_{r,\cdot} u(t,x)\|^2_{\tilde L^2(\R^d)}  dr <\frac1n\Big\}\cap \Omega_m\right) \leq
 \bP\left(I(\delta) + \frac{1}{2} J(\delta) > \frac{1}{4 m^2} \psi(\delta) - \frac1n \right)\\
 & \quad \leq \left( \frac{1}{4 m^2} \psi(\delta) - \frac1n\right)^{-1} \Big(\E[I(\delta)] + \frac{1}{2} \E[|J(\delta)|]\Big)\\
 & \quad \leq \left( \frac{1}{4 m^2} \psi(\delta) - \frac1n\right)^{-1} \left( (C_t + C_t') \phi(\delta)  + g_{t,x}(\delta) \psi(\delta) +o(\delta)\right).
\end{align*}
Taking $n\rightarrow \infty$, one gets by  using \eqref{phipsi} and \eqref{ineq-psi},
\begin{align*}
    & \bP\left( \Big\{\int_{0}^t \|D_{r,\cdot} u(t,x)\|^2_{\tilde L^2(\R^d)}  dr =0 \Big\}\cap \Omega_m\right) \\
&\leq  C m^2 \left(A_\delta + g_{t,x}(\delta)+o(1)\right).
\end{align*}
Next, we take $\delta\to 0$. Noting that $\lim_{\delta\to 0}(g_{t,x}(\delta)+A_{\delta})=0$ by \eqref{e:gtx0} and the fact that $A_\delta=2\int_0^\delta |s|^{-\beta_0} ds$,  we can obtain
$$\bP\left(\Big\{\int_{0}^t \|D_{r,\cdot} u(t,x)\|^2_{\tilde L^2(\R^d)}  dr  =0\Big\}\cap \Omega_m\right)=0.$$
This concludes the proof of \eqref{integral-positive} and hence the proof of Theorem \ref{Thm-density-Skorohod-1}.
\end{proof}

\begin{remark}
As pointed out in \cite{BQS2019EJP}, the main difficulty in proving the regularity of the law for the Skorohod solution comes from the time-space correlation of the underlying noise. Inspired by the method used in \cite{BQS2019EJP}, 
 we transform our problem into the case that the noise is white in time by considering the space $\tilde \cH$. We remark that the localization argument \eqref{eq:33} does not necessarily hold  if we work on $\|Du(t,x)\|^2_{\mathcal H}$ directly, noting that $\int_0^t \int_0^t |r-s|^{-\beta_0} f(r)f(s) drds$ is not increasing in $t$ in general (see, e.g., \cite[Example 3.5]{hp09} for a counterexample). 
\end{remark}

The following Lemma is used in the proof of Theorem \ref{Thm-density-Skorohod-1}.

\begin{lemma}\label{e:tilde}
If $\|f\|_{\cH} = 0$, we have $\|f\|_{\tilde \cH} = 0$ and $\|f(t,\cdot)\|_{\tilde L^2(\bR^d)} = 0$ for almost all $t \in \bR_+$. Moreover, if  $\mathrm{supp}(\mu)=\R^d$, we have $\|f(t,\cdot)\|_{L^2(\bR^d)} =0$ for almost all $t \in \bR_+$.
\end{lemma}

\begin{proof}
Assume that $\|f\|_{\cH}=0$. We use the convention that $f(t,x) = 0$ for $t<0$. Applying Parseval–Plancherel identity 
 to \eqref{e:inner-prod}, we get
\begin{align*}
    \int_{\bR} \int_{\bR^d} \left| \cF f(\cdot,\cdot)(u,\xi) \right|^2 |u|^{\beta_0-1} \mu(d\xi) du = 0.
\end{align*}
Thus, we have $    \left| \cF f(\cdot,\cdot)(u,\xi) \right|^2 = 0$
almost everywhere on $\bR \times \mathrm{supp}(\mu)$ and hence
\begin{align} \label{eq-123}
    \|f\|_{\tilde {\cH}}^2
    = \int_{\bR} \int_{\R^d} \left| \cF f(s,\cdot)(\xi) \right|^2 \mu(d\xi)ds
    = \int_{\bR} \int_{\R^d} \left| \cF f(\cdot,\cdot)(u,\xi)\right|^2 \mu(d\xi)du
    = 0.
\end{align}
The identity \eqref{eq-123} yields
\begin{align} \label{eq-1234}
   \|f(s,\cdot)\|_{\tilde L^2(\bR^d)}^2 = \int_{\R^d} \left| \cF f(s,\cdot)(\xi) \right|^2 \mu(d\xi)
    = 0
\end{align}
for almost all $s \in \bR_+$. 
For those $s \in \bR_+$ such that \eqref{eq-1234} holds, we have $    \left| \cF f(s,\cdot)(\xi) \right|^2 = 0$
for almost all $\xi \in \mathrm{supp}(\mu)$. Hence, if $\mathrm{supp}(\mu) = \bR^d$, we have $  \|f(s,\cdot)\|_{L^2(\bR^d)}^2
    = \int_{\R^d} \left| \cF f(s,\cdot)(\xi) \right|^2 d\xi
    = 0.$
This means that $\|f(s,\cdot)\|_{L^2(\bR^d)}^2 = 0$ for almost all $s \in \bR_+$ when $\mathrm{supp}(\mu) = \bR^d$.
\end{proof}

\begin{remark}
The argument in the beginning of the proof of Theorem \ref{Thm-density-Skorohod-1} was initially used in~\cite{BQS2019EJP} (see the proof of Theorem 1.2 therein), where space $L^2(\bR^d)$ instead of $\tilde L^2(\bR^d)$ was used under the implicit additional condition $\mathrm{supp}(\mu)=\R^d$.
\end{remark}




{\bf Acknowledgement} We would like to thank Xia Chen and Yaozhong Hu for inspiring discussions.  J. Song is partially supported by National Natural Science Foundation of China (No. 12521001). W. Yuan is partially supported by National Natural Science Foundation of China (No. 12501183), Guangdong Basic and Applied Basic Research Foundation (No. 2026A1515030040), and a Grant of the Department of Science and Technology of Guangdong Province (No. 2024QN11X161). W. Yuan gratefully acknowledges the financial support of ERC Consolidator Grant 815703 ``STAMFORD: Statistical Methods for High Dimensional Diffusions'' when he was a postdoctoral researcher at University of Luxembourg.



\appendix

\section{Feynman--Kac formula}\label{sec:FK}
Consider the following partial differential equation 
\begin{equation}\label{e:pde}
   \begin{cases}
        \frac{\partial }{\partial t} u(t,x)=\mathcal L u(t,x) + f(t,x)u(t,x), &t\ge 0, x\in\R^d,\\
        u(0,x)=u_0(x), &x\in \R^d,
        \end{cases}
\end{equation}
where $\mathcal L$ is the infinitesimal generator of a Feller process $X$, $u_0(x)$ is a bounded measurable function, and $f(t,x)$ is a measurable function.  The Feynman--Kac formula for \eqref{e:pde} can be found in \cite[Theorem 3.47]{Liggett}. Nevertheless, we provide our version of Feynman--Kac formula which suits our purpose.  

Assume 
\begin{equation}\label{e:con-f}
\E_{X}\left[\exp\left(\left| \int_0^t f(t-s, X^x_s)ds\right|\right)\right]<\infty, \text{ for  all } x\in\R^d.    \end{equation}
Then, the following Feynman--Kac representation
\begin{equation}\label{e:FK'}
u(t,x)=\E_{X}\left[u_0(X^x_t)\exp\left( \int_0^t f(t-s, X^x_s)ds\right)\right]
\end{equation}
is a Duhamel's solution to \eqref{e:pde}, i.e.,
\begin{equation}\label{e:duhamel}
    u(t,x) = \int_\R p_t^{(x)}(y)u_0(y) dy + \int_0^t \int_{\R^d} p_{t-s}^{(x)}(y) u(s,y) f(s,y) dyds. 
\end{equation}
\begin{proof}
One can verify directly that $u(t,x)$ given by \eqref{e:FK'} satisfies \eqref{e:duhamel}.  For simplicity, we assume $u_0(x)\equiv1$. Plugging the expression  \eqref{e:FK'} to \eqref{e:duhamel}, we have that  the right-hand side of \eqref{e:duhamel} is
\begin{equation}\label{e:rhs}
\begin{aligned}
&1+ \int_0^t \E_{X} [u(s, X^x_{t-s}) f(s, X^x_{t-s})] ds\\
&=1+ \int_0^t \E_{X} [u(t-s, X^x_{s}) f(t-s, X^x_{s})] ds\\
&= 1+ \int_0^t \E_{X} \left[\tilde \E_{\tilde X} \left[\exp\left( \int_0^{t-s} f(t-s-r, \tilde X_r^{X_s^x}) dr\right) \right ] f(t-s, X^x_{s})\right] ds,
\end{aligned}
\end{equation}
where $\tilde X$ is an independent copy of $X$ and $\tilde \E$ means the expectation with respect to $\tilde X$. Applying Taylor's expansion to the function $e^x$ and then taking expectation, one can show that the resulting series of the right-hand side of \eqref{e:rhs} is absolute convergent under the condition \eqref{e:con-f} and coincides with the series expansion  of $u(t,x)=\E_{X}\left[\exp\left( \int_0^t f(t-s, X_s^x)ds\right)\right]$  on the left-hand side of \eqref{e:duhamel}.  
\end{proof}

\bibliographystyle{plain}
\bibliography{Reference-SHE}

\end{document}